\documentclass[opre]{arXiv}

\RequirePackage{bm}
\RequirePackage{endnotes}

\usepackage{fix-cm}
\DeclareFontShape{OML}{cmm}{b}{it}{<->cmmib10}{}
\DeclareFontShape{OMS}{cmsy}{b}{n}{<->cmsy10}{}

\OneAndAHalfSpacedXI

\usepackage{algorithm}
\usepackage[noend]{algpseudocode}

\algrenewcommand\algorithmicrequire{\textbf{Input:}}
\algrenewcommand\algorithmicensure{\textbf{Output:}}

\usepackage{tikz}
\usepackage{amsmath,amssymb,amsfonts}
\usepackage{graphicx,color}
\usepackage{url}
\usepackage{subfiles}
\usepackage{booktabs}
\usepackage{mathtools}
\usepackage{microtype}
\usepackage[english]{babel}
\usepackage{caption}
\usepackage{verbatim}
\usepackage[commandnameprefix=always]{changes}
\definechangesauthor[name={Qingyuan Xu}, color=purple]{QX}
\usepackage{subcaption}
\usepackage{hyperref}       
\hypersetup{
  colorlinks=true,
  linkcolor=red,
  urlcolor=magenta,
  citecolor=blue
}

\usepackage{natbib}
 \bibpunct[, ]{(}{)}{,}{a}{}{,}%
 \def\bibfont{\small}%
\usetikzlibrary{arrows.meta,calc}

\newcommand{\X}{\mathcal X}
\newcommand{\XiSet}{\Xi}
\newcommand{\I}{\mathcal I}
\newcommand{\R}{\mathbb R}

\DeclareMathOperator{\state}{\textbf{state}}
\DeclareMathOperator{\refpoint}{\textbf{ref}}

\EquationsNumberedThrough    

\TheoremsNumberedThrough     
\ECRepeatTheorems  %

\MANUSCRIPTNO{ }

\begin{document}




\TITLE{Soft Separation for Adaptive Robust Optimization}

\RUNAUTHOR{Q. Xu and R. Jiang}
\RUNTITLE{Soft Separation for Adaptive Robust Optimization}

\ARTICLEAUTHORS{%
\AUTHOR{Qingyuan Xu, Ruiwei Jiang}
\AFF{Department of Industrial and Operations Engineering, University of Michigan, \EMAIL{\{qyxu, ruiwei\}@umich.edu}}

} 

\ABSTRACT{%
We propose an algorithmic framework for solving adaptive robust optimization with provable guarantees on both tractability and solution accuracy. The framework introduces soft separation, a probabilistic mechanism for identifying worst-case uncertainty realizations via a time-inhomogeneous Markov chain. Rather than solving an exact separation problem in each iteration, which is intractable in general, the chain carries adversarial information across iterations, co-evolves with the optimization iterates, and recovers exact separation at terminal iterates with high probability. For continuous first-stage (here-and-now) decisions, we design a first-order method that uses soft separation to produce adaptive gradient estimates. Notably, we prove \emph{polynomial-time convergence} in expectation to the global optimum. For mixed-integer here-and-now decisions, we embed soft separation within a branch-and-cut framework to generate valid cuts for the robust objective and obtain a high-probability certificate of global optimality. Numerical experiments demonstrate that the proposed methods scale favorably with problem dimension and scenario size relative to state-of-the-art approaches. The instances and code are available online at \url{https://github.com/xuqy2002/SoftSeparationARO}.


}%




\maketitle


\section{Introduction}\label{sec:Intro}

Adaptive robust optimization (ARO) is an important framework for optimization under uncertainty. 
To hedge against adversarial uncertainty, it adapts to the uncertain realization through adjustable second-stage (also known as ``wait-and-see'') recourse decisions, which allow ARO to provide reliable decisions while remaining less conservative than static robust optimization alternatives. ARO finds wide-ranging applications in, e.g., healthcare~\citep{kong2013scheduling,qi2017mitigating,ryu2025nurse}, power systems~\citep{jiang2011robust,an2014exploring,moreira2024distribution}, and supply chains~\citep{lu2015reliable,cheng2018two,qi2024sequential}. 
Formally, ARO is formulated as
\begin{equation}
\label{main}
\begin{aligned} 
F^\star \ := \ \min_{x\in\mathcal X} 
\ F(x) := \left\{ f(x)+\max_{\xi\in\Xi}Q(x,\xi)
\right\},
\end{aligned} \tag{ARO}
\end{equation}
where $f$ denotes a cost function, $x\in\mathcal X\subseteq\mathbb R^{d_x}$ 
denotes the first-stage (or ``here-and-now'') decision, and $\xi\in\Xi\subseteq\mathbb R^{d_\xi}$ denotes the uncertain parameter with \(\Xi\) representing an uncertainty set. Both \(x\) and \(\xi\) may be  continuous or discrete. For fixed $x$ and $\xi$, the recourse function is defined through a {\color{black} convex conic program
\begin{equation}
\label{recourse}
Q(x,\xi)
:=
\min_y \; b^\top y
\quad
\text{s.t.}
\quad
G y + E x + U\xi - h \in \mathcal{K},
\end{equation}}
\noindent where $y \in \mathbb{R}^{d_y}$ denotes the wait-and-see recourse decisions and \(\mathcal{K} \subseteq \mathbb{R}^m\) denotes a convex cone. For example, when \(\mathcal{K} := \mathbb{R}^m_+\), the recourse formulation~\eqref{recourse} reduces to a linear program and~\eqref{main} is a two-stage robust \emph{linear} program. Despite wide applications,~\eqref{main} is computationally challenging~\citep{minoux20112}, largely because of the inner max-min problem. For fixed \(x\), this problem seeks to identify a worst-case scenario \(\xi^\star \in \argmax_{\xi \in \Xi}Q(x, \xi)\), or equivalently, exactly separate the epigraph of the worst-case recourse function \(\max_{\xi \in \Xi}Q(x, \xi)\). Exact solution algorithms for~\eqref{main} solve the problem \emph{repeatedly}. For example, the column-and-constraint generation (C\&CG) algorithm~\citep{ZengZhao2013} repeatedly refines a relaxation to~\eqref{main} by appending new $\xi^\star$, together with the corresponding recourse variables and constraints. Likewise, the Benders decomposition algorithm~\citep[see, e.g.,][]{jiang2011robust} repeatedly separates the worst-case recourse function by a supporting hyperplane of \(Q(x, \xi^\star)\), also known as an ``optimality cut.''

Unfortunately, the exact separation problem is usually intractable. For example, when \(\mathcal{K}=\mathbb{R}^m_+\) and the uncertainty set \(\Xi\) admits a parametric form such as a polyhedron or an ellipsoid, dualizing the recourse linear program~\eqref{recourse} recasts the problem as a bilinear program, which is generally NP-complete~\citep{bennett1993bilinear}. In addition, when \(\Xi\) is a finite set of scenarios with no clear structure, the exact separation boils down to enumerating \(\Xi\) and solving the corresponding recourse problem~\eqref{recourse} for each scenario, which quickly becomes prohibitive as the scenario size increases (see Section~\ref{sec:experiments} for numerical demonstration).


We argue, however, that an exact separation is not necessary in every step of an iterative algorithm like C\&CG or Benders. 
Early in the iterative algorithm, the incumbent solution can be far from optimum and an adversarial (but not necessarily the worst) scenario can already reveal useful directions for improvement. Moreover, adversarial scenarios often persist across nearby incumbent solutions. 
Hence, the adversarial scenario found for an incumbent solution need not be discarded immediately after the incumbent is updated. Instead, it is computationally much cheaper to carry forward and refine it for future incumbents.

These observations motivate \textbf{soft separation}. Instead of solving a fresh exact separation problem in each iteration, we track adversarial scenarios through a time-\emph{in}homogeneous Markov chain over the uncertainty set.  
At iteration $n$, the transition kernel uses the current incumbent $x_n$, while the chain state records an adversarial scenario. 
As $x_n$ changes, the chain updates this state rather than restarting the adversarial search from scratch. Each separation therefore reduces to a simple scenario-wise update, while the Markov chain carries adversarial information across iterations. To back up this idea rigorously, we show that, when the \(x_n\) sequence changes in a controlled way and the soft separation is appropriately designed, the Markov chain identifies adversarial scenarios for all later-stage iterates of \(x_n\) with high probability. Then, we apply this soft-separation oracle in two settings: for \(x\) being continuous, the oracle produces Markovian gradient estimates for a robust surrogate; and for \(x\) being mixed-integer, the oracle generates valid cuts for the worst-case recourse function in a branch-and-cut framework. 
Notably, for continuous \(x\), we prove polynomial-time convergence of \(x_n\), in expectation, to the global optimum of~\eqref{main}. To our best knowledge, this provides the first \emph{tractable} and \emph{tight} approximation algorithm for general~\eqref{main}.

\subsection{Literature Review}
Robust optimization (RO) models can be categorized by how the uncertainty set \(\Xi\) is characterized. The majority of the literature focuses on parametric uncertainty sets such as polyhedra~\citep[see, e.g.,][]{bertsimas2004price} or ellipsoids~\citep[see, e.g.,][]{ben1999robust}. These uncertainty sets have good interpretation and provide robustness guarantees that, with high probability, any feasible solution to a RO model remains feasible when exposed to out-of-sample random realizations of \(\xi\). Such guarantees usually rely on distributional assumptions of \(\xi\) such as independence (among all \(\xi\)-entries), light-tailedness, symmetry, or boundedness~\citep[see, e.g., Table 4 of][for a summary]{bertsimas2021probabilistic}. In addition, \(\Xi\) can be characterized by independent scenarios of \(\xi\) drawn from its (possibly latent) probability distribution. This gives rise to a finite \(\Xi\) and the ensuing~\eqref{main} admits both a priori~\citep{campi2008exact} and a posteriori~\citep{campi2018wait} robustness guarantees. Different from the parametric uncertainty sets, the robustness guarantees herein are free of distributional assumptions, except for the independence of the scenarios in \(\Xi\). This paper seeks to develop algorithms for both parametric and scenario uncertainty sets.

Existing algorithmic work for solving~\eqref{main} falls into two streams, with the first aiming to solve~\eqref{main} to global optimum and solve the exact separation more efficiently, and the second resorting to tractable approximations of~\eqref{main}. 
When \(\Xi\) admits a parametric form such as a polyhedron or an ellipsoid, the C\&CG algorithm by~\cite{ZengZhao2013} solves a relaxation of~\eqref{main} by considering a subset of \(\xi\)-scenarios. To achieve global optimum, C\&CG iteratively solves the exact separation problem and appends the ensuing \(\xi^\star\) to this subset. Despite the challenges of the exact separation, C\&CG has demonstrated empirical success in real-world engineering systems such as power grids~\citep[see, e.g.,][]{yuan2016robust}, transportation~\citep{huang2023column}, and healthcare~\citep{ji2022two}. Recently,~\citet{tsang2023inexact} proposed a variant of C\&CG that solves the~\eqref{main} relaxations inexactly while still maintaining a finite convergence to the global optimum. In addition, variants of C\&CG have been applied to solve~\eqref{main} with mixed-integer recourse~\citep{zhao2012exact} and decision-dependent uncertainty~\citep{zeng2022two}. As compared to these exact algorithms, this paper studies solution approaches with tractability guarantee, e.g., converging to \(\epsilon\)-optimality of~\eqref{main} within a number of iterations that scales polynomially in \(1/\epsilon\).

On the other hand, to address the computational challenge of~\eqref{main}, decision rules have been proposed to restrict the recourse variables \(y\) in~\eqref{recourse} as functions of \(\xi\). For example,~\cite{bertsimas2013approximability} study a static policy, in which \(y\) is restricted to be a constant function of \(\xi\). This casts~\eqref{main} as a one-stage robust optimization problem and regains computational tractability~\citep[see, e.g.,][]{ben1999robust,bertsimas2004price}. Later,~\cite{han2023finite} extend the static policy to a piecewise static one and retain good interpretability. Moreover,~\cite{ben2004adjustable} propose an affine decision rule (ADR), which restricts \(y\) to be an affine function of \(\xi\), and subsequent studies generalize ADR through lifting~\citep[see, e.g.,][]{bertsimas2019adaptive}, Fourier-Motzkin elimination~\citep{zhen2018adjustable}, and to piecewise affine variants~\citep[see, e.g.,][]{chen2008linear,georghiou2020primal}. As in static policies, ADR and its generalizations produce conservative but tractable approximations for~\eqref{main}. These policies generally do not solve~\eqref{main} to global optimum, with a few exceptions, e.g., when \(\Xi\) is a standard simplex~\citep{bertsimas2012power} or a box~\citep{zhen2018adjustable} or when~\eqref{main} fulfills additional structural and technical assumptions~\citep{georghiou2026optimality}. Nevertheless, these policies admit approximation guarantees when, e.g., all entries of \(\xi\) are nonnegative~\citep{bertsimas2011geometric,bertsimas2015performance} or all recourse matrix entries are nonnegative~\citep{ElHousniGoyal2021,ElHousniFoussoulGoyal2024}. As compared to these decision rules, our algorithms seek to find near-optimal solutions to a general~\eqref{main} model.

We summarize our main contributions as follows.
\begin{enumerate}
    \item For~\eqref{main} with continuous decision variables and a scenario uncertainty set \(\Xi\), we adopt 
    a first-order algorithm that approximates the gradients of the robust objective using a time-inhomogeneous Markov chain over \(\Xi\), thereby avoiding exhaustive exact separation at each iteration. We prove polynomial-time convergence, in expectation, to the global optimum of~\eqref{main}.
    \item We formalize this mechanism as a soft separation framework and establish probabilistic guarantees for identifying adversarial scenarios. We then extend the framework beyond finite scenarios to general uncertainty sets.
    \item For~\eqref{main} with mixed-integer decisions, we embed soft separation within a branch-and-cut algorithm to produce deterministically valid cuts for the robust objective. Leveraging the probabilistic guarantees of soft separation, we obtain a high-probability certificate of global optimality.
    \item We demonstrate the proposed algorithms in extensive numerical experiments spanning continuous decisions and mixed-integer decisions, as well as scenario, ellipsoidal, and budget uncertainty sets. Across these settings, the proposed methods provide near-optimal solutions to~\eqref{main} and exhibit favorable scalability.
\end{enumerate}

The remainder of the paper is organized as follows. Section~\ref{sec:2} studies the first-order algorithm to solve~\eqref{main} with continuous decision variables and a scenario uncertainty set. The proposed gradient estimation method is formalized as the soft separation framework in Section~\ref{sec:3} and then applied to~\eqref{main} with mixed-integer decision variables in Section~\ref{sec:4}. We report the numerical experiment results in Section~\ref{sec:experiments}. We present all proofs in Appendix~\ref{app:proof}.

\paragraph{Notation:} we denote by \(\mathbf{0}, \mathbf{1}\) all-zero and all-one vectors, by \(\mathrm{Id}\) an identity matrix, by \(\mathbb{N}_+\) the nonnegative integers, by \([n]:=\{1, \ldots, n\}\) the set of running indices for \(n \in \mathbb{N}_+\), by \(\mathbb{B}^d\) the unit \(\ell_2\)-norm ball in \(\mathbb{R}^d\), by \(\mathbb{S}^{d-1}\) the sphere of \(\mathbb{B}^d\), 
and by \(O(\cdot)\) the big-O notation with \(f(n) = O(g(n))\) indicating that there exists an \(M > 0\) such that \(|f(n)|\leq Mg(n)\) for all sufficiently large \(n\).


\section{Continuous Decisions and Scenario Uncertainty}
\label{sec:2}

We start with a baseline~\eqref{main} model with continuous decision variables \(x\) and a scenario uncertainty set \(\Xi := \{\xi^{(i)}: i \in \mathcal{I}\}\), where \(\mathcal{I}\) is an index set. This setup is without loss of generality (WLOG) when \(\Xi\) is characterized by scenarios of \(\xi\)~\citep{campi2008exact,campi2018wait}, or when \(\Xi\) is polyhedral because \(Q(x, \xi)\) is convex in \(\xi\), implying that \(\Xi\) can be replaced by the set of its extreme points WLOG. 
For ease of exposition, we make the following assumptions throughout this section.
\begin{assumption}[Continuous \(x\) and cost function]
\label{assump:5}
The set $\X\subseteq\R^{d_x}$ is compact and convex with \(\|x - x'\|_1 \leq D\) for all \(x, x' \in \X\). The first-stage cost function $f:\X\to\R$ is convex and $L_f$-Lipschitz on an open neighborhood of $\X$.
\end{assumption}


\begin{assumption}[Relatively complete recourse]
\label{assump:1}
For every $(x,\xi)\in\X\times\XiSet$, the recourse function \(Q(x, \xi)\) is finite, that is, \(|Q(x, \xi)| < \infty\).
\end{assumption}


\begin{assumption}[Linear recourse problem]
\label{assump:linear-recourse}
\(\mathcal{K}=\mathbb{R}^m_+\).
\end{assumption}

Assumptions~\ref{assump:5}--\ref{assump:1} are standard in the literature of~\ref{main}. In particular, when~\eqref{main} is a two-stage robust \emph{linear} program with \(f(x) := c^{\top}x\), Assumption~\ref{assump:5} holds trivially with \(L_f = \|c\|\). We make Assumption~\ref{assump:linear-recourse} for ease of presenting the main ideas without technical clutter, and we extend the results of this section to the general conic case in Appendix~\ref{app:conic}. 
In general, $Q(x,\xi)$ lacks smoothness in \(x\) (and differentiability, even under Assumption~\ref{assump:linear-recourse}). We therefore work with its subgradients obtained from dual optimal solutions of the recourse problem. As a preliminary, the following lemma recalls relevant properties of \(Q\left(x,\xi\right)\).



\begin{lemma}
\label{lem:recourse_lipschitz_smoothing}
For each scenario \(\xi^{(i)}\), the function \(Q(x, \xi^{(i)})\) admits a representation
\begin{equation}
\label{QP}
Q\left(x,\xi^{(i)}\right) \ 
= \ 
\max_{\substack{\pi \geq 0\\ G^\top\pi=b}}
\Big\{
\pi^\top\!\left(h - E x - U\,\xi^{(i)}\right)
\Big\}.
\end{equation}
Let $\pi^\star(x,i)$ denote a maximizer of~\eqref{QP},
chosen to minimize $\|E^\top\pi\|_1$ in case of alternative dual
optimal solutions. Then, it holds that
\(
-E^\top\pi^\star(x,i)
\in \partial_x Q\left(x,\xi^{(i)}\right)
\). 
Furthermore, there exists a constant $L_Q>0$ such that, with
\(
L_Q' := L_Q\|E\|
\)
and
\(
L_\xi' := L_Q\|U\|
\), 
the following hold:

\noindent \emph{(1) Lipschitz continuity in \(x\):}
For every $\xi\in\Xi$, $Q(\cdot,\xi)$ is $L_Q'$-Lipschitz continuous on $\mathcal X$;

\noindent \emph{(2) Lipschitz continuity in \(\xi\):}
For every $x \in \X$, $Q(x,\cdot)$ is $L_\xi'$-Lipschitz continuous on $\Xi$;

\noindent \emph{(3) Bounded subgradient selection:} 
For every $x\in\mathcal X$ and every scenario $\xi^{(i)}$, the selected subgradient is measurable and satisfies 
\(
\left\|E^\top\pi^\star(x,i)\right\|
\leq L_Q'
\).
\end{lemma}

We note that \(\pi^\star(x,i)\) in Lemma~\ref{lem:recourse_lipschitz_smoothing} can be found by solving an auxiliary linear program, which minimizes \(\|E^{\top}\pi\|_1\) over the polyhedron \(\{\pi \geq 0: G^{\top}\pi = b, \pi^\top\!\left(h - E x - U\,\xi^{(i)}\right) = Q(x, \xi^{(i)})\}\), where \(Q(x, \xi^{(i)})\) has been pre-computed in the linear program~\eqref{QP}.

\subsection{Mellowmax Approximation and Algorithm}
Exact separation is challenging because evaluating \(\max_{\xi \in \Xi} Q(x, \xi)\) demands maximizing a convex function. To this end, we propose the following Mellowmax approximation to ``soften'' the maximization operator and connect it with an expectation operator. 
For \(w>0\), we define
\[
\mathcal{M}_{w}(x) \; := \; 
w \log\!\left( \frac{1}{|\mathcal{I}|}\sum_{i=1}^{|\mathcal{I}|} 
    \exp\!\left\{\tfrac{Q\left(x,\xi^{(i)}\right)}{w}\right\} \right).
\]
By construction, \(\mathcal{M}_{w}\) replaces the maximization operator in the exact separation with a log-sum-exponential form. 
The following lemma quantifies its approximation error.
\begin{lemma}[Uniform error bound of Mellowmax,~\cite{pmlr-v70-asadi17a}]
\label{lem:mellowmax-approx}
For every \(x\in\mathcal X\) and \(w>0\), it holds that
\[
\max_{\xi \in \Xi} \, Q(x, \xi)-w\log|\mathcal I|
\ \le \
\mathcal{M}_{w}(x)
\ \le \
\max_{\xi \in \Xi} \, Q(x, \xi).
\]
Moreover,
\(
\lim_{w\to0^+}\mathcal{M}_{w}(x)
=
\max_{\xi \in \Xi} Q_(x, \xi).
\)
\end{lemma}
Lemma~\ref{lem:mellowmax-approx} suggests that \(\mathcal{M}_{w}(x)\) approximates \(\max_{\xi \in \Xi}Q_(x, \xi)\) to any prescribed accuracy by choosing a sufficiently small \(w\). Notably, the approximation error depends \emph{logarithmically} on \(|\mathcal{I}|\), making \(\mathcal{M}_{w}(x)\) appealing for large uncertainty sets. We therefore consider the following optimization problem,
\begin{equation}
\min_{x \in \mathcal{X}} \;\; F_{w} (x) \ := \ f(x) 
+ w \log\!\left( \frac{1}{|\mathcal{I}|}\sum_{i=1}^{|\mathcal{I}|} 
    \exp\!\left\{ \tfrac{Q\left(x,\xi^{(i)}\right)}{w} \right\} \right) \tag{Approx} \label{main-approx}
\end{equation}
using \(F_{w} (x)\) as a surrogate of the original robust objective function in~\eqref{main}. 
Notice that \(F_{w}\) is convex but need not be differentiable on
$\mathcal X$. Hence, a natural thought is to apply the standard
projected subgradient method to~\eqref{main-approx} and generate
a near-optimal solution to~\eqref{main}. However, an examination of a subgradient reveals that
\[
\begin{aligned}
\partial F_w(x) \ \ni \ z_w(x)
&:=
g_f(x)
+\sum_{i\in\mathcal I}
\frac{\exp(Q(x,\xi^{(i)})/w)}
{\sum_{j\in\mathcal I}\exp(Q(x,\xi^{(j)})/w)}
\left(-E^\top\pi^\star(x,i)\right)
\\
&=
\sum_{i\in\mathcal I}
\left(g_f(x)-E^\top\pi^\star(x,i)\right)p_x(i),
\end{aligned}
\]
where $g_f(x)\in\partial f(x)$ is a bounded measurable
subgradient selection and
\[
p_x(i)
:=
\frac{\exp\left(Q(x,\xi^{(i)})/w\right)}
{\sum_{j\in\mathcal I}\exp\left(Q(x,\xi^{(j)})/w\right)}.
\]
A direct implementation would necessitate enumerating all
scenarios in $\Xi$ at every step of the subgradient method,
which is prohibitively expensive. If we resort to a stochastic
approximation, then in each iteration $n$ we sample an
$I_{n+1}\in\mathcal I$ conditionally on $\mathcal F_n$ according to
\(
\mathbb P\bigl(I_{n+1}=i\mid\mathcal F_n\bigr)
=
p_{x_n}(i)
\)
and generate a random proxy 
\(
Z_{n+1}
:=
g_f(x_n)-E^\top\pi^\star(x_n,I_{n+1})
\)
for the selected subgradient. Here, \(x_n\) is the decision incumbent at iteration \(n\) and \(\mathcal F_n\) denotes the natural filtration generated by the
algorithm. By construction,
\(
\mathbb E\bigl[Z_{n+1}\mid\mathcal F_n\bigr]
\in\partial F_w(x_n),
\)
so that
\(
x_{n+1}
=
\Pi_{\mathcal X}\bigl(x_n-\alpha_n Z_{n+1}\bigr)
\)
is a stochastic projected subgradient update with an
unbiased subgradient estimator.  

Unfortunately, this construction does not resolve the computational challenge for two reasons. First, a direct sampling from \(p_{x_n}\) requires evaluating all scenario weights and consequently computing all \(Q(x, \xi^{(i)})\). 
Second, the target distribution \(p_{x_n}\) changes with the iterate \(x_n\). 
Thus, even when a suitable sampling method is available, we may need to rerun it in every
iteration and repeating this procedure offsets much of the computational gain. 

As an alternative, we propose to not insist on an exact and unbiased gradient estimate in every iteration. Instead, we maintain an inexact estimate of \(p_{x_n}(i)\) using a time-inhomogeneous Markov chain. This chain favors scenarios \(i \in \mathcal{I}\) with larger \(Q(x_n, \xi^{(i)})\) values, which imitate \(p_{x_n}(i)\). In addition, it co-evolves with \(x_n\) throughout the algorithm and coincides with \(p_{x_n}(i)\) in the limit. 
Specifically, we adopt a sample-rejection procedure with a proposal distribution \(q(j|i)\) and an acceptance distribution \(A_x(i, j)\). Iteratively, when the Markov chain is at state \(i \in \mathcal{I}\), we first sample a proposal state \(j \in \mathcal{I}\) according to \(q(\cdot|i)\) and then accept or reject the new state \(j\) with probability \(A_x(i, j)\). This gives rise to a parameterized Markov kernel \(\{P_x: x \in \mathcal{X}\}\) with, for all \(i, j \in \mathcal{I}\),
\[
P_x(i,j) \ := \ 
\begin{cases}
q(j\mid i)\,A_x(i,j), & \text{if } j\neq i\\[3pt]
1-\sum_{k\neq i}q(k\mid i)\,A_x(i,k), & \text{if } j=i.
\end{cases}
\]
For fixed \(x\), the Markov chain induced by \(P_x\) is irreducible whenever $q(j\mid i)>0$ for all \(i, j \in \mathcal{I}\) and aperiodic if $q(i\mid i)>0$. In particular, for a finite \(\mathcal{I}\), irreducibility guarantees the existence of a unique
stationary distribution, and we shall henceforth assume that such uniqueness holds. We give two examples for such a Markov chain.
\begin{example}[Metropolis--Hastings (MH) Acceptance] \label{ex:MH-proposal}
For fixed proposal distribution \(q(\cdot|\cdot)\) and incumbent \(x\), the MH acceptance probability puts
\[
A_x^{\mathrm{MH}}(i,j)
\ := \ \min\!\left\{\,1,\,
\frac{p_x(j)\,q(i\mid j)}{p_x(i)\,q(j\mid i)}\right\} = \ \min\left\{1,\exp\!\Big(\tfrac{Q \left(x,\xi^{(j)}\right)-Q\left(x,\xi^{(i)}\right)}{w}\Big) \cdot \frac{q(i \mid j)}{q(j \mid i)}\right\},
\quad \forall i,j\in\mathcal{I}.
\]
It follows that 
$p_x(i)P_x(i,j)=p_x(j)P_x(j,i)$ for all $i,j$
and hence $p_x$ is the stationary distribution of the Markov chain. \hfill \(\Box\)
\end{example}

\begin{example}[Barker Acceptance] \label{ex:barker-proposal}
A smoother acceptance rule proposed by Barker puts
\[
A_x^{\mathrm{B}}(i,j)
\ := \ \frac{r_x(i,j)}{1+r_x(i,j)}, \quad \text{where}\ 
r_x(i,j) \ := \ \frac{p_x(j)\,q(i\mid j)}{p_x(i)\,q(j\mid i)} \ = \ 
\exp\!\Big(\tfrac{Q\left(x,\xi^{(j)}\right)-Q\left(x,\xi^{(i)}\right)}{w}\Big)\cdot\frac{q(i \mid j)}{q(j \mid i)}
\]
for all \(i, j \in \mathcal{I}\). It follows that \(p_x(i)P_x(i,j)=p_x(j)P_x(j,i)\) and so the Markov chain admits \(p_x\) as the stationary distribution. \hfill \(\Box\)
\end{example}
Examples~\ref{ex:MH-proposal}--\ref{ex:barker-proposal} demonstrate that the proposed procedure does not demand evaluating all the recourse functions \(Q(x, \xi^{(i)})\) in each iteration. In fact, under either MH or Barker acceptance, one evaluates the function exactly twice, one for the current state \(i\) and the other for the proposed state \(j\). This equals solving two linear programs~\eqref{QP} and hence is computationally much cheaper than the exact gradient estimation or its stochastic approximation. We shall rely on either MH or Barker acceptance distributions in the remainder of the paper. 

In addition, the proposed procedure is different from a conventional Markov chain Monte Carlo procedure, which 
would freeze \(x_n\) and apply
\(P_{x_n}\) for many transitions, possibly after a burn-in period, to obtain
a sample close to following the stationary distribution \(p_{x_n}\). Since \(p_{x_n}\)
changes with \(x_n\), this inner sampling procedure would incur higher computational burden in every optimization iteration. Instead, we carry \(x_n\) and \(P_{x_n}\) forward and apply only a small number of transitions, possibly just one, before updating \(x_n\). The ensuing Markov chain and the optimization iterates therefore evolve on the same time scale. Concretely, after sampling a scenario \(I_{n+1}\) according to \(P_{x_n}\), we form
\[
Z_{x_n}(I_{n+1})
:=
g_f(x_n)-E^\top\pi^\star(x_n,I_{n+1})
\]
and update
\[
x_{n+1}
=
\Pi_{\mathcal X}\bigl(
x_n-\alpha_{n+1}Z_{x_n}(I_{n+1})
\bigr).
\]
based on a certain step size \(\alpha_{n+1}\) to be specified later. This yields a coupled framework in which the state is carried
forward to identify an informative scenario, while the ensuing
scenario-wise gradient updates the decision incumbent. We summarize the complete
framework in Algorithm~\ref{alg:MH-SGD2}.

\begin{algorithm}[h]
\caption{Projected Stochastic Subgradient Descent with Soft Separation}
\label{alg:MH-SGD2}
\begin{algorithmic}[1]
\Require Proposal kernel $q(\cdot\mid\cdot)>0$, $x_0$, parameters $w>0$, iteration budget $N$, state space \(\I\)
\Ensure Weighted average iterate $\bar x_N$
\State $x\gets x_0$, $I_0\sim\operatorname{Unif}\{\I\}$, $S\gets0$, $\bar x\gets0$
\For{$n=0,\ldots,N-1$}
    \State $\alpha_{n+1}\gets(n+1)^{-1/2}$ and sample $J_n\sim q(\cdot\mid I_n)$
    \State Evaluate $Q(x,\xi^{(I_n)})$ and $Q(x,\xi^{(J_n)})$ by solving \eqref{QP}
    \State Compute the MH or Barker acceptance probability $A_x(I_n,J_n)$
    \State With probability $A_x(I_n,J_n)$, set $I_{n+1}\gets J_n$
    \State Obtain the optimal dual solution  \(\pi^\star(x,\xi^{(I_{n+1})})\) from \eqref{QP}
  
    \State $Z_{x_n}(I_{n+1})\gets g_f(x)-E^\top\pi^\star(x,\xi^{(I_{n+1})})$
    \State $x\gets\Pi_{\mathcal X}\!\left(x-\alpha_{n+1}Z_{x_n}(I_{n+1})\right)$
    \vspace{0.5em}
    \State $ \begin{cases}
        S\gets S+\alpha_{n+1},
          \ \bar x\gets\bar x+\alpha_{n+1}x \quad &\text{(Step-size Average)}\\[0.5em]
        S\gets S+(n+1), \ 
    \bar{x}\gets\bar{x}+(n+1)x \quad &\text{(Linear Average)}
    \end{cases}
    $
    \vspace{0.5em}
\EndFor
\State \Return $\bar x_N\gets\bar x/S$
\end{algorithmic}
\end{algorithm}

\subsection{Theoretical Guarantee}
Although the one-step Markov sample \(Z_{n+1}\) generally does not yield a conditionally unbiased subgradient estimate, our analysis shows that the resulting bias remains controllable because the chain tracks the evolving target distribution \(p_{x_n}\) sufficiently closely. Problems of this type, in which first-order estimates are constructed from observations generated by a
decision-dependent Markov chain, arise naturally in policy-gradient reinforcement learning~\citep{SuttonBarto2018,CheDongTong2026}, state-dependent performative prediction~\citep{LiWai2022}, and queueing and inventory systems~\citep{LiEtAl2026}. Our problem of interest~\eqref{main} does not by itself possess such an adaptive data structure, nor does it admit a direct stochastic subgradient implementation of this form. Rather, this structure emerges from the specific soft-separation mechanism developed in the last subsection.

We establish theoretical guarantee in two major steps.  We first use a Poisson-equation decomposition to control the
cumulative error induced by the time-inhomogeneous Markov chain. Then,
we combine this error bound with the first-order recursion to
establish convergence of Algorithm~\ref{alg:MH-SGD2}.

To proceed, we first introduce the following lemma.

\begin{lemma}[Bounded subgradients of $F_w$]\label{prop:L-smooth}
Under Assumptions~\ref{assump:5}--\ref{assump:1}, $F_w$ is convex on $\mathcal X$ and admits the subgradient selection $z_w(x)
:=
\sum_{i\in\mathcal I}
p_x(i)\left(g_f(x)-E^\top\pi^\star(x,i)\right)
\in\partial F_w(x).$ It holds that 
\[
\|Z_x(i)\|\leq L_f+L'_Q,
\quad 
\|z_w(x)\|\leq L_f+L'_Q.
\]
\end{lemma}

Next, we analyze the error induced by the one-step Markov sample through a centered function-value difference. Fix an $x^\star\in\argmin_{x\in\mathcal X}F_w(x)$ and denote 
\[
e_{n+1}
:=
Q\left(x_n,\xi^{(I_{n+1})}\right)-Q\left(x^\star,\xi^{(I_{n+1})}\right)
-
\sum_{j\in\mathcal I}p_{x_n}(j)
\bigl(Q(x_n,\xi^{(j)})-Q(x^\star,\xi^{(j)})\bigr),
\]
the centered recourse-value difference. We recall that \(e_{n+1}\) is generally neither
independent nor conditionally mean-zero unless the sampling distribution coincides with the stationary distribution \(p_{x_n}\). 
We analyze the transient behavior of \(e_{n+1}\) through the so-called fundamental matrix
\(
\big( \operatorname{Id} - P_x + \mathbf{1}\,p_x^{\top} \big)^{-1}
\) of a Markov chain. 
For a finite and irreducible Markov kernel \(P_x\), it holds that 
\[
\big( \operatorname{Id} - P_x + \mathbf{1}\,p_x^{\top} \big)^{-1} - \mathbf{1}\,p_x^{\top} \ = \ \sum_{n=0}^{\infty} \left(P_x^n - \mathbf{1}\,p_x^{\top}\right),
\]
where \(P_x^n - \mathbf{1}\,p_x^{\top}\) evaluates the difference between the \(n\)-step Markov kernel and the stationary distribution and hence the right-hand side measures the accumulated (transient) deviation of the Markov chain from its stationarity. Hence, a bounded accumulated deviation demands us to control the magnifying power of the fundamental matrix, which we formalize below as an assumption.




\begin{assumption}[Uniform norm bound of the fundamental matrix]\label{ass:invertibility}
With respect to any fixed operator norm, e.g., $\|\cdot\|_{\infty\to \infty}$, it holds that
\[
\sup_{x\in\mathcal X}\big\|\,(\operatorname{Id}-P_x+1 p_x^\top)^{-1}\big\| \;\le\; C_* 
\]
for a constant $0 < C_* < \infty$.
\end{assumption}
The validity of Assumption~\ref{ass:invertibility} and the uniform norm bound \(C_*\) depend on \(P_x\) and hence on the proposal kernel \(q(\cdot|\cdot)\). In other words, different proposal and acceptance distributions give rise to different \(C_*\). Next, we recall a general approach to bounding \(C_*\) and then elaborate its applications to various proposal kernels.
\begin{lemma}[Doeblin minorization]\label{lem:doeblin}
If $P_x$ satisfies the Doeblin minorization condition
\[
P_x(i,j)\;\ge\;\gamma_x\,\nu_x(j) \quad \forall i,j \in \mathcal{I}
\]
for some constant \(\gamma_x>0\) and probability measure \(\nu_x\) on \(\mathcal{I}\), then it holds that \(\big\|(\operatorname{Id}-P_x+\mathbf 1 p_x^\top)^{-1}\big\|_{\infty\to \infty} \leq 2/\gamma_x\).
\end{lemma}

\begin{example}[Independent Proposal]
\label{example:1.3}
Consider an independent proposal with $q(j\mid i)\equiv q(j)$ and suppose that there exists a $c\ge 1$, independent of $x$, such that
\[
c^{-1}p_x(j)\ \le\ q(j)\ \le\ c\,p_x(j),\qquad \forall x\in\mathcal X,\ \forall j\in\mathcal I.
\]
Then for all $x$ and $i\ne j$, it holds that
\(
r_x(i,j)=\frac{p_x(j)q(i)}{p_x(i)q(j)}\in[c^{-2},c^{2}].
\)

First, for the MH acceptance (see Example~\ref{ex:MH-proposal}), we have \(A_x^{\mathrm{MH}}(i,j)=\min\{1,r_x(i,j)\} \geq \min\{1,c^{-2}\} = c^{-2}\). Then, 
\(
P_x(i,j) = q(j)\,A_x^{\mathrm{MH}}(i,j)
\ge\left(1/c^2\right)q(j)
\)
whenever \(i \neq j\).

Second, for the Barker acceptance (see Example~\ref{ex:barker-proposal}), we have \(
A_x^{\mathrm{B}}(i,j)
=\frac{r_x(i,j)}{1+r_x(i,j)}
\;\ge\;\frac{c^{-2}}{1+c^{-2}}
\;=\;\frac{1}{1+c^{2}}
\). Then, 
\(
P_x(i,j) = q(j)\,A_x^{\mathrm{B}}(i,j)
\ge\left(1/(1+c^2)\right)q(j)
\)
whenever \(i \neq j\).

Finally, for both MH and Barker acceptances, we have $0\le A_x^{\mathrm{MH}}(i,j), A_x^{\mathrm{B}}(i,j)\le 1$ for all $i\neq j$. Hence, by definition of the transition kernel,
\(
P_x(i,i)
= 1 - \sum_{j\neq i} q(j)\,A_x(i,j)
\ge 1 - \sum_{j\neq i} q(j)
= q(i),
\)
where \(A_x(i,j)\) refers to either \(A_x^{\mathrm{MH}}(i,j)\) or \(A_x^{\mathrm{B}}(i,j)\).

Therefore, the Doeblin minorization holds, uniformly on \(\X\), for the MH acceptance with $\gamma_{\mathrm{MH}}=1/c^2$ and 
$\nu_{\mathrm{MH}}(j)=q(j)$, and for the Barker acceptance with 
$\gamma_{\mathrm{B}}=1/(1+c^2)$ and $\nu_{\mathrm{B}}(j)=q(j)$. It follows that \(C_*^{\mathrm{MH}} \leq 2c^2\) and \(C_*^{\mathrm{B}} \leq 2(1+c^2)\). \hfill \(\Box\)

\end{example}

\begin{example}[Approximately Correct Proposal]\label{ex:ratio-aligned}
Consider a proposal distribution that is approximately correct, in the sense that
\[
\frac{1}{\kappa}\frac{p_x(i)}{p_x(j)}\ \le\ \frac{q(i\mid j)}{q(j\mid i)}\ \le\ 
\kappa\frac{p_x(i)}{p_x(j)} \quad \forall i,j \in \I
\]
for a constant \(\kappa \geq 1\). In other words, \(q(\cdot \mid \cdot)\) is largely aligned with the target ratios. Then, $r_x(i,j)\in[1/\kappa,\kappa]$ and so
\(
A_x^{\mathrm{MH}}(i,j)\ge \frac{1}{\kappa},
\)
\(
A_x^{\mathrm{B}}(i,j)\ge \frac{1}{1+\kappa}
\). 
Define \(q_{\min}(j):=\inf_{i \in \I} q(j\mid i)>0\). Then, \(P_x(i,i)\ge q(i\mid i)\) and whenever \(i \neq j\), it holds that
\begin{align*}
P_x(i,j) \ \ge \ \kappa^{-1}q_{\min}(j) \ =: \ b_j^{\mathrm{MH}} \quad & \quad (\text{MH}) \\
P_x(i,j) \ \ge \ (1+\kappa)^{-1}q_{\min}(j) \ =: \ b_j^{\mathrm{B}}. \quad & \quad (\text{Barker})
\end{align*}
Hence, the Doeblin minorization holds, uniformly on \(\X\), for the MH acceptance with $\gamma_{\mathrm{MH}} := \sum_j \min\{b_j^{\mathrm{MH}}, q(j\mid j)\}$ and $\nu_{\mathrm{MH}}(j) := \min\{b_j^{\mathrm{MH}}, q(j\mid j)\}/\gamma_{\mathrm{MH}}$, and for the Barker acceptance with $\gamma_{\mathrm{B}} := \sum_j \min\{b_j^{\mathrm{B}}, q(j\mid j)\}$ and $\nu_{\mathrm{B}}(j) := \min\{b_j^{\mathrm{B}}, q(j\mid j)\}/\gamma_{\mathrm{B}}$. Therefore, \(C_*^{\mathrm{MH}} \leq 2/\gamma_{\mathrm{MH}}\) and \(C_*^{\mathrm{B}} \leq 2/\gamma_{\mathrm{B}}\). \hfill \(\Box\)
\end{example}

\begin{example}[Uniform Proposal]\label{ex:uniform-proposal}
Consider a uniform proposal with \(q(j\mid i) = |\I|^{-1}\) for all \(i,j \in \I\). Denote $\Delta:=\sup_{x \in \X}\max_{i,j \in \I}|Q(x, \xi^{(i)})-Q(x, \xi^{(j)})|$. Then, \(\Delta < \infty\) by Assumption~\ref{assump:1} and \(A_x^{\mathrm{MH}}(i,j)\ge e^{-\Delta/w}\), \(A_x^{\mathrm{B}}(i,j)\ge (1+e^{\Delta/w})^{-1}\). It follows that
\begin{align*}
P_x(i,j)\ge \frac{e^{-\Delta/w}}{|\mathcal I|}\quad(\mathrm{MH}),\qquad
P_x(i,j)\ge \frac{(1+e^{\Delta/w})^{-1}}{|\mathcal I|}\quad(\mathrm{Barker}).
\end{align*}
Therefore, the Doeblin minorization holds, uniformly on \(\X\), for the MH acceptance with $\gamma_{\mathrm{MH}}=e^{-\Delta/w}$ and 
$\nu_{\mathrm{MH}}(j)=q(j)$, and for the Barker acceptance with 
$\gamma_{\mathrm{B}}=(1+e^{\Delta/w})^{-1}$ and $\nu_{\mathrm{B}}(j)=q(j)$. It follows that \(C_*^{\mathrm{MH}} \leq 2e^{\Delta/w}\) and \(C_*^{\mathrm{B}} \leq 2(1+e^{\Delta/w})\).

Furthermore, \(C_*\) can be bounded without using \(\Delta\). Define \(p_{\max}:=\sup_{x\in\mathcal X}\max_{i\in\mathcal I} p_x(i)\). Then, \(A_x^{\mathrm{MH}}(i,j) =\min\left\{1,\frac{p_x(j)}{p_x(i)}\right\} \ge p_x(j)/p_{\max}\) for all \(i,j \in \I\). It follows that \(P_x(i,j) \ge p_x(j)/(|\mathcal I|p_{\max})\), that is, the Doeblin minorization holds for the MH acceptance with $\gamma_{\mathrm{MH}}=1/(|\mathcal I|p_{\max})$ and $\nu_{\mathrm{MH}}(j)=p_x(j)$. Likewise, it holds for the Barker acceptance with $\gamma_{\mathrm{B}}=1/(2|\mathcal I|p_{\max})$ and $\nu_{\mathrm{B}}(j)=p_x(j)$. These strengthen the bounds for \(C_*\) to be \(C_*^{\mathrm{MH}} \leq \min\{2e^{\Delta/w}, 2|\I|p_{\max}\}\) and \(C_*^{\mathrm{B}} \leq \min\{2(1+e^{\Delta/w}), 4|\I|p_{\max}\}\). \hfill \(\Box\)
\end{example}

The independent proposal (Example~\ref{example:1.3}) and the approximately correct proposal (Example~\ref{ex:ratio-aligned}) are useful when we can roughly locate the ``bad'' scenarios \(\xi\) with a large \(Q(x, \xi)\), because their ensuing \(p_x(\cdot)\) values are significantly larger than those of the ``good'' scenarios. For example, if \(Q(x, \xi)\) models the total cost of resource allocation (such as staffing and workforce scheduling) with random demands \(\xi\), then \(Q(x, \xi)\) is monotone (entrywise) in \(\xi\) and a scenario becomes worse as it increases~\citep{ryu2025nurse,aykin1996optimal}. In this case, a reasonable proposal distribution sets \(q(j)\) (respectively, \(q(j\mid \cdot)\)) higher for a ``bad'' scenario \(\xi^{(j)}\) that is entrywise maximal. This yields a small \(c\) (respectively, \(\kappa\)), which in turn implies a small \(C_*\).


On the other hand, one can apply the uniform proposal (Example~\ref{ex:uniform-proposal}) without any prior knowledge about~\eqref{main} or bounds learned from intermediate iterations of Algorithm~\ref{alg:MH-SGD2}. We adopt both uniform and approximately correct proposals for the numerical experiments reported in Section~\ref{sec:experiments}. In the following example, we discuss how to sample uniformly from \(\Xi\) (or its extreme points or boundary) for several commonly-applied uncertainty sets.
\begin{example}[Uniformly Sampling An Uncertainty Set]

\begin{enumerate}
    \item If \(\Xi\) is discrete and finite with \(\Xi = \{\xi^{(i)}: i \in \I\}\), then sampling uniformly from \(\Xi\) is equivalent to doing so from \(\I\).
    \item Consider a box uncertainty set with \(\Xi = \{\xi: \ell \leq \xi \leq u\}\) and we seek to sample uniformly from its extreme points. Notice that \(\Xi\) is scaled from an elementary \(\ell_{\infty}\)-ball because \(\Xi = \{\xi: \xi = \ell + Du, \ u \in[0,1]^{d_\xi}\}\), where \(D := \operatorname{diag}(u-\ell)\). We first sample \(u_i \sim \operatorname{Bernoulli}(1/2)\) independently for each entry of \(u\). Then, the ensuing \(\xi = \ell + Du\) is a uniform sample from the \(\Xi\)-extreme points.
    \item Consider a scaled \(\ell_1\) uncertainty set with \(\Xi = \{\xi: \xi = c + Du, \ \|u\|_1 \leq 1\}\), where \(D = \operatorname{diag}(D_1, \ldots, D_{d_{\xi}})\) and we seek to sample uniformly from its extreme points. To this end, we sample an \(i \sim \operatorname{Unif}\{1, \ldots, d_{\xi}\}\) and the ensuing \(\xi = c \pm D e_i\) is a uniform sample from the \(\Xi\)-extreme points.
    \item Consider a two-sided budget uncertainty set \(\Xi = \{\xi: |\xi_i| \leq 1, \sum_{i=1}^{d_{\xi}} |\xi_i| \leq \Gamma\}\), where \(\Gamma := k + \rho\) denotes the ``budget of uncertainty''~\citep{bertsimas2004price} with an integer \(k \in [d_{\xi}]\) and a fractional \(\rho \in [0, 1)\). In order to sample uniformly from the extreme points of \(\Xi\), we first notice that every extreme point has exactly (i) \(k\) entries equal to \(+1\) or \(-1\), (ii) an additional entry equal to \(+\rho\) or \(-\rho\), and (iii) all remaining entries equal to \(0\). Hence, we first sample a \(k\)-element subset \(S \subseteq [d_{\xi}]\) uniformly, and then choose \(j \in [d_{\xi}]\setminus S\) uniformly. Finally, we assign values to the \(k+1\) chosen entries uniformly from \(\{+1, -1\}\) to obtain a random sub-vector \(s \in \mathbb{R}^{k+1}\). Then, the ensuing \(\xi\) with
    \[
    \xi_i = \left\{\begin{array}{ll}
        s_i & \quad \text{if \(i \in S\)} \\[0.5em]
        \rho s_j & \quad \text{if \(i = j\)} \\[0.5em]
        0 & \quad \text{otherwise}
    \end{array}\right.
    \]
    for all \(i \in [d_{\xi}]\) is a uniform sample from the \(\Xi\)-extreme points.
    \item Consider an ellipsoidal uncertainty set with \(\Xi = \big\{\xi: \xi = c + H^{-1/2}u, \ \|u\|_2 \leq 1\big\}\) and we seek to sample uniformly from its boundary. We first generate a normalized \(u = g/\|g\|_2\) from a Gaussian random vector \(g \sim \mathcal{N}(0, I)\). Then, we accept \(u\) with probability \(\sqrt{u^{\top}Hu/\lambda_{\max}(H)}\), where \(\lambda_{\max}(H)\) denotes the largest eigenvalue of the Hessian matrix \(H\). We will either resample \(u\) in case it is rejected, or accept it and the ensuing \(\xi = c + H^{-1/2}u\) is a uniform sample from the \(\Xi\)-boundary. \hfill \(\Box\)
\end{enumerate}
\end{example}

Having addressed the non-independent stochastic approximation error, we now establish the stability of the incumbent trajectory. 
To this end, we examine the cumulative descent error $
\mathbb{E}\!\left[-2\sum_{k=1}^{n}\alpha_k e_k\right].
$
By the definition of subgradients and $e_k$, we have 
\[
\begin{aligned}
Z_{x_{k-1}}(I_k)^{\top}(x_{k-1}-x^\star)
&\ge f(x_{k-1})-f(x^\star)
+Q\left(x_{k-1},\xi^{(I_k)}\right)-Q\left(x^\star,\xi^{(I_k)}\right)\\
&=f(x_{k-1})-f(x^\star)
+\sum_{i\in\mathcal I}p_{x_{k-1}}(i)
\bigl(Q(x_{k-1},\xi^{(i)})-Q(x^\star,\xi^{(i)})\bigr)
+e_k\\
&\ge F_w(x_{k-1})-F_w(x^\star)+e_k,
\end{aligned}
\]
where \(x^\star \in \operatorname{argmin}_{x \in \X}F_w(x)\) and the last line follows from Jensen's inequality because the logarithm function is concave and \(p_{x_{k-1}}(\cdot)\) is a probability distribution on \(\I\):
\[
\begin{aligned}
F_w(x^\star)-F_w(x_{k-1})
&=f(x^\star)-f(x_{k-1})
+w\log\!\left(
\sum_{i\in\mathcal I}p_{x_{k-1}}(i)
\exp\!\left\{
\frac{Q(x^\star,\xi^{(i)})-Q(x_{k-1},\xi^{(i)})}{w}
\right\}\right)\\
&\ge f(x^\star)-f(x_{k-1})
+\sum_{i\in\mathcal I}p_{x_{k-1}}(i)
\bigl(Q(x^\star,\xi^{(i)})-Q(x_{k-1},\xi^{(i)})\bigr).
\end{aligned}
\] Thus, $-\alpha_k e_k$ enters the optimality gap \(F_w(x_{k-1}) - F_w(x^\star)\) through the Markov sample $Z_{x_{k-1}}(I_k)$, and a smaller
cumulative error yields a tighter convergence bound.
The next proposition bounds the cumulative descent error
through the step sizes \(\alpha_k\).

\begin{proposition}[Cumulative descent error bound]\label{prop:poisson-noV}
Suppose that the step sizes are nonincreasing with $\alpha_k>0$. 
Define constants \(B_Z := 2C_*DL_Q'\), \(L_Z:= C_*\left(2L_Q'+2D(L_Q')^2/w \right)
+2C_*^2DL_Q'\left(L_P+2L_Q'/w\right) \), where \(L_P\) denotes the Lipschitz modulus of \(P_x\) such that \(\sum_{j \in \I}|P_x(i,j)-P_{x'}(i,j)| \;\le\;  L_P\|x-x'\|\) for all \(i \in \I\) (see Lemmas~\ref{lem:lip barker}--\ref{lem:lipschitz-MH-nom}). Then, for any $n\ge 1$, it holds that
\begin{equation}\label{eq:poisson-bound}
\mathbb E\!\left[-2\sum_{k=1}^{n}\alpha_k e_k\right]
\le C_1\sum_{k=1}^{n}\alpha_k^2+C_0,
\end{equation}
where nonnegative constants
$C_0,C_1$ depend only on $(B_Z,L_Z,L_P,L_f,L'_Q)$
through
\[
C_0=4B_Z\alpha_1, \quad 
C_1=2(L'_Q+L_f)\left(L_PB_Z+L_Z\right).
\]
\end{proposition}

Proposition~\ref{prop:poisson-noV} implies that the cumulative descent error can be controlled by choosing the step sizes properly. Intuitively, the step size \(\alpha_k\) bounds the change in the iterates \(\|x_k - x_{k-1}\|\) and allows the Markov chain to carry the memory about adversarial scenarios forward. We are now ready to present the main convergence guarantee for Algorithm~\ref{alg:MH-SGD2}.
\begin{theorem}
\label{thm:continuous_convergence}
Under Assumptions~\ref{assump:5}--\ref{ass:invertibility}, 
suppose that the step sizes $\{\alpha_k\}$ are positive and nonincreasing. 
Then, for any $n \ge 1$, we have
\[
 0 \le \mathbb{E}\big[F_{w}(\bar{x}_n)\big] - F_{w}^\star
 \le
\frac{
\mathbb E[\|x_0-x^\star\|^2]+C_0
+\bigl(C_1+3(L_f+L_Q')^2\bigr)\sum_{k=1}^n\alpha_k^2
}{
2\sum_{k=1}^n\alpha_k
},
\]
where \(F^\star_w := \min_{x \in \mathcal{X}} F_{w} (x)\). Moreover, if we choose the step sizes as $\alpha_n \propto 1/\sqrt{n+1}$, 
then Algorithm~\ref{alg:MH-SGD2} achieves a convergence rate to $F_{w}^\star$ in the order of \(
O\!\left(
\left(C_*+ C_*^2/w \right)
n^{-1/2} \log n
\right)
\) 
for the step-size average and 
\(
O\!\left(
\left(C_*+ C_*^2/w \right)
n^{-1/2} 
\right) 
\) 
for the linear average, respectively. 
\end{theorem}

Theorem~\ref{thm:continuous_convergence} establishes polynomial convergence of Algorithm~\ref{alg:MH-SGD2} to the global minimum of \(F_{w}(x)\). But Lemma~\ref{lem:mellowmax-approx} implies 
that \(F_{w}(\cdot)\) can approximate the true objective function of~\eqref{main} to any prescribed accuracy by choosing \(w\) sufficiently small. Therefore, Algorithm~\ref{alg:MH-SGD2} provides a near-optimal solution to~\eqref{main} in polynomial time. 
Furthermore, we can iteratively adjust \(w\) in Algorithm~\ref{alg:MH-SGD2}, producing \(w_n\) in iteration \(n\), to gradually obtain a tighter approximation \(F_{w_n}(\cdot)\) for \(F(\cdot)\) as \(\bar{x}_n\) progresses. This adaptive variant of Algorithm~\ref{alg:MH-SGD2} converges to \(F^\star\) in polynomial time.
{\color{black}
\begin{proposition}
\label{prop:adaptive-w-1over3}
Under Assumptions~\ref{assump:5}--\ref{ass:invertibility}, define
$\displaystyle
I^\star(x)
:=
\argmax_{i\in\mathcal I}Q(x,\xi^{(i)})
$, $\displaystyle \widetilde{\Delta}_n := \max_{I \in \I}Q(x_n,\xi^{(I)})-\max_{j \notin I^\star(x_n)}Q(x_n,\xi^{(j)})$, and suppose that there exists a \(\Delta_\star > 0\) such that \(\Delta_\star \leq \widetilde\Delta_n\) for all sufficiently large \(n\). At iteration $n\ge0$, use step size $\alpha_{n}\propto n^{-1/2}$ and temperature parameter $w_n = (2\Delta_\star)/(\log((n+2)|\I|^2))$ for the Mellowmax approximation in Algorithm~\ref{alg:MH-SGD2}. Then, the linear average \(\bar{x}_n\) satisfies 
\[
0\le\mathbb E[F(\bar x_n)]-F^\star \le O\left(
\frac{1+C_*^2\log(n|\mathcal I|)}{\sqrt n}
\right). 
\]
\end{proposition}

}

\section{Soft Separation}
\label{sec:3}

The last section uses soft separation in one concrete role: generating first-order estimates for the Mellowmax surrogate. The soft-separation mechanism itself, however, is independent of the particular iterative algorithm. In this section, we abstract it from
Algorithm~\ref{alg:MH-SGD2} to establish method-independent guarantees
for tracking a time-varying adversarial distribution and to identify adversarial scenarios for all later-stage incumbents. We then apply these results to generalize the convergence guarantee of Algorithm~\ref{alg:MH-SGD2} beyond finite scenarios to general uncertainty sets. 

\subsection{The Soft-Separation Oracle}
\label{sec:3.1}


Recall that soft separation replaces the exact separation with a stateful process, which updates the state through a transition rule toward scenarios that are adversarial for the decision incumbent. 
To formalize this process, we first let $\mathcal I$ be a measurable state space equipped with its
Borel $\sigma$-algebra $\mathcal B(\mathcal I)$ and \(I\mapsto \xi_I\in\XiSet\) be a measurable mapping from states to
uncertainty realizations. We assume that the representation is sufficiently
rich to preserve the worst-case recourse value, namely,
\(
\sup_{I\in\mathcal I} \, Q(x,\xi_I)
=
\sup_{\xi\in\XiSet} \, Q(x,\xi)
\)
for all \(x \in \X\). 
Note here that \(\I\) need not be finite or discrete as in Section~\ref{sec:2}. 
For example, \(\I\) may be the boundary of a sphere.
\begin{example}[Linear recourse with an ellipsoidal uncertainty set]
\label{example:6}
Consider a linear recourse model with \(\mathcal{K} = \mathbb{R}^m_+\), 
\(
Q(x,\xi)
=
\min_y
\left\{
b^\top y:
Gy\ge h-Ex-U\xi
\right\}
\), 
and an ellipsoidal uncertainty set
\(
\Xi
=
\left\{
\bar\xi+A u:
\|u\|_2\le 1
\right\}
\) with \(A \in \mathbb{R}^{d_\xi \times d_u}\). 
Because $Q(x,\cdot)$ is convex, its maximum over $\Xi$ is attained on the
boundary of the ellipsoid. We may therefore parameterize the adversarial
state by
\[
\hskip0.35\textwidth \mathcal I := \mathbb S^{d_u-1},
\qquad
\xi_I := \bar\xi+A I. \hskip0.35\textwidth \Box
\]
\end{example}

\begin{example}
[Quadratic recourse with an Ellipsoidal uncertainty set]
\label{example:7}
Consider a quadratic recourse model with \(\mathcal{K}=\mathbb{R}^m_+\), 
\(
Q(x,\xi)
=
\min_y
\left\{
\frac{1}{2}y^\top H y
+
(b+B\xi)^\top y
:
Gy\ge h-Ex-U\xi
\right\}
\), 
and an ellipsoidal uncertainty set \(\Xi = \{\bar{\xi}+Au: \|u\|_2 \leq 1\}\), where $H \succeq 0$ and \(A \in \mathbb{R}^{d_\xi \times d_u}\). 
Because $Q(x,\xi)$ is not necessarily convex in $\xi$, the worst-case scenario may appear anywhere in \(\Xi\). Consequently, we may parameterize the adversarial state by
\[
\hskip0.365\textwidth \mathcal I := \mathbb B^{d_u},
\qquad
\xi_I := \bar\xi+A I. \hskip0.35\textwidth \Box
\]
\end{example}

Second, to rank the states at a given incumbent \(x\), we assign each
\(I\in\mathcal I\) a measurable score
\(
\widehat{Q}_x:\mathcal I\to\mathbb R
\). 
The score may be exact with \(\widehat{Q}_x(I) \equiv Q(x, \xi_I)\) or approximate with
\(
\sup_{I\in\mathcal I}
|
Q(x,\xi_I)-\widehat{Q}_x(I)
|
\le \varepsilon
\) for all \(x \in \X\) 
and \(\varepsilon\ge 0\) represents a uniform error bound. For example, the Moreau envelope \(\widehat Q_x(I)
:= 
\min_{v\in\R^{d_x}}
\left\{Q\left(v,\xi_I\right)+({1}/{2\tau})\|v-x\|^2\right\}\) admits a uniform error bound \(\varepsilon = (\tau/2) (L_Q')^2\). 

Soft separation uses \(\widehat{Q}_x\) to bias the adversarial state toward realizations with larger \(Q(x, \xi)\) values. Specifically, let \(\nu\) be a finite reference measure on \(\mathcal I\) with full support (e.g., \(\nu\) is the normalized Lebesgue measure on \(\I\)). For a temperature parameter \(w>0\), we approximate the worst-case recourse function \(\max_{\xi \in \Xi}Q(x, \xi)\) by the Mellowmax function 
\[
\mathcal M_{w}^{\nu}(x)
:=
w\log
\int_{\mathcal I}
\exp\left\{
\frac{\widehat Q_x(I)}{w}
\right\}
\nu(dI).
\]
\(\mathcal M_{w}^{\nu}(\cdot)\) generalizes \(\mathcal{M}_w(\cdot)\), which was defined in Section~\ref{sec:2} for a finite \(\I\), to a general \(\I\). We show in~\ref{apx-prop:continuous-mellowmax} that \(\mathcal M_{w}^{\nu}(\cdot)\) admits an analogous uniform error bound for approximating \(\max_{\xi \in \Xi}Q(x, \xi)\) as in Lemma~\ref{lem:mellowmax-approx}. Accordingly, the robust objective function of~\eqref{main} is approximated by 
\(
F_{w}^{\nu}(x)
:=
f(x)+\mathcal M_{w}^{\nu}(x)
\). Under Assumption~\ref{ass:lipschitz-scores}, if \(\widehat{Q}_x\) and $f$ are differentiable then \(\|\nabla_x \widehat{Q}_x(I)\|\) is bounded uniformly and so 
\(
    \nabla F_{w}^{\nu}(x)
    =
    \nabla f(x)
    +
    \int_{\mathcal I}
        \nabla_x \widehat{Q}_x(I)\,p_x^w(dI)
\), where the soft adversarial distribution \(p^w_x\) is defined through
\begin{equation*}
\frac{d p_x^w}{d\nu}(I)
=
\frac{
\exp\!\bigl(\widehat{Q}_x(I)/w\bigr)
}{
\displaystyle
\int_{\mathcal I}
\exp\!\bigl(\widehat{Q}_x(J)/w\bigr)\,d\nu(J)
} \qquad \forall I \in \I.
\end{equation*}
The distribution \(p_x^w\) assigns greater mass to states with larger scores, and a smaller value of \(w\) produces stronger concentration around the highest-scoring states. We remark that \(p_x^w\) is well-defined whenever $\nu$ is a finite measure and \(\widehat{Q}_x(I)\) is bounded. 

Unfortunately, \(p_x^w\) is challenging to sample from because of the integral \(\int_{\mathcal I}
\exp\!\bigl(\widehat{Q}_x(J)/w\bigr)\,d\nu(J)\). In addition, we need to adapt the sampling from \(p_x^w\) to the incumbent \(x\), which may be updated frequently in decision-making. To address these, we pair \(p_x^w\) with 
\begin{itemize}
    \item a Markov transition kernel \(P_x\) on \(\I\) such that \(p_x^w\) is an invariant distribution of \(P_x\) for any fixed \(x \in \X\). For example, \(P_x\) may be constructed through a proposal-acceptance scheme as in Examples~\ref{ex:MH-proposal}--\ref{ex:uniform-proposal} or a Langevin diffusion-type dynamics~\citep{roberts1996exponential}. The proposal may also exploit the geometry of $\Xi$ and information from the recourse function \(Q(x, \xi)\) and the score \(\widehat{Q}_x(I)\). When the adversarial search can be restricted to a smooth boundary, as in linear recourse with ellipsoidal uncertainty (Example~\ref{example:6}), one may use geodesic or tangent-space proposals~\citep{ByrneGirolami2013,ZappaHolmesCerfonGoodman2018}. For a full-dimensional convex uncertainty set (Example~\ref{example:7}), convex-body random walks such as hit-and-run provide feasible proposals~\citep{LovaszVempala2006}. Primal-dual information from \(Q(x, \xi)\) may further be used to construct gradient-informed proposals. For example, when a smooth approximation to $\widehat Q_x$ is available, one may employ Hamiltonian-type proposals~\citep{Neal2011}.
    \item an iterative optimization algorithm that generates a sequence \(\{x_n: n \in \mathbb{N}_+\}\) of incumbents. For example, this algorithm may be the projected subgradient descent as in Algorithm~\ref{alg:MH-SGD2} or a branch-and-cut algorithm (see Section~\ref{sec:4}). 
\end{itemize}
We then formalize the soft-separation oracle as follows.


\begin{definition}[Markovian Soft-separation Oracle]
\label{def:1}
Fix a temperature parameter \(w>0\), a score accuracy
\(\varepsilon\ge 0\), and a state representation
\(\xi_I: \I \rightarrow \Xi\). Given an incumbent sequence
\(\{x_n: n \in \mathbb{N}_+\}\), a \emph{Markovian soft-separation oracle} is a
time-inhomogeneous Markov process \(\{I_n: n \in \mathbb{N}_+\}\) on
\(\mathcal I\) satisfying \(I_{n+1}\sim P_{x_n}(I_n,\cdot),\) where, for every \(x\in\mathcal X\), the transition kernel \(P_x\)
admits the soft adversarial distribution \(p_x^w\) as an invariant distribution.
\end{definition}


A key insight of the oracle is that it decouples
adversarial relevance from exact identification in every iteration. 
Although the incumbent sequence and
the adversarial process co-evolve, their interaction can be controlled:
when the incumbent changes gradually, the oracle tracks the corresponding
adversarial target and can approach, or recover with high probability,
exact separation at later-stage iterates. We show these properties
next.

\subsection{Certification for Identifying Adversarial Scenarios}
\label{sec:2.2}
We denote by $(\X,\|\cdot\|)$ a metric space for incumbent solutions. 
Our analysis applies to both discrete (finite) and continuous (uncountable) state spaces. 
We start by imposing regularity conditions on the oracle. 
\begin{assumption}[Lipschitz score]
\label{ass:lipschitz-scores}
There exists an \(L'_{Q}<\infty\) such that, for every \(i\in\mathcal I\) and all
\(x,x'\in\mathcal X\),
\[
|\widehat{Q}_x(i)-\widehat{Q}_{x'}(i)|
\le
L'_{Q}\|x-x'\|.
\]
\end{assumption}

\begin{assumption}[Lipschitz transition kernel]
\label{ass:general-kernel-lipschitz}
There exists an \(L_P\) such that, for all \(x, y \in \X\),
\[
\sup_{I\in\mathcal I}
\|P_x(I,\cdot)-P_y(I,\cdot)\|_{\mathrm{TV}}
\ \le \ L_P\|x-y\|,
\]
where \(\|P-Q\|_{\mathrm{TV}}:=\sup_{B\in\mathcal B(\mathcal I)}|P(B)-Q(B)|\) denotes the total variation between measures \(P, Q\).
\end{assumption} 

\begin{assumption}[Uniform mixing]
\label{ass:uniform-mixing}
There exists an \(\bar\eta\in[0,1)\) such that, for every \(x\in\mathcal X\), the Dobrushin coefficient of \(P_x\) satisfies
\[
\eta_{\mathrm{Dob}}(P_x) \ := \ \sup_{I,J\in\mathcal I} \left\| P_x(I,\cdot)-P_x(J,\cdot) \right\|_{\mathrm{TV}} \ \le \ \bar\eta.
\]
\end{assumption}

Assumption~\ref{ass:lipschitz-scores} holds for the recourse function \(Q(x,\xi)\), i.e., for \(\widehat{Q}_x(\cdot) \equiv Q(x, \cdot)\) (see Lemma~\ref{lem:recourse_lipschitz_smoothing}, whose proof applies to a general \(\I\)) and consequently for its Moreau envelope as well. 
In addition, Assumption~\ref{ass:general-kernel-lipschitz} follows naturally from Assumption~\ref{ass:lipschitz-scores}. For example, for both MH and Barker acceptance, we show the Lipschitz continuity of \(P_x\) for a discrete and finite \(\I\) in Lemmas~\ref{lem:lip barker}--\ref{lem:lipschitz-MH-nom} and for a general \(\I\) in Lemma~\ref{lem:lip-kernal-general} through Lipschitz scores. 
Finally, Assumption~\ref{ass:uniform-mixing} enables uniform mixing of the Markov process characterized by \(P_x\)~\citep[see][]{Dobrushin1956} and it holds when \(P_x\) is induced by appropriate proposal and acceptance distributions. For example, \(\eta_{\mathrm{Dob}}(P_x)\) can be uniformly bounded when the proposal distribution is uniform on \(\I\) and the acceptance distribution is either MH or Barker.



\begin{example}[Uniform Mixing with Uniform Proposal] \label{ex:uniform-mixing}
We continue Example~\ref{example:6} with \(
\mathcal I=\mathbb S^{d_\xi-1}\),
\(\xi_I=\bar\xi+AI\), and \(\Pi:=\{\lambda\ge0:G^\top\lambda=b\}\). Consider a proposal distribution \(q(dJ \mid I)=\nu(dJ)\), where \(\nu\) is the uniform probability measure on \(\I\), and the score \(\widehat{Q}_x(I) = Q(x, \xi_I)\) for all \(I \in \I\). Then, \(A_x^{\mathrm{MH}}(I,J)\ge e^{-\Delta/w}\) and \(A_x^{\mathrm{B}}(I,J)\ge (1+e^{\Delta/w})^{-1}\) for all \(I, J \in \I\) (recall that $\Delta=\sup_{x \in \X}\sup_{i,j \in \I}|Q(x, \xi^{(i)})-Q(x, \xi^{(j)})|$ and \(\Delta < \infty\) by Assumption~\ref{assump:1}). It follows that
\begin{align*}
P_x(I,dJ) \ \geq \ q(dJ|I) A_x^{\mathrm{MH}}(I,dJ) \ \ge \ e^{-\Delta/w}\nu(dJ) \quad & \quad (\text{MH}) \\
P_x(I,dJ) \ \geq \ q(dJ|I) A_x^{\mathrm{B}}(I,dJ) \ \ge \ (1+e^{\Delta/w})^{-1}\nu(dJ). \quad & \quad (\text{Barker})
\end{align*}
Then, Doeblin minorization~\citep[see, e.g., Lemma 4.3.13 of][]{cappe2005inference} implies that \(\eta_{\mathrm{Dob}}(P_x) \leq 1 - e^{-\Delta/w}\) under the MH acceptance and \(\eta_{\mathrm{Dob}}(P_x) \leq 1 - (1+e^{\Delta/w})^{-1}\) under the Barker acceptance.

Alternatively, we can bound \(\eta_{\mathrm{Dob}}(P_x)\) through the Lipschitz continuity of \(\widehat{Q}_x\). Specifically, recall that \(Q(x, \cdot)\) is \(L'_{\xi}\)-Lipschitz continuous in \(\xi\) (see Lemma~\ref{lem:recourse_lipschitz_smoothing}). Then, \(\widehat{Q}_x(\cdot)\) is \((L'_{\xi}\|A\|_{2 \rightarrow 2})\)-Lipschitz continuous in \(I\). It follows that, for all \(I, J \in \I\), 
\begin{align*}
A_x^{\mathrm{MH}}(I,J) \ = \ \min\left\{1, \ 
\frac{(dp_x^w/d\nu)(J)}{(dp_x^w/d\nu)(I)}\right\}
\ \ge & \ 
\frac{(dp_x^w/d\nu)(J)}
{\|dp_x^w/d\nu\|_\infty} \\[0.5em]
A_x^{\mathrm{B}}(I,J) \ = \ 
\frac{(dp_x^w/d\nu)(J)}{(dp_x^w/d\nu)(I) + (dp_x^w/d\nu)(J)}
\ \ge & \ 
\frac{(dp_x^w/d\nu)(J)}
{2\|dp_x^w/d\nu\|_\infty}.
\end{align*}
Hence, for every measurable set \(B \subseteq \I\),
\begin{align*}
P_x(I,B) \ge
\int_B A_x^{\mathrm{MH}}(I,J)\,\nu(dJ)
\ \ge \ \left\|\frac{dp_x^w}{d\nu}\right\|_\infty^{-1}
\int_B\frac{dp_x^w}{d\nu}(J)\,\nu(dJ)
\ = \ \left\|\frac{dp_x^w}{d\nu}\right\|_\infty^{-1}
p_x^w(B) \quad & \quad (\text{MH}) \\
P_x(I,B) \ge 
\int_B A_x^{\mathrm{B}}(I,J)\,\nu(dJ)
\ \ge \ \frac{1}{2} \left\|\frac{dp_x^w}{d\nu}\right\|_\infty^{-1}
\int_B\frac{dp_x^w}{d\nu}(J)\,\nu(dJ)
\ = \ \frac{1}{2} \left\|\frac{dp_x^w}{d\nu}\right\|_\infty^{-1}
p_x^w(B). \quad & \quad (\text{Barker})
\end{align*}
Finally, it can be shown that \(\left\|dp_x^w/d\nu\right\|_\infty^{-1} \geq C_{d_\xi}
(
w/(L_\xi\|A\|_{2\to2})
)^{d_\xi-1}\) for a constant \(C_{d_\xi} > 0\), whose value depends on \(d_\xi\) only (see a proof and the specific value of \(C_{d_\xi}\) in Appendix~\ref{apx-ex:uniform-mixing}). Therefore, Doeblin minorization~\citep[see, e.g., Lemma 4.3.13 of][]{cappe2005inference} implies that \(\eta_{\mathrm{Dob}}(P_x) \leq 1-
\max\big\{
e^{-\Delta/w},
\;
C_{d_\xi}
(
w/(L_\xi\|A\|_{2\to2})
)^{d_\xi-1}
\big\}\) under the MH acceptance and \(\eta_{\mathrm{Dob}}(P_x) \leq 1-
\max\big\{
e^{-\Delta/w},
\;
C_{d_\xi}/2\cdot
(
w/(L_\xi\|A\|_{2\to2})
)^{d_\xi-1}
\big\}\) under the Barker acceptance. \hfill \(\Box\)
\end{example}

Our target is to show that the soft-separation oracle discovers adversarial scenarios as the incumbent sequence and the ensuing Markov kernels co-evolve. To this end, we next make the notion of adversary concrete.
\begin{definition}
[$\widetilde\Delta$-adversarial scenario set]
\label{def:adv-set}
We define \(\displaystyle Q^\star(x) := \sup_{I \in \I} Q(x, \xi_I)\) and, for \(\widetilde{\Delta} \geq 0\), 
\[
I_{\widetilde\Delta}(x):=
\{I\in\mathcal I:Q(x,\xi_I)
\ge Q^\star(x)-\widetilde\Delta\}.
\]
In addition, with \(\widehat Q^\star(x) := \sup_{I \in \I} \widehat{Q}_x(I)\), we define
\[
\widehat I_{\widetilde\Delta}(x):=
\{I\in\mathcal I:\widehat Q_x(I)
\ge \widehat Q^\star(x)-\widetilde\Delta \}.
\]
\end{definition}
Intuitively, \(I_{\widetilde\Delta}(x)\) consists of scenarios that are close to the worst-case one with respect to the (true) recourse function and \(\widehat I_{\widetilde\Delta}(x)\) refers to its counterpart with respect to the (approximate) score function. The main result of this section shows that, after sufficiently long co-evolution, the soft-separation oracle identifies $\widetilde\Delta$-adversarial scenarios with high probability.

\begin{theorem}[High-probability identification for general \(\I\)]
\label{thm:1}
Under Assumptions~\ref{ass:lipschitz-scores}--\ref{ass:uniform-mixing}, consider the incumbent sequence $\{x_n: n\in \mathbb{N}_+\}$ on a measurable space $(\mathcal{X},\mathcal{B}_{\mathcal{X}})$
with respect to the natural filtration $\{\mathcal{F}_n: n\in \mathbb{N}_+\}$ generated by $(\mathcal{I}_m, X_m)_{m\le n}$:
\[
X_0, I_0 \ \to \ I_1 \ \to \ X_1 \ \to \ I_2 \ \to \ X_2 \ \to \ \cdots
\]
Conditional on $\mathcal{F}_n$, the next state $\mathcal{I}_{n+1}$ is sampled from $\mathcal{I}_n$ through the one-step Markov kernel $P_{x_n}$, which admits an invariant distribution $p^w_{x_n}$ on $\mathcal{I}$.

Fix $p > 0$, \(\widetilde\Delta > 2\varepsilon\), and suppose that the incumbent movement satisfies 
\(
\|X_n-X_{n-1}\|\le c\alpha_n
\)
almost surely, where $\{\alpha_n: n \in \mathbb{N}_+\}$ satisfies $\alpha_n\le C_\alpha n^{-p}$ for some $C_{\alpha} > 0$.  Then, for every sufficiently large $n$ such that
$f(n) := \lceil (p+1)\log n/|\log\bar\eta|\rceil<n/2$ and $2^pcC_\alpha
f(n) n^{-p}  \le {(\widetilde\Delta-2\varepsilon)}/{(4 L'_{Q})}$, it holds that
\[
\mathbb P\Big(I_n\notin
I_{\widetilde\Delta}(X_{n-1})\Big)
\le \bar\eta^{f(n)}
+\frac{2^pcL_PC_\alpha}{1-\bar\eta}
f(n) n^{-p}
+  \sup_{x \in \mathcal X} \left\{\frac{\nu\left( \widehat I_{(\widetilde\Delta-2\varepsilon)/2} 
(x )^c \right)}{\nu\left( \widehat I_{(\widetilde\Delta-2\varepsilon)/4} 
(x ) \right)}\right\} \cdot \exp \left\{- \frac{\widetilde{\Delta} -2\varepsilon}{4 w} \right\}.
\]
In particular, the first two terms on the right-hand side scale in order \(O\big((L_P/(1-\bar{\eta}))n^{-p}\log n\big)\).

\end{theorem}

In Theorem~\ref{thm:1}, the first two terms of the probability bound for not identifying a \(\widetilde\Delta\)-adversarial scenario decay as \(n\) increases. For example, if we choose \(p=1\), then the step size \(\alpha_n \propto n^{-1}\) implies a decay in the order \(O(n^{-1}\log n)\). In addition, the final term evaluates the impact of the approximate score function \(\widehat{Q}_x\) on the identification probability. We notice that this term quickly decreases as \((\widetilde\Delta-2\varepsilon)\) increases. We elaborate this in the following example, whose proof can be found in Appendix~\ref{app:proof-concentration}.
{\color{black}
\begin{example}
\label{ex:concentration}
Suppose that \(\I = \Xi \subseteq \mathbb{R}^{d_\xi}\) with \(\xi_I = I\), $\Xi$ is
convex, full-dimensional, and compact with a diameter $D_{\mathcal I}:=\operatorname{diam}(\Xi)<\infty$, and $\widehat Q_x(I) = Q(x, \xi_I)$ with \(\varepsilon = 0\). Take $\nu$ as the normalized Lebesgue measure
on $\Xi$. Then, for any \(\widetilde\Delta > 0\), it holds that
\[
\sup_{x\in\mathcal X}
\left\{
\frac{
\nu\bigl(\widehat I_{\widetilde\Delta/2}(x)^c\bigr)
}{
\nu\bigl(\widehat I_{\widetilde\Delta/4}(x)\bigr)
}
\right\}
\le
\max\left\{
1,\frac{4L'_\xi D_{\mathcal I}}{\widetilde\Delta}
\right\}^{d_\xi}
\]
and therefore the probability bound of Theorem~\ref{thm:1} is in the order \(O\big(n^{-1}\log n + \exp\{-\widetilde\Delta/(4w)\}\big)\). In other words, the longer we run the Markov chain and the larger \(\widetilde\Delta\) is, the soft-separation oracle becomes more powerful for identifying adversarial scenarios. \hfill \(\Box\)
\end{example}

We turn to the probability bound for identifying adversarial scenarios in a finite state space \(\I\). The following theorem takes advantage of the finiteness of \(\I\) and a uniform gap between maximal and other recourse costs to yield a geometric decay.

\begin{theorem}[High-probability identification for finite \(\I\)]
\label{thm:finite-high-prob}
Under the assumptions of Theorem~\ref{thm:1}, suppose that \(\I\) is finite. Let \(\nu\) be the uniform distribution on \(\I\), \(I^\star(x) := \{i \in \I: Q(x, \xi^{(i)}) = \max_{I \in \I}Q(x, \xi^{(I)})\}\), \(\widetilde{\Delta}_n := \max_{I \in \I}Q(x_n,\xi^{(I)})-\max_{j \notin I^\star(x_n)}Q(x_n,\xi^{(j)})\), and suppose that there exists a \(\Delta_\star > 4\varepsilon\) such that \(\Delta_\star \leq \widetilde\Delta_n\) for all \(n \in \mathbb{N}_+\). Then, for all \(n \geq 2\), it holds that
\[
\mathbb P\Big(I_n\notin I^\star(X_{n-1})\Big)
\le
\bar\eta^{\,n-1}\mathbb P(I_1\notin I^\star(X_0))
+\frac{2c\bar\eta L_Q'}{\Delta_\star-4\varepsilon}
\sum_{t=1}^{n-1}\bar\eta^{\,n-1-t}\alpha_t +\frac{1-\bar\eta^{\,n-1}}{1-\bar\eta} (|\mathcal I|-1)
e^{-(\Delta_\star-2\varepsilon)/w}.
\]
In addition, for all $n \ge n_0:=\max\left\{2,\,
1+\left\lceil
\left(\frac{2cC_\alpha L_Q'}{\Delta_\star-4\varepsilon}\right)^{1/p}
\right\rceil\right\}$, it holds that 
\[
\mathbb P\Big(I_n\notin I^\star(X_{n-1})\Big) \le 
\bar\eta^{\,n-n_0}+  \frac{1-\bar\eta^{\,n-1}}{1-\bar\eta} (|\mathcal I|-1)
e^{-(\Delta_\star-2\varepsilon)/w}.
\]

\end{theorem}

In Algorithm~\ref{alg:MH-SGD2}, one can apply the soft separation to establish 
\(\underline{F}_n := \min_{x \in \X} \big\{f(x) + \max_{m \in [n]} Q(x, \xi_{I_m})\}\) and \(\overline{F}_n := f(x_{n-1}) + Q(x_{n-1}, \xi_{I_n}) + \tilde{\Delta}\) in each iteration \(n\). We notice that \(\underline{F}_n\) is a deterministically valid lower bound of \(F^\star\) while, by Theorems~\ref{thm:1}--\ref{thm:finite-high-prob}, \(\overline{F}_n\) is a high-confidence upper bound for \(F^\star\). Hence, one can design an early termination criterion for Algorithm~\ref{alg:MH-SGD2} based on \(\underline{F}_n\) and \(\overline{F}_n\). Finally, we extend Theorem~\ref{thm:continuous_convergence}  to a general state space \(\I\) through the soft-separation oracle.

\begin{theorem}
\label{thm:continuous-state-convergence}
Fix $w>0$, let $P_x$ be the Markov kernel induced by the uniform proposal distribution on \(\I\) and the MH acceptance, and set the step sizes as \(\alpha_n = \alpha_0/\sqrt{n}\) in Algorithm~\ref{alg:MH-SGD2} for an \(\alpha_0 > 0\). 
Then, under Assumptions~\ref{assump:5},~\ref{ass:general-kernel-lipschitz}, and~\ref{ass:uniform-mixing}, it holds that
\[
\mathbb E\left[
F_{w}^{\nu}(\bar x_n)
\right]
-
\min_{x\in\mathcal X}F_{w}^{\nu}(x)
=
\begin{cases}
   O\left[
 \left(C_*+ C_*^2/w\right)
 \frac{\log n}{\sqrt n}\right] &  \text{\emph{for Step-size Average}}
 \\[0.5em]
 O\left[
 \left(C_*+ C_*^2/w\right)
 \frac{1}{\sqrt n}\right] & \text{\emph{for Linear Average}}
\end{cases}
\]
where $C_* = (1+\bar\eta)/(1-\bar\eta)$. 
\end{theorem}

\section{Mixed-Integer Here-and-Now Decisions}
\label{sec:4}
We extend the application of the soft separation to~\eqref{main} with mixed-integer decision variables, i.e., \(\X \subseteq \mathbb Z^{p}\times\R^{d_x-p}\). For ease of exposition, we write $x := (x_z,x_c)$, where $x_z$ and $x_c$ denote the integer and continuous entries of \(x\), respectively. We shall develop a branch-and-cut method that generates valid cuts for the robust objective through the soft-separation oracle from Section~\ref{sec:3}. To avoid clutter, we restrict ourselves to the \emph{linear} recourse, i.e., we make Assumptions~\ref{assump:1}--\ref{assump:linear-recourse} and work with a finite state space \(\I\).


We rewrite~\eqref{main} as an epigraphic formulation
\[
\min_{x\in\X,\theta\in\R}\; f(x)+\theta
\qquad
\text{s.t.}
\quad
\theta\ge Q(x,\xi^{(i)}) \quad \forall i\in\I.
\]
For any scenario \(i \in \I\), strong duality gives
\(
Q(x,\xi^{(i)})
=
\max_{\pi\in\Pi}\,
\pi^\top(h-Ex-U\xi^{(i)})\).
Hence, for all \(\bar\pi_i \in
\argmax_{\pi\in\Pi} \, \pi^\top(h-E\bar x-U\xi^{(i)})\) with respect to any \(\bar{x} \in \X\), the (dual Benders') cut
\begin{equation}
\theta \ \ge \ \bar\pi_i^\top(h-Ex-U\xi^{(i)}) \label{eq:dual-benders}
\end{equation}
is valid for the robust epigraph. We apply the soft separation to identify adversarial scenarios \(\xi^{(i)}\) and then apply the corresponding Benders' cut to solve~\eqref{main} through branch-and-cut. 

\subsection{Applying the soft-separation oracle in branch-and-cut}

At an integer-feasible incumbent $(\bar x,\bar\theta)$, 
we apply the soft separation established in Section~\ref{sec:3} to return a (high-confidence) adversarial scenario and generate a deterministically valid cut. 
The main technical issue, however, is that the probabilistic guarantee of Theorem~\ref{thm:finite-high-prob} requires the incumbent sequence 
to evolve gradually. In particular, the
theorem assumes that consecutive incumbents satisfy \(\|x_n-x_{n-1}\|\le \alpha_n\) for a decaying sequence $\alpha_n\sim n^{-p}$. A generic
branch-and-bound incumbent sequence need not satisfy this requirement, because
consecutive incumbents may have different integer components and are far apart. To make the tracking analysis applicable, we maintain a Markov chain for each newly discovered integer assignment \(x_z\) and will continue using the same chain for all future incumbents \(x\) that share the same integer assignment.

Specifically, for each feasible solution \(\bar x \equiv (\bar x_z,\bar x_c) \in \mathbb Z^{p}\times\R^{d_x-p}\) encountered in the branch-and-bound tree, we maintain \(\state (\bar x_z)\), \(\refpoint(\bar x_z)\), and \(\kappa(\bar x_z)\). Here, $\state(\bar x_z)\in\mathcal I$ records the most recently sampled adversarial scenario
associated with $\bar x_z$,
$\refpoint(\bar x_z)\in\mathbb R^{d_c}$ records the continuous entries of the corresponding
incumbent, and $\kappa(\bar x_z)\in\mathbb N_+$ records the number of previous oracle
calls associated with $\bar x_z$. 
When \(\bar{x}\) is encountered, we compare \(\bar x_c\) with the previously recorded $\refpoint(\bar x_z)$.  
If
\begin{equation}
\label{eq:move-threshold}
\|\bar x_c-\refpoint(\bar x_z)\|
\le
\frac{c_0}{\max\{\kappa(\bar x_z),1\}^{p}},
\end{equation}
then \(\bar{x}\) stays within the neighborhood designated in Theorem~\ref{thm:finite-high-prob} with a width $\alpha_n=c_0 \max\{\kappa(\bar x_z),1\}^{-p}$. The oracle therefore performs only one transition under $P_{\bar x}(\state(\bar x_z), \cdot)$. 
On the other hand, if~\eqref{eq:move-threshold} fails to hold, then the target
distribution \(p_{\bar{x}}\) may have changed too much for the one-step tracking argument
to apply. In this case, the oracle performs $g(\kappa(\bar{x}_z)) \geq 1$ transitions under the
fixed kernel $P_{\bar x}$ before returning an adversarial scenario. 
This yields an adaptive Markov chain, where the incumbent evolves at a different pace from the state, as opposed to a co-evolution as in Theorem~\ref{thm:finite-high-prob}. We therefore extend Theorem~\ref{thm:finite-high-prob} to find $g(\kappa(\bar{x}_z))$.
\begin{corollary}[Identification under adaptive transition steps]
\label{cor:adaptive-transition-identification}
Suppose the assumptions of Theorem~\ref{thm:finite-high-prob} hold, with $\Delta_\star>4\varepsilon$. 
Consider the Markov chain associated with any fixed integer decision \(\bar{x}_z\), let $\kappa=\kappa(\bar x_z)\ge1$, and we perform one transition whenever~\eqref{eq:move-threshold} holds and
$g(\kappa)$ transitions otherwise, where
\begin{equation} \label{eq:transition-budget}
g\!\left(\kappa\right)
:=
\max\left\{
    1,
    \left\lceil
    \frac{
        p\log\!\left(\max\{\kappa,1\}\right)
        -
        \log\!\left(
            \dfrac{2c_0\bar\eta L'_{Q}}
            {\Delta_\star-4\varepsilon}
        \right)
    }{
        |\log\bar\eta|
    }
    \right\rceil
\right\} = O(\log(\kappa)).
\end{equation}
Then, it holds that 
\[
\mathbb P\!\Big(
    \state(\bar x_z)
    \notin I^\star\!\big(\refpoint(\bar x_z)\big)
\Big)
    \le
\bar\eta^{\,\kappa-1}
+\frac{2c_0\bar\eta L_{Q'}}{\Delta_\star-4\varepsilon}
\sum_{t=1}^{\kappa-1}
\bar\eta^{\,\kappa-1-t}t^{-p}
+\frac{|\mathcal I|-1}{1-\bar\eta}
\left(1-\bar\eta^{\,\kappa-1}\right)
\exp\!\left\{-\frac{\Delta_\star-2\varepsilon}{w}\right\}.
\]
In particular, with constant
\(
  C_{\bar\eta,p}
    :=
    \frac{1}{1-\bar\eta}
    \left[
        2^p
        +\bar\eta^{-1/2}
        \left(
            2p/(e|\log\bar\eta|)
        \right)^p
    \right]
\), the right-hand side of the last inequality is bounded from above by 
\[
 \widehat\delta(\kappa)
    :=
    \frac{|\mathcal I|-1}{1-\bar\eta}
    \exp\!\left\{-\frac{\Delta_\star-2\varepsilon}{w}\right\}
     +\bar\eta^{\,\kappa-1}
    +\frac{2c_0\bar\eta L_{Q'}}
    {\Delta_\star-4\varepsilon}
     C_{\bar\eta,p}\kappa^{-p}.
\]
\end{corollary}
We summarize this adaptive variant of the soft-separation oracle in Algorithm~\ref{alg:oracle}.

\begin{algorithm}[H]
\caption{Adaptive Soft-Separation Oracle}
\label{alg:oracle}
\begin{algorithmic}[1]
\Require Incumbent $(\bar x,\bar\theta)$ with
$\bar x=(\bar x_z,\bar x_c)$; memory maps $\state(\cdot)$, $\refpoint(\cdot)$, $\kappa(\cdot)$
\Ensure State $I$, valid cut, and violation indicator

\If{$\bar x_z$ has not been observed before}
    \State Initialize
    $\state(\bar x_z)\in\I$,
    $\refpoint(\bar x_z)\gets\bar x_c$, and
    $\kappa(\bar x_z)\gets0$
\EndIf

\vspace{0.5em}
\State
$\displaystyle
m\gets
\begin{cases}
1,
& \text{if} \ \kappa(\bar x_z)=0\ \text{or}\ 
  \|\bar x_c-\refpoint(\bar x_z)\|\le c_0\max\{\kappa(\bar x_z),1\}^{-p}\\[0.5em]
g(\kappa(\bar x_z)),
& \text{otherwise, through~\eqref{eq:transition-budget}}
\end{cases}
$
\vspace{0.5em}
\State Sample $I\sim P_{\bar x}^{m}(\state(\bar x_z),\cdot)$
\State Define a valid cut through~\eqref{eq:dual-benders} with \(i = I\)
\State
$\mathrm{violated}
\gets
\BFone\{Q(\bar x,\xi^{(I)})>\bar\theta\}$

\State
$\state(\bar x_z)\gets I$,
$\refpoint(\bar x_z)\gets\bar x_c$, and
$\kappa(\bar x_z)\gets \kappa(\bar x_z)+1$

\State \Return $I$, the valid cut, and $\mathrm{violated}$
\end{algorithmic}
\end{algorithm}

\subsection{Branch-and-cut with certification}
\label{subsec:mip-algorithm}
We incorporate the adaptive soft-separation oracle (Algorithm~\ref{alg:oracle}) in a branch-and-cut algorithm for solving~\eqref{main} with mixed-integer \(x\). When implementing this in Gurobi or CPLEX, we can embed Algorithm~\ref{alg:oracle} via the \texttt{lazy callback}. Let $\mathcal C \subseteq \I\times\Pi$ denote the pool of cuts produced by Algorithm~\ref{alg:oracle}. This produces a relaxation of~\eqref{main},
\begin{equation}
\label{eq:integer-master}
L(\mathcal C)
:=
\min_{x\in\overline{\X},\;\theta\in\mathbb R}
\left\{
f(x)+\theta: \quad
\theta\ge
\pi^\top(h-Ex-U\xi^{(i)}),
\ \ 
(i,\pi)\in\mathcal C
\right\},
\end{equation}
where \(\overline{\X}\) denotes a relaxation of \(\X\) produced by the branch-and-cut method. Because every cut in $\mathcal C$ is deterministically valid, \(L(\mathcal C)\) is a lower bound for \(F^\star\), i.e., 
\(
L(\mathcal C)\le F^\star
\). On the other hand, branch-and-cut terminates when the upper bound
\[
\widehat{U} \ := \ f(\bar{x}) + \bar{\theta}
\]
associated with an integer incumbent \((\bar{x}, \bar\theta)\) is close enough to \(L(\mathcal{C})\), e.g., when \(\widehat{U} - L(\mathcal{C}) \leq \varepsilon_{\text{opt}}\) with a prespecified optimality gap \(\varepsilon_{\text{opt}}\). However, \(\widehat{U}\) is not necessarily a valid upper bound for \(F^\star\) because \(\bar{\theta}\) may underestimate \(Q^\star(\bar{x}) \equiv \max_{\xi \in \Xi}Q(\bar{x}, \xi)\). Consequently, the branch-and-cut method alone may not certify an \(\varepsilon_{\text{opt}}\)-optimal solution to~\eqref{main}. To this end, we propose to embed a certification procedure within the branch-and-cut to guarantee that \((\bar{x}, \bar{\theta})\) is an \(\varepsilon_{\text{opt}}\)-optimal solution with probability at least \(1 - \varepsilon_{\text{false}}\), where \(\varepsilon_{\text{false}}\) denotes a threshold for falsely identifying an optimal solution. 

This procedure solves~\eqref{main} in a single branch-and-bound tree and, for each integer incumbent \(\bar{x}\) encountered, employs Algorithm~\ref{alg:oracle} to (i) find a Benders' cut~\eqref{eq:dual-benders} and (ii) certify a high-confidence upper bound \(\widehat U\). 
Whenever a violated cut is found, we add it to $\mathcal C$ and continue the branch-and-bound. Otherwise, we repeat calling Algorithm~\ref{alg:oracle} at \(\bar{x}\) for a sufficiently many times to certify \(\widehat U\). 

First, 
to find a Benders' cut, we make one call to Algorithm~\ref{alg:oracle} when we encounter an integer incumbent $(\bar x,\bar \theta)$ in the branch-and-bound tree. If no violated cut is found, we record $k=\kappa (\bar x_z)$ after this call. Consequently, \(k\) includes all previous calls associated with $\bar x_z$, including earlier certification calls specified next. 


Second, to certify upper bounds, we assign a label $j=1,2,\ldots$ to each Markov chain before its first sampling and define
\begin{equation}
\label{eq:certification-calls}
\rho:=\bar\eta+(|\I|-1)
\exp\!\left\{-\frac{\Delta_\star-2\epsilon}{w}\right\}, \quad r:=\frac{j(j+1)k(k+1)}
{\epsilon_{\mathrm{false}}} \min\big\{1,\widehat\delta(k)\big\}.
\end{equation}
At each integer incumbent $(\bar x,\bar\theta)$, we perform $\max\left\{
0,\left\lceil
\log_{1/\rho}\!\left(
r
\right)
\right\rceil
\right\}$ additional calls to Algorithm~\ref{alg:oracle}. Note that, because \(\bar{x}_c = \refpoint(\bar{x}_z)\) is fixed throughout these calls, each additional call uses one transition of \(P_{\bar{x}}\) (i.e., \(m=1\) in Algorithm~\ref{alg:oracle}). If no cuts are found violated in these calls, then we certify \(\widehat U = f(\bar{x})+\bar{\theta}\) as an upper bound.

We now analyze the probability of incorrectly certifying an upper bound \(\widehat{U}\) with \(\bar{\theta} < Q^\star(\bar{x})\). Recall, from the proof of Theorem~\ref{thm:finite-high-prob}, that
\begin{equation}
\label{eq:certification-detection}
P_x\bigl(I, I^\star(x)\bigr)
\ge 1-\rho, \quad \forall x \in \X, I\notin I^\star(x),
\end{equation}
where we choose a temperature parameter \(w \leq (\Delta_\star-2\epsilon)/
\log\!\bigl(2(|\mathcal I|-1)/(1-\bar\eta)\bigr)\) to ensure that \(\rho \leq (1+\bar\eta)/2 < 1\). 
Because the first call of Algorithm~\ref{alg:oracle} (in the ``cut'' phase) does not find a Benders' cut, the state sampled therein lies outside of \(I^\star(\bar{x})\). Consequently, inequality~\eqref{eq:certification-detection} designates that each additional call of Algorithm~\ref{alg:oracle} (in the ``certification'' phase) fails to find \(\bar{\theta} < Q^\star(\bar{x})\) with probability at most \(\rho\). Combining this conditional bound with the marginal bound of Corollary~\ref{cor:adaptive-transition-identification} gives, for each fixed $(j,k)$, 
\[
\mathbb P\!\left\{
\begin{array}{c}
\text{certification at } (j,k)\\[0.5em]
\text{with } Q^\star(\bar x)>\bar\theta
\end{array}
\right\}
\le
\min\big\{1,\widehat\delta(k)\big\}
\, \rho^{\max\{0,\lceil\log_{1/\rho}r\rceil\}}
\le
\frac{\epsilon_{\mathrm{false}}}
     {j(j+1)k(k+1)}.
\]
Note that each call of Algorithm~\ref{alg:oracle} increases $\kappa(\bar x_z)$ by one, so every certification is associated with a unique \((j,k)\) pair. Then, a union bound gives
\[
\mathbb P
\left\{
\begin{array}{c}
\text{certifying any } (\bar{x},\bar{\theta})\\[0.5em]
\text{with } Q^\star(\bar x)>\bar\theta
\end{array}
\right\}
\le\epsilon_{\mathrm{false}}\sum_{j=1}^{\infty}\sum_{k=1}^{\infty}
\frac{1}{j(j+1)k(k+1)}=\epsilon_{\mathrm{false}}.
\]
Finally, we notice that each certification concludes with either an acceptance of \(\widehat U\) or a violated cut, and does so within finite iterations because \(\log_{1/\rho}\!\left(
r
\right) < \infty\). This, together with the finiteness of cuts and the finite size of the branch-and-bound tree, guarantees that the proposed branch-and-cut method is terminated finitely. 
We summarize the proposed certification procedure in Algorithm~\ref{alg:outer-loop} and its probabilistic guarantee in the following theorem.
\begin{theorem}
\label{thm:outer-loop}
Algorithm~\ref{alg:outer-loop} terminates finitely 
almost surely. In addition, with probability at least
$1-\varepsilon_{\mathrm{false}}$, its returned solution
$(x^\star,\theta^\star)$ satisfies \(0 \leq \big(f(x^\star) + {\theta}^\star\big) - F^\star \leq \varepsilon_{\text{opt}}\).
\end{theorem}

\begin{algorithm}[ht]
\caption{Branch-and-Cut with Certification}
\label{alg:outer-loop}
\begin{algorithmic}[1]
\Require Cut pool $\mathcal C$; maps $\state$, $\refpoint$, $\kappa$;
budget $\epsilon_{\mathrm{false}}$;
temperature parameter $w$; 
\Ensure Solution $x^\star$
\State Solve the root-node relaxation using Algorithm~\ref{alg:MH-SGD2} 
and warm-start \(\mathcal C \gets \big\{\big(I_{n+1}, \pi^\star(x_n, \xi^{(I_{n+1})})\big)\big\}_{n}\)
\State Set $j\gets0$, $\widehat U \gets +\infty$, \(L(\mathcal C) \gets \text{root-node optimal value}\)
\While{$\widehat U-L(\mathcal C) > \epsilon_{\mathrm{opt}}$} solve~\eqref{main} using branch-and-bound
    \For{each integer incumbent $(\bar x, \bar \theta)$ encountered} \Comment{\texttt{lazy callback}}
        \If{the chain for $x_z$ has no label} assign label $j$ to this chain and $j\gets j+1$
        \EndIf
         \State Call Algorithm~\ref{alg:oracle} at $(\bar x, \bar \theta)$, set $k \gets \kappa(\bar{x}_z)$, and set $\rho,r$ through~\eqref{eq:certification-calls} \Comment{``cut'' phase}
        \While{true}
           
            \If{a violated cut is returned}
                add it to $\mathcal C$ and \textbf{break}
            \ElsIf{$r\le1$} \Comment{``certification'' phase}
                \If{$f(\bar x)+\bar \theta < \widehat U$}
    update $\widehat U\gets f(\bar x)+\bar\theta$ and
                    $x^\star \gets \bar{x}$
                \EndIf
                \State Certify \(\widehat {U}\) and \textbf{break}
             \Else \hspace{0.2em} call Algorithm~\ref{alg:oracle} at $(\bar x, \bar \theta)$ and update $r\gets\rho r$
             \EndIf
        \EndWhile
    \EndFor
    \State Update $L(\mathcal C) \gets$ the global lower bound of the branch-and-bound tree
\EndWhile
\State \Return $x^\star$
\end{algorithmic}
\end{algorithm}

\section{Numerical Experiments} \label{sec:experiments}

We evaluate the effectiveness and scalability of the proposed Algorithms~\ref{alg:MH-SGD2} and~\ref{alg:outer-loop} numerically. In particular, Section~\ref{subsec:continuous-first-stage} applies Algorithm~\ref{alg:MH-SGD2} to solving~\eqref{main} models with continuous \(x\), and Section~\ref{subsec:integer-first-stage} applies Algorithm~\ref{alg:outer-loop} to address mixed-integer \(x\). In both cases, we compare the proposed algorithm to C\&CG and ADR under various instance settings. Unless otherwise stated, each configuration is repeated over 20 independent replications, and the reported curves show the median trajectory together with the 25--75\% interquartile band. All experiments were conducted in Python on a MacBook Pro equipped with an Apple M3 Pro Chip and 18 GB RAM. All instances and code are available online at \url{https://github.com/xuqy2002/SoftSeparationARO}.

As benchmarks, we implement C\&CG of~\cite{ZengZhao2013} to solve the instances to global optimum and the ADR approximation of~\cite{ben2004adjustable}. To compare fairly with C\&CG, in the case of an ellipsoidal uncertainty set, we follow Electronic Companion A-4.1 in~\cite{ZengZhao2013} to solve its subproblems based on the KKT condition of the transportation linear program and the big-\(M\) constants suggested therein; in the case of scenario uncertainty set, we compute \(\max_{\xi \in \Xi}Q(x, \xi)\) through enumeration in C\&CG, which we find significantly faster than computing it through the KKT condition plus an SOS1 constraint among the \(\xi\)-scenarios; and in the case of a budget uncertainty set, we follow Electronic Companion A-4.2 in~\cite{ZengZhao2013} to solve its subproblems based on strong duality and the big-\(M\) constants in~\cite{gabrel2014robust}, as~\cite{ZengZhao2013} suggest. Other implementation details can be found in Appendix~\ref{app:exp}.

\subsection{Continuous Here-and-Now Decisions}
\label{subsec:continuous-first-stage}

We consider the adaptive robust location-transportation problem as in~\cite{ZengZhao2013}. The here-and-now decision allocates continuous capacity across $m$ supply locations at a linear investment cost. After the uncertain demand vector $\xi \in \mathbb R^n$ is realized, a wait-and-see transportation linear program ships goods to $n$ demand locations subject to the installed capacities, with any unfulfilled demand served by a high-cost emergency source. The demand at location $j \in [n]$ is modeled by
\(
d_j(\xi)=\overline d_j+\widehat d_j\xi_j
\), where the primitive uncertainty \(\xi\) belongs with an uncertainty set.

We generate random instances as follows. We set $m=n$ (i.e., equal number of supply and demand nodes), and sample the integer-valued capacity-investment costs, transportation costs, emergency-supply costs, and nominal demands $\overline d_j$ uniformly and independently from the intervals $[10,30]$, $[20,40]$, $[100,200]$, and $[200,300]$, respectively. The demand sensitivity $\widehat d_j$ is a uniform random fraction $[0.1,0.5]$ of $\overline d_j$. We consider three types of uncertainty sets: firstly, an ellipsoidal uncertainty set
\begin{equation}
\Xi
:=
\Big\{
\bar\xi+R\operatorname{diag}(a_1,\ldots,a_n)u: \ 
\|u\|_2\le1
\Big\}, \quad R^TR=\textrm{Id}, \tag{Ellipsoidal} \label{ellipsoidal-uncertainty}
\end{equation}
where $\bar\xi=0.5\cdot\mathbf1$, the semi-axis lengths $a_j$ are sampled
independently from $\mathrm{Unif}[0.10,0.50]$, and $R$ is an
independent random rotation matrix drawn from the Haar
distribution on the rotation group $\mathrm{SO}(n)$; and secondly, a scenario uncertainty set consisting of finite scenarios of \(\xi\), which are drawn uniformly from the boundary of the ellipsoidal uncertainty set. Consequently, these two uncertainty sets produce similar \(F^\star\). Thirdly, a budget uncertainty set is defined through
\begin{equation}
\Xi
:=
\Big\{
\xi \in [\mathbf{0}, \mathbf{1}]: \ \mathbf{1}^{\top}\xi \leq n\Gamma
\Big\}, \tag{Budget} \label{budget-uncertainty}
\end{equation}
where \(n\Gamma\) is the so-called ``budget of uncertainty.''

\subsubsection{Hyperparameter Configuration} \label{subsec:cvg-analysis}
We first examine the convergence of Algorithm~\ref{alg:MH-SGD2} under various choices of \(|\I|\) (the size of the scenario uncertainty set), \(w\) (temperature parameter for the Mellowmax approximation), and \(\nu\) (the proposal distribution on \(\I\)). 
To this end, we evaluate the robust objective function value of each decision incumbent using exact separation, and report the evolution of the ensuing relative optimality gap when compared to the true optimal value. In all evaluations, we use the linear average as the decision incumbent.

\begin{figure}[!t]
    \centering
    \includegraphics[width=0.78\textwidth]{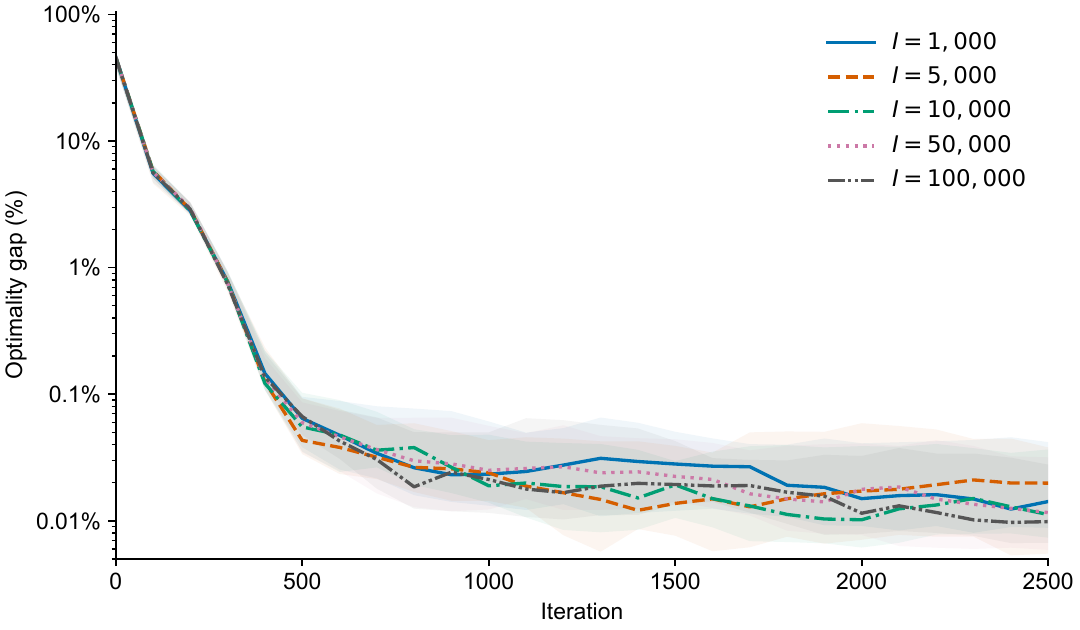}
    \caption{Convergence of Algorithm~\ref{alg:MH-SGD2} under various scenario sizes $|\I|$ with uniform proposal and 
    $w=0.001$. Each selected solution is reevaluated using the exact
    separation. The solid curves show medians over 20 replications, and the
    shaded regions show the 25--75\% interquartile bands.}
    \label{fig:task2-vary-I}
\end{figure}

\paragraph{Convergence under varying scenario sizes.}
We first examine the convergence of Algorithm~\ref{alg:MH-SGD2} with $|\I| \in \{1000,5000,10000,50000,100000\}$, while fixing 
$w=0.001$ and \(\nu\) to be uniform on \(\I\). Figure~\ref{fig:task2-vary-I} reports the
convergence trajectories. We observe that Algorithm~\ref{alg:MH-SGD2} displays broadly similar convergence behavior across all tested
scenario sizes. In each trajectory, the relative optimality gap decreases sharply in
the initial iterations and then stabilizes around the interval \([0.01\%, 0.1\%]\).
In particular, increasing $|\I|$ from $1{,}000$ to $100{,}000$ does not lead to a systematic deterioration in either convergence speed or the attained gap over the iteration horizon considered. This demonstrates that the convergence of Algorithm~\ref{alg:MH-SGD2} has a mild reliance on the scenario size.

\begin{figure}[!t]
    \centering
    \includegraphics[width=0.78\textwidth]{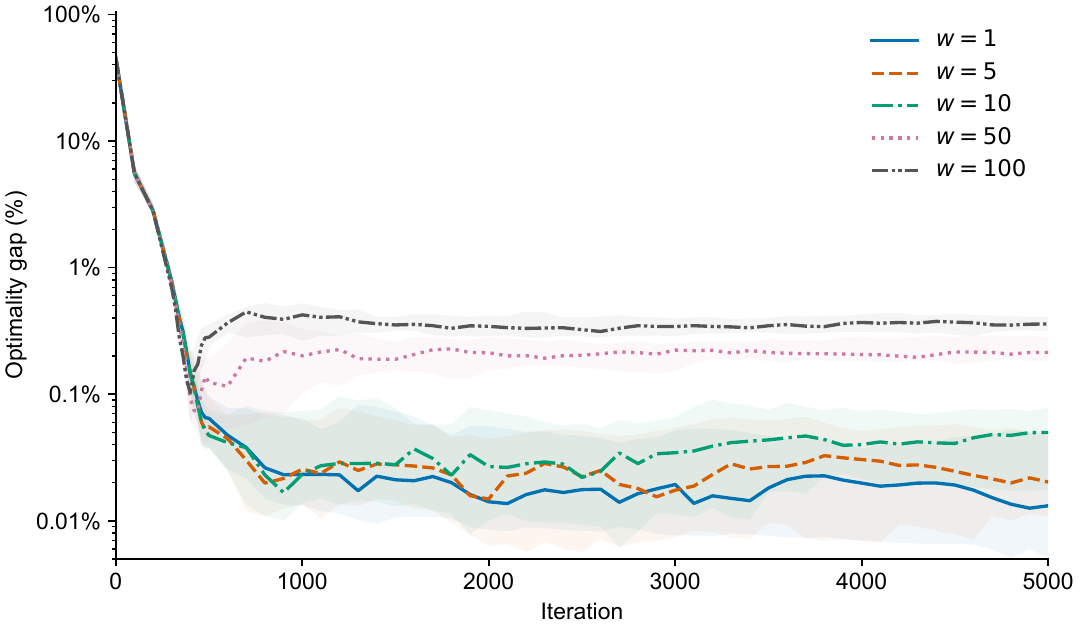}
    \caption{Convergence of Algorithm~\ref{alg:MH-SGD2} under various $w$ with uniform proposal and $|\I|=1{,}000$.
    Each selected solution is reevaluated using the exact separation. The
    solid curves show medians over 20 replications, and the shaded regions show the
    25--75\% interquartile bands.}
    \label{fig:task2-w}
\end{figure}

\paragraph{Sensitivity to the temperature parameter \(w\).}
We next study the convergence with $w$ varying among \(\{1,5,10,50,100\}\), while fixing \(\nu\) to be uniform on \(\I\) and $|\I|=1{,}000$. From Figure~\ref{fig:task2-w}, we observe that the optimality gap displays a phase change, with \(w \in \{50, 100\}\) resulting in significant gaps between 0.1\% and 1\%, and lower temperature such as \(w \leq 10\) producing much smaller gaps in \([0.01\%, 0.1\%]\). In addition, the gap generally improves as the \(w\)-value decreases. This reflects the Mellowmax approximation error in Lemma~\ref{lem:mellowmax-approx}. We note that the error gap therein is uniform in \(x\) and so it makes sense that Figure~\ref{fig:task2-w} displays a much higher accuracy around optimum. In the remaining experiments, we shall fix \(w = 0.001\) for an accurate approximation of~\eqref{main}.


\begin{figure}[!t]
    \centering

    \includegraphics[width=0.48\textwidth]
    {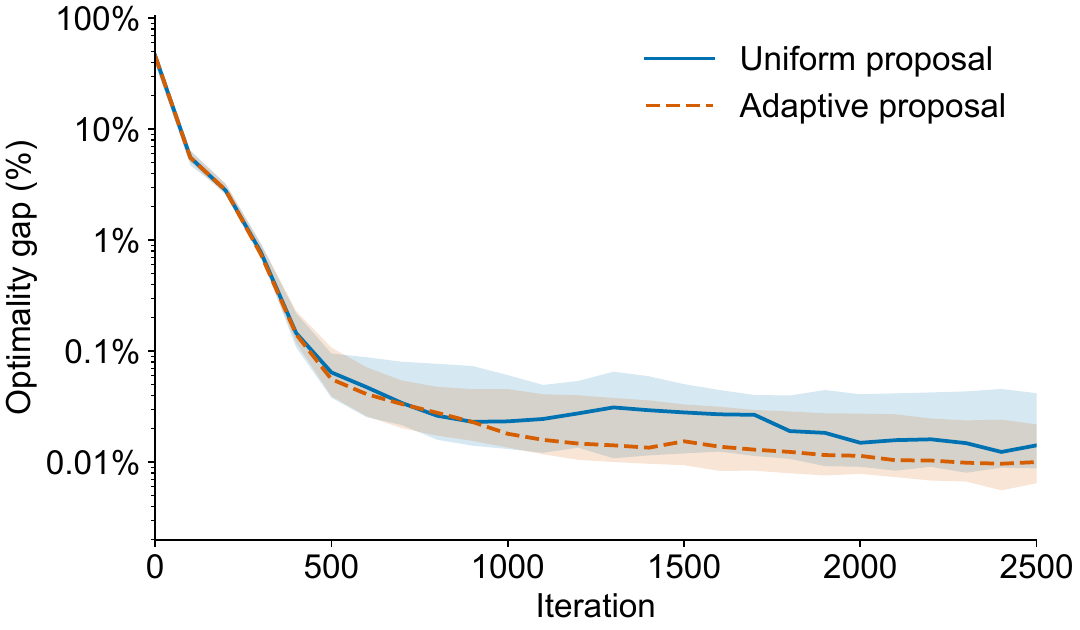}
    \hfill
    \includegraphics[width=0.48\textwidth]
    {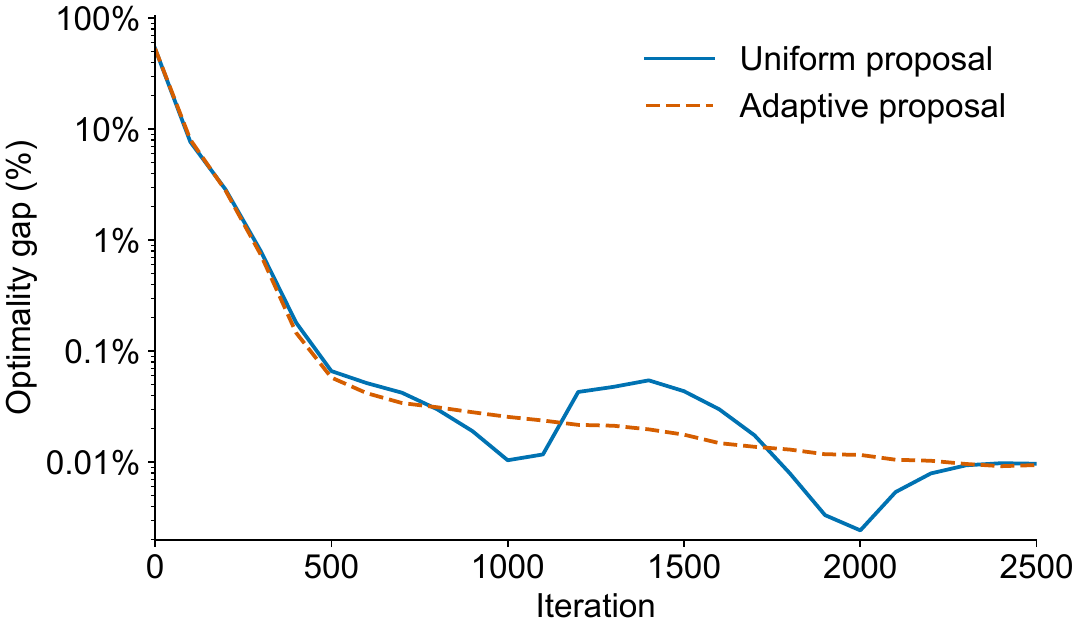}

    \caption{Adaptive proposal versus uniform proposal with
    $w=0.001$ and $|\I|=1{,}000$.
    Each selected solution is reevaluated using the exact separation.
    The left panel reports the median trajectories over 20 independent replications,
    with the shaded regions indicating the 25--75\% interquartile bands.
    The right panel shows the convergence trajectories from one individual replication.}
    \label{fig:task2-adaptive-vs-uniform}
\end{figure}

\paragraph{Adaptive proposal versus uniform proposal.}
\label{adaptive_exp}
Finally, we examine the proposal distribution \(\nu\). In Figure~\ref{fig:task2-adaptive-vs-uniform}--right, we depict the convergence of Algorithm~\ref{alg:MH-SGD2} with the uniform proposal, \(w = 0.001\), and \(|\I|=1,000\). The solid curve displays a decaying trend in the optimality gap but also shows rebounds in later iterations (note that the gap is reported in a logarithmic scale). Intuitively, we can attribute these rebounds to the uniform proposal because, in later iterations, a uniform \(\nu\) ignorant of \(Q(x, \xi)\) may propose less adversarial states which, if accepted, move the decision incumbents away from optimum.

Observing this, we propose an adaptive proposal distribution to account for \(Q(x, \xi)\). To this end, we use the subgradient $-U^\top\pi^\star(x,i)\in \partial_\xi Q(x,\xi^{(i)})$ at the current state \(i \in \I\) to tilt new scenario generation toward a higher \(Q(x, \xi)\) value. Specifically, for the ellipsoidal uncertainty set defined in~\eqref{ellipsoidal-uncertainty}, we draw the primitive random vector \(u\) from the von Mises--Fisher (vMF) distribution on $\mathbb S^{n-1}$, which has a density $f_{\mathrm{vMF}}(u;\mu,\kappa) \propto\exp\!\left(\kappa\mu^\top u\right)$ with respect to the uniform spherical measure. Here, \(\mu := \frac{A^\top (-U^\top\pi^\star(x,i))}{\|A^\top (-U^\top\pi^\star(x,i))\|}\) with \(A := R\text{diag}(a_1, \ldots, a_n)\) is the mean direction and \(\kappa \geq 0\) controls the concentration around \(\mu\). For example, \(\kappa = 0\) reduces vMF to the uniform distribution on \(\mathbb{S}^{n-1}\) and a larger \(\kappa\) places more probability mass near \(\mu\). Then, we propose \(\bar\xi + Au\). For the scenario uncertainty set, we adapt the proposal distribution to be 
\(
q(j|i) \propto \frac{\exp\!\left(\kappa\mu^\top u^{(j)}\right)}
{|\det A|\,\|A^{-\top}u^{(j)}\|}
\). For the budget uncertainty set, the adaptive proposal admits an efficient sampling procedure through dynamic programming, which we detail in Appendix~\ref{app:adaptive-proposal-sampling}.

We compare the adaptive proposal with \(\kappa = 10\) and the uniform proposal in Figure~\ref{fig:task2-adaptive-vs-uniform}. We observe that the adaptive proposal produces a more stable convergence (see the right panel) and converges to a smaller optimality gap (see the left panel). This makes sense because vMF tilts the proposal toward adversarial states and reduces ineffective transitions. It can also be explained by the Poisson resolvent constant \(C_*\). Indeed, the adaptive proposal is closer to the ideal sampling distribution \(p_x\) than the uniform proposal, leading to a smaller \(C_*\) value (see Example~\ref{ex:ratio-aligned}) and a faster convergence (see Theorems~\ref{thm:continuous_convergence} and~\ref{thm:continuous-state-convergence}).

\subsubsection{Computing Time}

We report the computing times (all in wall clock seconds) of solving~\eqref{main} instances through Algorithm~\ref{alg:MH-SGD2} using \(w = 0.001\) and the adaptive proposal distribution.

\begin{table}[!htbp]
\centering
\small
\caption{Runtime comparison in instances with \emph{continuous} here-and-now decisions and ellipsoidal uncertainty sets. Entries are
mean $\pm$ standard deviation in seconds over five independent replications.
}
\label{tab:continuous-ellipsoidal-runtime}
\begin{tabular}{lcccc}
\toprule
\(n\) &  Algorithm~\ref{alg:MH-SGD2} & C\&CG & ADR \\
\midrule
$3$
    & $0.13 \pm 0.06$
    & $0.045 \pm 0.016$
    & $\mathbf{0.029 \pm 0.004}$ \\ 
$10$
    & $\mathbf{0.17 \pm 0.07}$
    & $2.73 \pm 0.43$
    & $1.33 \pm 0.51$ \\
$30$
    & $\mathbf{0.32 \pm 0.10}$
    & $772.78 \pm 376.24$
    & $138.52 \pm 11.39$ \\
$50$
    & $\mathbf{1.07 \pm 0.25} $
    & $>1800$
    & $623.83 \pm 107.04$ \\
$70$
    & $\mathbf{2.07 \pm 0.55 } $
    & $>1800$
    & $>1800$ \\
$100$
   & $\mathbf{3.56 \pm 0.65 } $
    & $>1800$
    & $>1800$ \\
\bottomrule
\end{tabular}
\end{table}

\paragraph{Ellipsoidal uncertainty set.}
We first consider the ellipsoidal uncertainty set~\eqref{ellipsoidal-uncertainty} with \(n\) varying from 3 to 100 and report the computing times in Table~\ref{tab:continuous-ellipsoidal-runtime}. For Algorithm~\ref{alg:MH-SGD2} and C\&CG, we report the time required to attain a relative optimality gap below $0.1\%$ with respect to the true optimal value. For ADR, we report the time required to solve the conservative approximation. We use a time limit of half an hour. We report ``\(>1800\)'' if the computing times in all replications exceed the limit, or otherwise truncate the beyond-limit computing times at 1,800s and then report mean $\pm$ standard deviation in seconds.

From Table~\ref{tab:continuous-ellipsoidal-runtime}, we observe that Algorithm~\ref{alg:MH-SGD2} shows better scalability than C\&CG and ADR as the problem dimension \(n\) increases. In particular, the computing times of Algorithm~\ref{alg:MH-SGD2} grow mildly and remain stable as \(n\) increases. On the contrary, those of C\&CG and ADR exceed the time limit in all replications once \(n\) reaches \(70\). This demonstrates the scalability of the proposed soft separation approach with problem dimensions.


\begin{table}[!htbp]
\centering
\caption{Runtime comparison in instances with \emph{continuous} here-and-now decisions and scenario uncertainty sets. Entries are
mean $\pm$ standard deviation in seconds over five independent replications. }
\label{tab:scenario-based-runtime}
\resizebox{\textwidth}{!}{
\begin{tabular}{lccc@{\hspace{0.7cm}}ccc@{\hspace{0.7cm}}ccc}
\toprule
& \multicolumn{3}{c}{$n=3$}
& \multicolumn{3}{c}{$n=30$}
& \multicolumn{3}{c}{$n=100$}
\\
\cmidrule(lr){2-4}
\cmidrule(lr){5-7}
\cmidrule(lr){8-10}
$|\I|$
& Algorithm~\ref{alg:MH-SGD2} & C\&CG & ADR
& Algorithm~\ref{alg:MH-SGD2} & C\&CG & ADR
& Algorithm~\ref{alg:MH-SGD2} & C\&CG & ADR
\\
\midrule
$1{,}000$
& $0.15 \pm 0.04$
& $0.17 \pm 0.03$
& $\mathbf{0.14 \pm 0.08}$
& $\mathbf{0.74 \pm 0.19}$
& $1.07 \pm 0.08$
& $>1800$
& $\mathbf{6.71 \pm 2.25}$
& $9.54 \pm 0.76  $
& $>1800$
\\

$5{,}000$
& $\mathbf{0.17 \pm 0.06}$
& $0.81 \pm 0.32$
& $2.52 \pm 0.14$
& $\mathbf{0.57 \pm 0.23}$
& $5.03 \pm 0.43$
& $>1800$
& $\mathbf{7.87 \pm 0.41}$
& $48.94 \pm 3.45 $
& $>1800$
\\

$10{,}000$
& $\mathbf{0.21 \pm 0.08}$
& $1.46 \pm 0.52$
& $4.85 \pm 0.80$
& $\mathbf{0.67 \pm 0.10 }$
& $9.92 \pm 0.61 $
& $>1800$
& $\mathbf{8.01 \pm 0.93}$
& $94.45 \pm 8.39$
& $>1800$
\\

$50{,}000$
& $\mathbf{0.37 \pm 0.14}$
& $5.46 \pm 1.93$
& $13.42 \pm 0.45$
& $\mathbf{1.02 \pm 0.32 }$
& $49.44 \pm 3.77 $
& $>1800$
& $\mathbf{7.72 \pm 0.27}$
& $ 433.20 \pm 32.83 $
& $>1800$
\\

$100{,}000$
& $\mathbf{0.58 \pm 0.21}$
& $12.58 \pm 5.35$
& $54.43 \pm 11.83$
& $\mathbf{1.26 \pm 0.25 }$
& $101.95 \pm 5.93 $
& $>1800$
& $\mathbf{8.29 \pm 0.29}$
& $844.09 \pm 94.49$
& $>1800$
\\
\bottomrule
\end{tabular}
}
\end{table}

\paragraph{Scenario uncertainty set.}
We next consider the scenario uncertainty set $\Xi := \{\xi^1,\ldots,\xi^{|\I|}\}$ with \(|\I|\) varying from 1,000 to 100,000 and \(n\) varying from 3 to 100. From Table~\ref{tab:scenario-based-runtime}, we observe that the runtime of Algorithm~\ref{alg:MH-SGD2} increases only moderately in \(|\I|\), reflecting the uniform error bound in Lemma~\ref{lem:mellowmax-approx} and the convergence rate in Theorem~\ref{thm:continuous_convergence}. In addition, the increase of Algorithm~\ref{alg:MH-SGD2} runtime in \(n\) is also significantly milder than those of C\&CG and ADR. In particular, ADR cannot solve any instance/replication within the time limit once \(n\) reaches 30. This observation supports the computational advantage of the soft separator: it avoids the scenario-wise scaling bottleneck, which limits C\&CG and ADR in higher-dimensional instances.




\begin{table*}[!t]
\centering
\small
\setlength{\tabcolsep}{3.5pt}
\caption{Runtime comparison for instances with \emph{continuous} here-and-now
decisions and budget uncertainty sets. Times are reported as mean $\pm$
standard deviation in seconds over five independent replications.}
\label{tab:continuous-budget-runtime}

\resizebox{\textwidth}{!}{%
\begin{tabular}{
ccccc
@{\hspace{2em}}
ccccc
}
\cmidrule(lr){1-5}
\cmidrule(lr){6-10}
$n$ & $\Gamma$ & Algorithm~\ref{alg:MH-SGD2} & C\&CG & ADR
&
$n$ & $\Gamma$ & Algorithm~\ref{alg:MH-SGD2} & C\&CG & ADR
\\
\midrule

30 & 0.1 & $1.46 \pm 0.25$
          & $\mathbf{1.13 \pm 0.61}$
          & $3.94 \pm 0.40$
&
70 & 0.1 & $\mathbf{17.79 \pm 2.35}$
          & $>1800$
          & $473.61 \pm 65.69$
\\

& 0.3 & $\mathbf{1.56 \pm 0.37}$
      & $18.54 \pm 16.68$
      & $7.78 \pm 1.94$
&
& 0.3 & $\mathbf{10.88 \pm 0.55}$
      & $>1800$
      & $768.04 \pm 188.41$
\\

& 0.5 & $\mathbf{2.33 \pm 1.15}$
      & $10.54 \pm 10.25$
      & $13.67 \pm 5.98$
&
& 0.5 & $\mathbf{8.02 \pm 2.35}$
      & $>1800$
      & $881.43 \pm 209.58$
\\

& 0.7 & $\mathbf{2.22 \pm 0.36}$
      & $2.74 \pm 2.92$
      & $14.31 \pm 6.08$
&
& 0.7 & $\mathbf{10.80 \pm 1.37}$
      & $>1800$
      & $1027.86 \pm 318.44$
\\

& 0.9 & $1.90 \pm 0.55$
      & $\mathbf{0.20 \pm 0.04}$
      & $16.63 \pm 4.97$
&
& 0.9 & $\mathbf{23.49 \pm 5.42}$
      & $63.93 \pm 54.06$
      & $1319.22 \pm 573.24$
\\

\midrule

50 & 0.1 & $\mathbf{8.51 \pm 1.54}$
          & $37.33 \pm 16.57$
          & $74.58 \pm 15.31$
&
100 & 0.1 & $\mathbf{38.79 \pm 7.96}$
          & $>1800$
          & $>1800$
\\

& 0.3 & $\mathbf{6.04 \pm 0.76}$
      & $>1800$
      & $247.19 \pm 145.87$
&
& 0.3 & $\mathbf{17.62 \pm 1.09}$
      & $>1800$
      & $>1800$
\\

& 0.5 & $\mathbf{9.61 \pm 0.73}$
      & $>1800$
      & $175.90 \pm 24.68$
&
& 0.5 & $\mathbf{27.39 \pm 3.40}$
      & $>1800$
      & $>1800$
\\

& 0.7 & $\mathbf{4.88 \pm 0.91}$
      & $138.00 \pm 127.46$
      & $152.43 \pm 52.32$
&
& 0.7 & $\mathbf{16.02 \pm 4.01}$
      & $>1800$
      & $>1800$
\\

& 0.9 & $4.70 \pm 0.87$
      & $\mathbf{2.51 \pm 1.21}$
      & $124.11 \pm 45.91$
&
& 0.9 & $\mathbf{38.86 \pm 2.74}$
      & $>1800$
      & $>1800$
\\

\bottomrule
\end{tabular}%
}
\end{table*}

\paragraph{Budget uncertainty set.} Finally, we consider the budget uncertainty set~\eqref{budget-uncertainty} with \(n\) varying from \(30\) to \(100\) and \(\Gamma\) varying between \(0.1\) and \(0.9\). From Table~\ref{tab:continuous-budget-runtime}, we observe that C\&CG performs the best when \(n\) is small and \(\Gamma\) is near-binary (e.g., when \(n = 30, \Gamma = 0.9\)). As \(n\) increases and \(\Gamma\) moves away from binary values, Algorithm~\ref{alg:MH-SGD2} outperforms C\&CG and ADR. For example, neither C\&CG nor ADR solves any instance or replication within the time limit once \(n\) reaches 100, while Algorithm~\ref{alg:MH-SGD2} solves all these instances. Similarly, the Algorithm~\ref{alg:MH-SGD2} runtime remains more stable than that of C\&CG as \(\Gamma\) varies. For example, the C\&CG runtime with \(\Gamma = 0.5\) (when the budget uncertainty set has more extreme points) is significantly longer than that with \(\Gamma = 0.1, 0.9\) (when the number of extreme points is less).

\subsection{Mixed-integer Here-and-Now Decisions}
\label{subsec:integer-first-stage}

We now test instances with integrality restrictions on the here-and-now decision variables. To this end, we add facility-opening decisions \(y_i \in \{0, 1\}\), each incurring a setup cost $f_i$, to the adaptive
robust location-transportation problem. We further impose constraints \(\text{capacity}_i \leq K_i y_i\) for all \(i \in [m]\), which ensures that capacities can only be installed at opened facilities. We follow the same approach to generate instances as in Section~\ref{subsec:continuous-first-stage} and we sample \(f_i\) and \(K_i\) uniformly from the integer ranges  \([300,400]\) and \([800,1200]\), respectively.

We report the computing times (all in wall clock seconds) of solving~\eqref{main} instances through Algorithm~\ref{alg:outer-loop} using \(w = 1\) and the uniform proposal distribution. When implementing Algorithm~\ref{alg:outer-loop}, we compute the true objective value of the solution \({x}^\star\) returned by the branch-and-cut method. Across all tested instances,
each ensuing objective value is within $0.1\%$ of the
true optimal value.

\begin{table}[!t]
\centering
\small
\caption{Runtime comparison in instances with \emph{mixed-integer} here-and-now decisions and ellipsoidal uncertainty set. Entries are
mean $\pm$ standard deviation in seconds over five independent replications.}
\label{tab:discrete-ellipsoidal-runtime}
\begin{tabular}{lccc}
\toprule
\(n\) & Algorithm~\ref{alg:outer-loop} & C\&CG & ADR \\
\midrule
$3$
& $0.17 \pm 0.02$
& $\mathbf{0.07 \pm 0.04}$
& $0.23 \pm 0.03$ \\

$10$
& $\mathbf{0.88 \pm 0.35}$
& $1.46 \pm 0.70$
& $41.95 \pm 5.16$ \\

$30$
& $\mathbf{3.05\pm0.37}$
& $1390.24 \pm 572.77$
& $>1800$ \\
\bottomrule
\end{tabular}
\end{table}

\paragraph{Ellipsoidal uncertainty set.} We report the computing times with ellipsoidal uncertainty sets in Table~\ref{tab:discrete-ellipsoidal-runtime}, where \(n \in \{3, 10, 30\}\). From this table, we observe that Algorithm~\ref{alg:outer-loop} demonstrates better scalability than C\&CG and ADR as \(n\) increases. For example, the average runtime of Algorithm~\ref{alg:outer-loop} becomes at least an order of magnitude shorter than those of C\&CG and ADR when \(n\) reaches 30. This demonstrates that, although we need to maintain multiple Markov chains in a branch-and-cut tree, the soft separation still manages to identify adversarial scenarios at a relatively lower computational burden.

\begin{table}[!t]
\centering
\caption{Runtime comparison in instances with \emph{mixed-integer} here-and-now decisions and
scenario uncertainty set. Entries are mean $\pm$ standard deviation in seconds over
five independent replications.}
\label{tab:integer-scenario-runtime}
\resizebox{\textwidth}{!}{
\begin{tabular}{lccc@{\hspace{0.05cm}}ccc@{\hspace{0.05cm}}ccc}
\toprule
& \multicolumn{3}{c}{$n=3$}
& \multicolumn{3}{c}{$n=10$}
& \multicolumn{3}{c}{$n=30$}
\\
\cmidrule(lr){2-4}
\cmidrule(lr){5-7}
\cmidrule(lr){8-10}
$|\I|$
& Algorithm~\ref{alg:outer-loop} & C\&CG & ADR
& Algorithm~\ref{alg:outer-loop} & C\&CG & ADR
& Algorithm~\ref{alg:outer-loop} & C\&CG & ADR
\\
\midrule
$1{,}000$
& $0.32 \pm 0.12$
& $\mathbf{0.18 \pm 0.08}$
& $1.38 \pm 0.55$
& $\mathbf{1.47 \pm 0.45}$
& $1.74 \pm 0.71$
& $373.22 \pm 90.40$
& $3.60 \pm 0.25$
& $\mathbf{2.85\pm0.48}$
& $>1800$
\\

$5{,}000$
& $\mathbf{0.37 \pm 0.15}$
& $0.64 \pm 0.31$
& $8.88 \pm 2.37$
& $\mathbf{1.44 \pm 0.48}$
& $5.01 \pm 1.54$
& $1397.60 \pm 594.33$
& $\mathbf{3.64 \pm 0.51}$
& $12.63\pm0.65$
& $>1800$
\\

$10{,}000$
& $\mathbf{0.35 \pm 0.18}$
& $1.37 \pm 0.69$
& $25.72 \pm 7.85$
& $\mathbf{1.37 \pm 0.44}$
& $11.75 \pm 6.75$
& $>1800$
& $\mathbf{3.58 \pm 0.38}$
& $24.88\pm1.41$
& $>1800$
\\

$50{,}000$
& $\mathbf{0.33 \pm 0.15}$
& $6.58 \pm 3.44$
& $118.33 \pm 28.03$
& $\mathbf{2.61 \pm 0.60}$
& $64.59\pm 28.64$
& $>1800$
& $\mathbf{3.84 \pm 0.42}$
& $151.00\pm38.73$
& $>1800$
\\

$100{,}000$
& $\mathbf{0.42 \pm 0.17}$
& $14.49 \pm 6.41$
& $256.19 \pm 99.19$
& $\mathbf{1.98 \pm 0.56}$
& $82.62 \pm 17.53$
& $>1800$
& $\mathbf{3.82\pm0.45 }$
& $238.43\pm12.47 $
& $>1800$
\\
\bottomrule
\end{tabular}
}
\end{table}

\paragraph{Scenario uncertainty set.} 
We report the computing times with scenario uncertainty sets in Table~\ref{tab:integer-scenario-runtime}, where \(n \in \{3, 10, 30\}\) and \(|\I|\) varies from \(1,000\) to \(100,000\). This table
shows that Algorithm~\ref{alg:outer-loop} scales favorably
with both \(n\) and \(|\I|\). Indeed, for each fixed \(n\), the runtime of Algorithm~\ref{alg:outer-loop} remains nearly constant as \(|\I|\) increases. In addition, when \(n = 30\), the ADR runtime exceeds the time limit in all replications, while Algorithm~\ref{alg:outer-loop} is able to solve all instances within the time limit.
\begin{table*}[!t]
\centering
\small
\setlength{\tabcolsep}{3.5pt}
\caption{Runtime comparison for instances with mixed-integer here-and-now
decisions and budget uncertainty. Times are reported as mean $\pm$
standard deviation in seconds over five independent replications.}
\label{tab:integer-budget-runtime}

\resizebox{\textwidth}{!}{%
\begin{tabular}{
ccccc
@{\hspace{2em}}
ccccc
}
\cmidrule(lr){1-5}
\cmidrule(lr){6-10}
$n$ & $\Gamma$ & Algorithm~\ref{alg:outer-loop} & C\&CG & ADR
&
$n$ & $\Gamma$ & Algorithm~\ref{alg:outer-loop} & C\&CG & ADR
\\
\midrule

30 & 0.1 & $4.73 \pm 1.25$
          & $\mathbf{0.28 \pm 0.05}$
          & $24.84 \pm 12.95$
&
70 & 0.1 & $\mathbf{38.97 \pm 6.40}$
          & $889.39 \pm 313.63$
          & $>1800$
\\

& 0.3 & $\mathbf{4.41 \pm 0.84}$
      & $5.45 \pm 3.19$
      & $59.31 \pm 25.07$
&
& 0.3 & $\mathbf{40.35 \pm 8.64}$
      & $>1800$
      & $>1800$
\\

& 0.5 & $\mathbf{3.73 \pm 0.31}$
      & $5.09 \pm 3.41$
      & $92.74 \pm 20.12$
&
& 0.5 & $\mathbf{35.39 \pm 2.47}$
      & $>1800$
      & $>1800$
\\

& 0.7 & $3.86 \pm 0.40$
      & $\mathbf{0.93 \pm 1.23}$
      & $113.54 \pm 38.49$
&
& 0.7 & $\mathbf{36.57 \pm 3.00}$
      & $449.30 \pm 473.53$
      & $>1800$
\\

& 0.9 & $4.03 \pm 0.65$
      & $\mathbf{0.19 \pm 0.03}$
      & $104.89 \pm 31.09$
&
& 0.9 & $\mathbf{43.07 \pm 9.19}$
      & $254.14 \pm 237.96$
      & $>1800$
\\

\midrule

50 & 0.1 & $24.58 \pm 13.12$
          & $\mathbf{17.37 \pm 7.83}$
          & $1364.19 \pm 631.61$
&
100 & 0.1 & $\mathbf{42.96 \pm 5.80}$
           & $>1800$
           & $>1800$
\\

& 0.3 & $\mathbf{22.72 \pm 15.61}$
      & $1515.93 \pm 686.48$
      & $1469.35 \pm 747.72$
&
& 0.3 & $\mathbf{42.93\pm 3.89}$
      & $>1800$
      & $>1800$
\\

& 0.5 & $\mathbf{19.10 \pm 4.07}$
      & $1307.20 \pm 775.33$
      & $>1800$
&
& 0.5 & $\mathbf{46.88 \pm 6.73}$
      & $>1800$
      & $>1800$
\\

& 0.7 & $\mathbf{17.01 \pm 7.43}$
      & $88.60 \pm 100.68$
      & $>1800$
&
& 0.7 & $\mathbf{46.26 \pm 4.87}$
      & $1469.80 \pm 795.10$
      & $>1800$
\\

& 0.9 & $14.76 \pm 2.50$
      & $\mathbf{1.32 \pm 0.84}$
      & $1316.66 \pm 667.30$
&
& 0.9 & $\mathbf{45.15 \pm 6.00}$
      & $1076.80 \pm 548.60$
      & $>1800$
\\

\bottomrule
\end{tabular}%
}
\end{table*}
\paragraph{Budget uncertainty set.} Finally, we report the computing times with the budget uncertainty set in Table~\ref{tab:integer-budget-runtime}, where \(n\) varies from 30 to 100 and \(\Gamma\) varies between 0.1 and 0.9. From this table, we observe that C\&CG performs the best when \(n\) is small and \(\Gamma\) is close to binary (e.g., when \(n=30, \Gamma = 0.1\)), and Algorithm~\ref{alg:outer-loop} outperforms C\&CG and ADR as \(n\) increases, especially when \(n = 100\). In addition, Algorithm~\ref{alg:outer-loop} shows better scalability with \(\Gamma\) than C\&CG and ADR. This confirms our earlier observation from Table~\ref{tab:continuous-budget-runtime}.  

\section{Conclusion} \label{sec:conclusion}
This paper studied alternative algorithms for solving adaptive robust optimization (\ref{main}) based on soft, rather than exact, separation. Soft separation adapts a time-inhomogeneous Markov chain to optimization iterates and generates high-confidence adversarial scenarios for~\ref{main}. We applied the soft separation to solve~\eqref{main} with both continuous and mixed-integer here-and-now decisions. In the continuous setting, we established polynomial-time convergence of a projected stochastic subgradient descent algorithm driven by soft separation; and in the mixed-integer setting, we derived a high-probability certificate of global optimality for a branch-and-cut method incorporating soft separation. Numerical results demonstrated the effectiveness and scalability of the proposed methods, particularly as the problem dimension and number of scenarios increase.




\newpage
\bibliographystyle{informs2014} 
\bibliography{references} 

\newpage

\renewcommand{\thesection}{EC.\arabic{section}}
\renewcommand{\theHsection}{EC.\arabic{section}}

\renewcommand{\thesubsection}{\thesection.\arabic{subsection}}
\renewcommand{\theHsubsection}{EC.\arabic{section}.\arabic{subsection}}

\renewcommand{\thesubsubsection}{\thesubsection.\arabic{subsubsection}}
\renewcommand{\theHsubsubsection}
{EC.\arabic{section}.\arabic{subsection}.\arabic{subsubsection}}

\ECSwitch
\ECHead{Supplementary Results}
\section{Extension to the General Conic~\eqref{main}}
\label{app:conic}
This section extends the analysis in Section~\ref{sec:2} to a general cone \(\mathcal K\) in~\eqref{main}. The main difference is that relatively complete recourse alone no longer
guarantees strong duality and a uniform bound of the dual optimal solutions. We retain
Assumptions~\ref{assump:5}--\ref{assump:1} and replace Assumption~\ref{assump:linear-recourse} with the following conic regularity
condition to prove counterparts of Lemmas~\ref{lem:recourse_lipschitz_smoothing} and~\ref{prop:L-smooth}.

\begin{assumption}[Conic recourse regularity]
\label{ass:conic-regularity}
The set $\mathcal K\subseteq\mathbb R^m$ is a nonempty closed convex cone with
nonempty interior.  For every $(x,\xi)\in\mathcal X\times\Xi$, there exists a
point $\bar y=\bar y(x,\xi)$ such that
\begin{equation}
  G\bar y+Ex+U\xi-h\in\operatorname{int}\mathcal K.
  \label{eq:conic-slater}
\end{equation}
\end{assumption}
Define
\(
  t(x,\xi):=h-Ex-U\xi
\), 
\(
  \Pi_{\mathcal K}:=
  \{\pi\in\mathcal K^\star:G^\top\pi=b\},
\)
where
$\mathcal K^\star:=\{\pi\in\mathbb R^m:\pi^\top z\geq 0
\ \text{for all }z\in\mathcal K\}$ is the dual cone.  By conic strong
duality and Assumption~\ref{assump:1},~\eqref{eq:conic-slater} implies that
\begin{equation}
  Q(x,\xi)
  =\max_{\pi\in\Pi_{\mathcal K}}\pi^\top t(x,\xi),
  \qquad \forall (x,\xi)\in\mathcal X\times\Xi,
  \label{eq:conic-recourse-dual}
\end{equation}
and the maximum is attained~\citep[see, e.g.,][]{ben2001lectures}. The 
regularity properties of \(Q(x, \xi)\) follow.

\begin{lemma}[Conic counterparts of Lemma~\ref{lem:recourse_lipschitz_smoothing}]
\label{lem:conic-moreau}
Let $\pi^\star(x,i)$ denote the minimum-norm maximizer of~\eqref{eq:conic-recourse-dual} at $(x,\xi^{(i)})$. Then,
\(
-E^\top\pi^\star(x,i)\in\partial_x\mathcal Q(x,\xi^{(i)}).
\)
Furthermore, there exists a constant $L_Q>0$ such that, with
$L_Q':=L_Q\|E\|$ and $L_\xi':=L_Q\|U\|$, the following hold: 

(1) Lipschitz continuity in $x$:
For every $\xi\in\Xi$, $\mathcal Q(\cdot,\xi)$ is $L_Q'$-Lipschitz continuous on $\mathcal X$;

(2) Lipschitz continuity in $\xi$:
For every $x\in\mathcal X$, $\mathcal Q(x,\cdot)$ is $L_\xi'$-Lipschitz continuous on $\Xi$;

(3) Bounded subgradient selection:
The selected dual optimizer is Borel measurable and satisfies $\|E^\top\pi^\star(x,i)\|\le L_Q'$ for every $x\in\mathcal X$ and every scenario $\xi^{(i)}$.
\end{lemma}

\begin{proof}{Proof of Lemma~\ref{lem:conic-moreau}:}
First, dual feasibility and optimality of $\pi^\star(x,i)$ give, for every $x'\in\mathcal X$, 
\(
Q(x',\xi^{(i)}) \ge \pi^\star (x,i)^\top (h-Ex'-U\xi^{(i)} ) = Q(x,\xi^{(i)})
+(-E^\top\pi^\star(x,i))^{\top}(x'-x)
\), 
which proves \(
-E^\top\pi^\star(x,i)\in\partial_x\mathcal Q(x,\xi^{(i)})
\). 
It remains to establish a bound that is uniform over
$\mathcal X\times\Xi$.  Consider the right-hand-side value function
\(
  \vartheta(t):=\inf_y\{b^\top y:Gy-t\in\mathcal K\}.
\) 
For every $t(x,\xi^{(i)})$ generated by
$x\in\mathcal X$, Assumption~\ref{ass:conic-regularity} places $t$ in the
interior of the domain on which $\vartheta$ is finite. Hence, $\vartheta$ is
locally Lipschitz at these right-hand sides~\citep[see][]{luan2022two}. Because $\mathcal X$ is compact and $\Xi$ is finite,
the sets
\(
  \{t(x,\xi^{(i)}):x\in\mathcal X\}
\) with \(i \in \I\) 
form a finite collection of compact sets.  Moreover, since the dual feasible set does not depend on the right-hand side, $\pi^\star(x,i)$ remains dual feasible at $t'$.
Weak duality and optimality at $t(x,\xi^{(i)})$ therefore give
{\footnotesize
\[
\vartheta(t') \ge \pi^\star(x,i)^\top t' =\pi^\star(x,i)^\top t(x,\xi^{(i)})+\pi^\star(x,i)^\top\bigl(t'-t(x,\xi^{(i)})\bigr) =\vartheta(t(x,\xi^{(i)})) +\pi^\star(x,i)^\top\bigl(t'-t(x,\xi^{(i)})\bigr),
\]
}
where the last equality follows from strong duality. 
It follows that \(\pi^\star(x,i) \in \partial\vartheta(t(x,\xi^{(i)}))\).

Local boundedness of the
subdifferential of a finite convex function therefore gives a common constant
$L_Q<\infty$ such that every dual optimizer $\pi^\star(x,i)$ in
\eqref{eq:conic-recourse-dual} satisfies
$\lVert\pi^\star(x,i)\rVert\leq L_Q$.  Setting
$L'_Q:=L_Q\| E \|$ and $L_\xi':=L_Q\|U\|$ gives $\|E^\top\pi^\star(x,i)\| \le L_Q'$ and the Lipschitz continuity in claims (1)--(2). \hfill \(\Box\)
\end{proof}

Lemma~\ref{lem:conic-moreau} supplies exactly the properties used
in the convergence analysis of Section~\ref{sec:2}.  In
particular, with
\(
Z_x(i):=g_f(x)-E^\top\pi^\star(x,i), \ z_w(x):=\sum_{i\in\I}p_x(i)Z_x(i)\in\partial F_w(x),
\)
we have $\|Z_x(i)\|\le L_f+L_Q', \|z_w(x)\|\le L_f+L_Q'.$ Consequently, Proposition~\ref{prop:poisson-noV} remains valid and the convergence analysis of Algorithm~\ref{alg:MH-SGD2} carries over without further changes.

\section{Adaptive Proposal for the Budget Uncertainty Set}
\label{app:adaptive-proposal-sampling}

We detail the exact samplers for the budget uncertainty sets in Section~\ref{subsec:continuous-first-stage}. 
For budget $\Gamma$, each extreme point of the budget-saturating face (which suffices for
our purposes since \(Q(x,\xi)\) is componentwise nondecreasing in \(\xi\)) has $\lfloor n \Gamma \rfloor$ coordinates equal to one and one
coordinate equal to $\Gamma-\lfloor n\Gamma \rfloor \in [0,1)$. Define
\[
r_\Gamma^2=\lfloor n\Gamma\rfloor+(n\Gamma-\lfloor n\Gamma\rfloor)^2-\frac{(n\Gamma)^2}{n}, \quad \mathrm{index}_x(i)
=
\frac{\kappa\,\left(\mathrm{Id}-\frac{1}{n}\mathbf 1\mathbf 1^\top\right)\left(-U^\top\pi^\star(x,i)\right)}
     {r_\Gamma\|\left(\mathrm{Id}-\frac{1}{n}\mathbf 1\mathbf 1^\top\right)\left(-U^\top\pi^\star(x,i)\right)\|}.
\]
The proposal distribution over the extreme points is $q_x^{\mathrm{budget}}(\xi'\mid\xi^{(i)})
\propto
 \exp\!\left(\mathrm{index}_x(i)^\top\xi'\right)
$. We use a dynamic program (DP) to compute the denominator (normalizer) of \(q_x^{\mathrm{budget}}(\xi'\mid\xi^{(i)})\) without enumerating the extreme points. We let $Z_j(r,h)$ be the total weight obtainable from coordinates $j,\ldots,n$ when $r$ unit coordinates and $h\in\{0,1\}$ fractional coordinates remain. With $Z_{n+1}(0,0)=1$ and all other terminal states equal to zero, the backward recursion of the DP is defined through
\[
Z_j(r,h) = Z_{j+1}(r,h)
+\mathbf 1_{\{r>0\}}e^{\mathrm{index}_x(j)}Z_{j+1}(r-1,h)  +\mathbf 1_{\{h=1\}}e^{(n\Gamma-\lfloor n\Gamma \rfloor) \mathrm{index}_x(j)}Z_{j+1}(r,0)
\]
Then, 
\(
\sum_{i \in\mathcal I}
\exp\!\left(\mathrm{index}_x(i)^\top\xi'\right) = Z_1\!\left(\lfloor n\Gamma \rfloor,\mathbf 1_{\{n\Gamma-\lfloor n\Gamma \rfloor>0\}}\right). 
\) 
Thus, a sample is generated sequentially by assigning coordinate $j$ the
value $0$, $1$, or $n\Gamma-\lfloor n\Gamma \rfloor>0$ with weight $Z_{j+1}(r,h)/Z_j(r,h)$, $(\mathbf 1_{\{r>0\}}e^{\mathrm{index}_x(j)}Z_{j+1}(r-1,h))/Z_j(r,h)$, or $(\mathbf 1_{\{h=1\}}e^{(n\Gamma-\lfloor n\Gamma \rfloor ) \mathrm{index}_x(j)}Z_{j+1}(r,0))/Z_j(r,h)$ respectively. After the draw, the DP state is updated through 
\[
(r,h)\longleftarrow
\begin{cases}
(r,h),& \text{if \(\xi_j=0\)},\\
(r-1,h),& \text{if \(\xi_j=1\)},\\
(r,0),& \text{if \(\xi_j=n\Gamma-\lfloor n\Gamma \rfloor\)}.
\end{cases}
\]
This DP defines a valid sampling procedure. Indeed, any infeasible assignment has zero probability because the corresponding
DP value function equals zero. Finally, we note that building the DP table requires $\mathcal O(n \lfloor n\Gamma \rfloor)$ operations and the
subsequent draw requires $\mathcal O(n)$ operations.

\section{Supplementary Implementation Details}
\label{app:exp}

\paragraph{1. Preconditioning subgradients.} When implementing Algorithm~\ref{alg:MH-SGD2} in Section~\ref{subsec:continuous-first-stage}, we precondition the subgradients \(Z_{x_n}(I_{n+1}) = g_f(x) - E^{\top}\pi^\star(x, \xi^{I_{n+1}})\) to reduce the sensitivity to heterogeneous magnitudes of the subgradient entries. Recall that, in the robust location-transportation problem, the first-stage cost is \(f(x) := c^{\top}x\) and the coefficients \(c \in \mathbb{R}^{2m}\) consist of two parts, with the first \(m\) entries denoting the facility opening costs and the last \(m\) entries denoting the capacity expansion costs. Then, \(g_f(x) \equiv c\). In addition, the first \(m\) entries of the dual vector \(E^{\top}\pi^\star(x, \xi^{I_{n+1}})\) equal zero because the facility opening variables \(y_i\) do not partake the second-stage transportation linear program, and the last \(m\) entries of \(E^{\top}\pi^\star(x, \xi^{I_{n+1}})\) lie within the interval \([\mathbf 0, h\mathbf 1]\), where \(h\) denotes the unit cost of using the emergency supply. Therefore, we scale \(Z_{x_n}(I_{n+1})\) by a diagonal matrix \(H \in \mathbb{R}^{2m\times 2m}\) with
\[
H_{ii}\propto\max\{|c_i|,10^{-8}\}, \quad
H_{m+i,m+i}\propto
\max\{|c_{m+i}|,|c_{m+i}-h|,10^{-8}\}
\quad \forall i\in[m],
\]
and we normalize \(H\) so that \(\det(H)=1\). Accordingly, we implement the subgradient descent as
\[
x
\gets
\Pi_{\mathcal X}
\bigl(x-\alpha_{n+1}H^{-1}Z_{x_n}(I_{n+1})\bigr),
\]
where we pre-compute \(H\) and fix it throughout all experiments.

\paragraph{2. Practical certification.}
When implementing Algorithm~\ref{alg:outer-loop} in Section~\ref{subsec:integer-first-stage}, we generalize the certification criterion $\bar\theta\ge Q^\star(\bar x)$ to be $\bar\theta + \varepsilon_{\mathrm{cert}} \ge Q^\star(\bar x)$ for a certification error $\varepsilon_{\mathrm{cert}}\in(0,\varepsilon_{\mathrm{opt}})$. Accordingly, if we establish a slightly more conservative upper bound \(\widehat U := f(\bar x)+\bar\theta+\varepsilon_{\mathrm{cert}}\), then the same probabilistic guarantee of Theorem~\ref{thm:outer-loop} continues to hold. Practically, the looser certification criterion makes it easier to find an adversarial scenario, giving rise to a smaller value of \(\rho\). For our independent uniform proposal \(\nu\), the corresponding one-call miss probability becomes 
\[
\rho_{\varepsilon_{\mathrm{cert}}}(\bar x,\bar\theta)
:=
1-
\nu\!\left(
\left\{\xi\in\Xi:
Q(\bar x,\xi)>
\bar\theta+\varepsilon_{\mathrm{cert}}
\right\}
\right).
\]
In practice, \(\rho_{\varepsilon_{\mathrm{cert}}}(\bar x,\bar\theta)\) can be estimated from the empirical quantiles of the recourse function values \(\big\{Q(x_n, \xi^{(I_{n+1})})\big\}_n\), which are observed in the warm-start of Algorithm~\ref{alg:outer-loop} (line 1). We use $\rho=0.9$ and $\varepsilon_{\mathrm{cert}}=100$ for all experiments in Section~\ref{subsec:integer-first-stage}.


\section{Technical Proofs}
\label{app:proof}

\subsection{Proof of Lemma \ref{lem:recourse_lipschitz_smoothing}}

\begin{proof}{Proof:}
First, to show the representation~\eqref{QP}, we write the dual formulation for \(Q(x, \xi)\) as
\(
\max_{\pi \in \Pi} \pi^{\top}(h-Ex-U\xi),
\)
where \(\pi\) denotes dual variables associated with the constraints in~\eqref{recourse} and $\Pi=\{\pi\ge0:G^\top\pi=b\}$ is the dual feasible region. The strong duality holds because of Assumption~\ref{assump:1}. 

Since $\Pi$ is a nonempty pointed polyhedron, its set of extreme points $\mathcal V:=\operatorname{ext}(\Pi)$ is nonempty and finite. It follows from Assumption~\ref{assump:1} that, for each \(Q(x, \xi)\), an optimal \(\pi^\star\) exists in \(\mathcal V\).

For \textit{fixed} $(x,i)$, the auxiliary \(\pi^\star\)-selection problem admits a 
linear programming formulation
\(
\min_{\pi \in \Pi,z} \{\mathbf 1^\top z: \ \pi^\top (h-Ex-U\xi)=Q(x,\xi^{(i)}), \ -z\le E^\top\pi\le z\}.
\)
This program is feasible and bounded below by zero, so its
minimum is attained. More generally, for \textit{every} nonempty face (denoted by \texttt{Face}) of $\Pi$, the same argument shows that $\min_{\pi\in \texttt{Face}}\|E^\top\pi\|_1$ is attained. We fix, for each \texttt{Face}, a representative $\widehat\pi_\texttt{Face} \in\argmin_{\pi\in \texttt{Face}}\|E^\top\pi\|_1$. The dual optimal set is a nonempty face of $\Pi$, denoted by $\texttt{Opt-Face}$, so the prescribed selection is $\pi^\star(x,i)=\widehat\pi_{\texttt{Opt-Face}}$. \textbf{Note that this mapping is measurable}. To see this, $\Pi$ has finitely many vertices and extreme rays. The region of $(x,\xi)$ on which $\texttt{Opt-Face}$ equals a given face is determined by finitely
many equalities and strict inequalities involving their
affine objective values, and is therefore Borel measurable, denoted by $R_{\texttt{Face}}$. On each such Borel region, the selection equals the fixed representative $\widehat\pi_\texttt{Face}$ of the corresponding face. Hence, the preimage of any open set under $\pi^\star$, namely $(\pi^\star)^{-1}(\mathcal O)$, is a finite union of such Borel regions (i.e. $\cup_{\texttt{Face}:\widehat{\pi}_{\texttt{Face}}\in \mathcal O} R_{\texttt{Face}}$), which is a Borel set and proves measurability.

Second, for every $x,y\in\mathcal X$ and every scenario $\xi^{(i)}$,
dual feasibility and optimality at $x$ give
\[
\begin{aligned}
Q(y,\xi^{(i)})
&=\max_{\pi \in \Pi}
\pi^\top(h-Ey-U\xi^{(i)})\\
&\ge \pi^\star(x,\xi^{(i)})^\top(h-Ey-U\xi^{(i)})\\
&=Q(x,\xi^{(i)})
+(-E^\top\pi^\star(x,\xi^{(i)}))^{\top}(y-x).
\end{aligned}
\]
Thus, $-E^\top\pi^\star(x,i)\in\partial_xQ(x,\xi^{(i)})$. 
It remains to show properties~{(1)--(3)}. Choose $L_Q:=\max\left\{1,\max_{v\in\mathcal V}\|v\|,
\max\|\widehat\pi_\texttt{Face}\|\right\}<\infty$. Then, for every $x_1,x_2\in\mathcal X$ and $\xi\in\Xi$, the finite
maximum representation yields
\(
|Q(x_1,\xi)-Q(x_2,\xi)| \le\max_{v\in\mathcal V}|v^\top E(x_1-x_2)| \le L_Q\|E\|\,\|x_1-x_2\|,
\) 
proving that $Q(\cdot,\xi)$ is $L_Q'$-Lipschitz on $\mathcal X$. 
Similarly, for all \(\xi^1, \xi^2 \in \Xi\), 
\(
|Q(x,\xi^1)-Q(x,\xi^2)| \le\max_{v\in\mathcal V}|v^\top U(\xi^1-\xi^2)| \le L_Q\|U\|\,\|\xi^1-\xi^2\|, 
\)
proving that $Q(x,\cdot)$ is $L'_\xi$-Lipschitz on $\Xi$. 

Finally, since \(\pi^\star(x,i)\) equals one of the fixed representatives $\widehat\pi_{\texttt{Face}}$, it holds that
\(
\|E^\top\pi^\star(x,i)\|
\le\|E\|\|\pi^\star(x,i)\|
\le L_Q\|E\|=L'_Q.
\)
Together with measurability, this proves property (3). \hfill \(\Box\)
\end{proof}

\subsection{Proof of Lemma~\ref{prop:L-smooth}}
\begin{proof}{Proof}
For any $x,y\in\mathcal X$, the definition of $p_x$ gives
\begin{align*}
F_w(y)-F_w(x) &=f(y)-f(x)
+w\log\!\left(
\sum_{i\in\mathcal I}p_x(i)
\exp\!\left(
\frac{Q(y,\xi^{(i)})-Q(x,\xi^{(i)})}{w}
\right)\right) 
\\
&\ge f(y)-f(x)
+\sum_{i\in\mathcal I}p_x(i)
\bigl(Q(y,\xi^{(i)})-Q(x,\xi^{(i)})\bigr)\\
&\ge g_f(x)^{\top}(y-x)
+\sum_{i\in\mathcal I}p_x(i)
(-E^\top\pi^\star(x,i))^{\top}(y-x) \ = \ z_w(x)^{\top}(y-x),
\end{align*}
where the first inequality follows from Jensen's inequality
and the second follows from the subgradient inequalities for
$f$ and $Q(\cdot,\xi^{(i)})$ by Lemma~\ref{lem:recourse_lipschitz_smoothing}. 
Thus, $z_w(x)\in\partial F_w(x)$.

For any $x\in\mathcal X$ and any unit vector $u$, convexity and
Lipschitz continuity on an open neighborhood of $\mathcal X$ give
$t \cdot g_f(x)^{\top}u \le f(x+tu)-f(x)\le L_ft$
for a sufficiently small $t>0$. Thus, $\|g_f(x)\|\le L_f$ and
\(
\|Z_x(i)\| =\|g_f(x)-E^\top\pi^\star(x,i)\| \le\|g_f(x)\|+\|E^\top\pi^\star(x,i)\| \le L_f+L_Q'.
\) 
Since $p_x(i)\ge0$ and $\sum_{i\in\mathcal I}p_x(i)=1$, we have 
\(
\|z_w(x)\|
=
\left\|\sum_{i\in\mathcal I}p_x(i)Z_x(i)\right\|
\le\sum_{i\in\mathcal I}p_x(i)\|Z_x(i)\|
\le L_f+L_Q'.
\) 
This completes the proof. \hfill \(\Box\)
    
\end{proof}

\subsection{Proof of Lemma~\ref{lem:doeblin}}
\begin{proof}{Proof:}
 For any stochastic matrix $S$, the coefficient of ergodicity
 (Dobrushin coefficient, see~\cite{Dobrushin1956}) is defined by
\[
 \tau(S)
 \;:=\;
 \frac12\max_{i,k}\|S(i,\cdot)-S(k,\cdot)\|_1
 \;=\;
 \max_{\substack{\tilde z^\top \mathbf 1=0\\ \|\tilde z\|_1=1}}\|S^\top \tilde z\|_1.
 \] Define \(
R_x := \big(\mathrm{Id}-P_x+\mathbf{1}p_x^\top\big)^{-1}.
\) Since $P_x$ is an irreducible and aperiodic Markov matrix on a finite
state space, it has a simple eigenvalue $1$ with right eigenvector
$\mathbf{1}$ and left eigenvector $p_x^\top$, while all other
eigenvalues $\lambda$ satisfy $|\lambda|<1$.

For any eigenvector $v$ with $P_xv=\lambda v$ and $\lambda\neq1$,
multiplying by $p_x^\top$ gives
$
\lambda p_x^\top v
=
p_x^\top P_xv
=
p_x^\top v,
$
hence
$
(\lambda-1)p_x^\top v=0,
$
and therefore $p_x^\top v=0$. 
Moreover, $\mathbb{R}^n
=
\operatorname{span}\{\mathbf1\}
\oplus
\ker(p_x^\top),$ and both subspaces are invariant under $P_x$. On
$\operatorname{span}\{\mathbf1\}$, the operator
$P_x-\mathbf1p_x^\top$ acts as zero, while on
$\ker(p_x^\top)$ it agrees with $P_x$.  Consequently,
\(
\sigma\big(P_x-\mathbf1p_x^\top\big)
=
\{0\}\cup
\big(\sigma(P_x)\setminus\{1\}\big)
\)
and hence
\(
\rho\big(P_x-\mathbf1p_x^\top\big)
=
\max\{|\lambda|:
\lambda\in\sigma(P_x)\setminus\{1\}\}
<1.
\)
Now fix $N\ge0$. For the finite sum, we have the exact geometric
identity
\(
\big(\mathrm{Id}-P_x+\mathbf1p_x^\top\big)
\sum_{t=0}^{N}
\big(P_x-\mathbf1p_x^\top\big)^t
=
\mathrm{Id}-
\big(P_x-\mathbf1p_x^\top\big)^{N+1}
\)
and similarly on the right.

Since
$\rho(P_x-\mathbf1p_x^\top)<1,$
we have
$
\big\|
(P_x-\mathbf1p_x^\top)^{N+1}
\big\|
\to0
$
as $N\to\infty$ in any operator norm. Taking limits yields
\(
R_x
=
\big(\mathrm{Id}-P_x+\mathbf1p_x^\top\big)^{-1}
=
\sum_{t=0}^{\infty}
\big(P_x-\mathbf1p_x^\top\big)^t.
\)
We next bound the $\infty\!\to\!\infty$ norm of \(R_x\).
For any vector $z\in\mathbb{R}^n$, define
$
\widetilde z
:=
z-(z^\top\mathbf1)p_x,
$ then
$
\widetilde z^\top\mathbf1=0.
$
Moreover,
$
z^\top(P_x-\mathbf1p_x^\top)
=
\widetilde z^\top P_x,
$
and, by induction, for every $t\ge1$,
\(
z^\top(P_x-\mathbf1p_x^\top)^t
=
\widetilde z^\top P_x^t.
\)
Since $p_x$ is a probability vector, it holds that 
\(
\|\widetilde z\|_1
=
\|z-(z^\top\mathbf1)p_x\|_1
\le
\|z\|_1
+
|z^\top\mathbf1|\,\|p_x\|_1
\le
2\|z\|_1.
\)
Thus, $\|z\|_1=1$ implies
$\|\widetilde z\|_1\le2$.
Using the row-vector characterization of the induced
$\infty\!\to\!\infty$ norm, for every $t\ge1$,
\[
\begin{aligned}
\big\|
(P_x-\mathbf1p_x^\top)^t
\big\|_{\infty\to\infty}
&=
\sup_{\|z\|_1=1}
\left\|
z^\top(P_x-\mathbf1p_x^\top)^t
\right\|_1\\
&=
\sup_{\|z\|_1=1}
\|\widetilde z^\top P_x^t\|_1 \ \le \ 
\sup_{\substack{
\widetilde z^\top\mathbf1=0\\
\|\widetilde z\|_1\le2}}
\|\widetilde z^\top P_x^t\|_1 \ = \
2\,\tau(P_x^t).
\end{aligned}
\]
By the submultiplicativity of the Dobrushin coefficient,
$
\tau(P_x^t)
\le
[\tau(P_x)]^t
$ for all \(t\ge1\). 
Therefore,
\(
\big\|
(P_x-\mathbf1p_x^\top)^t
\big\|_{\infty\to\infty}
\le
2[\tau(P_x)]^t.
\)
Finally, using the geometric representation of $R_x$ and the triangle
inequality, we conclude that 
\[
\begin{aligned}
\|R_x\|_{\infty\to\infty}
&\le
1+
\sum_{t=1}^{\infty}
\big\|
(P_x-\mathbf1p_x^\top)^t
\big\|_{\infty\to\infty}\\
&\le
1+
2\sum_{t=1}^{\infty}
[\tau(P_x)]^t
\ = \ 
\frac{1+\tau(P_x)}{1-\tau(P_x)}
\ \le \ 
\frac{2}{1-\tau(P_x)}.
\end{aligned}
\]
We further bound it through a lemma.  
\begin{lemma}[Doeblin minorization implies bounded coefficient]
If $P_x$ satisfies the Doeblin minorization condition $P_x(i,\cdot)\;\ge\;\gamma_x\,\nu_x(\cdot),
 $ for some probability measure $\nu_x$ and constant $\gamma_x>0$, 
 then the ergodicity coefficient satisfies $\tau(P_x)\le1-\gamma_x$.
 \end{lemma}
 This is the standard result from literature; see \cite[Eq.~(2) and Theorem~1, Part~6]{MakurSingh2024}. Combining all completes the proof. \hfill \(\Box\)

\end{proof}

\subsection{Proof of Proposition~\ref{prop:poisson-noV}}
The proof of Proposition~\ref{prop:poisson-noV} uses the following three lemmas. Lemmas~\ref{lem:lip barker}--\ref{lem:lipschitz-MH-nom} prove Lipschitz continuity of the Markov kernel \(P_x(\cdot, \cdot)\), and then Lemma~\ref{prop:2} bounds the sample error. In addition, Lemma~\ref{lem:lip-kernal-general} generalizes Lemmas~\ref{lem:lip barker}--\ref{lem:lipschitz-MH-nom} to a general state space. Lemma~\ref{lem:lip-kernal-general} is not directly used in the proof of Proposition~\ref{prop:poisson-noV} but will be useful later in Section~\ref{sec:3}. Lemma~\ref{lem:target-distribution-lipschitz} establishes the Lipschitz
continuity of \(p_x\) on a general state space and is used both below
and in Section~\ref{sec:3}. 

\begin{lemma}[Element/row-wise Lipschitz continuity of the Barker kernel]\label{lem:lip barker}
Consider a finite state space $\mathcal{I}$ and a family of target distributions 
$p_x(i)\propto \exp\!\big(Q_{i}(x)/w\big)$, where each function 
$Q_{i}(\cdot)$ is $L_Q'$--Lipschitz on \(\X\) and $w>0$ is fixed.  
Let $q(j\mid i)>0$ denote a fixed proposal kernel, and define the Barker transition matrix 
$P_x=(P_x(i,j))_{i,j\in\mathcal{I}}$ by
\[
P_x(i,j)
=
\begin{cases}
q(j\mid i)\,A_x(i,j), & j\neq i,\\[3pt]
1-\sum_{k\neq i}q(k\mid i)\,A_x(i,k), & j=i,
\end{cases}
\quad\text{where}\quad
A_x(i,j)=\dfrac{r_x(i,j)}{1+r_x(i,j)},
\]
and $r_x(i,j):=\dfrac{p_x(j)\,q(i\mid j)}{p_x(i)\,q(j\mid i)}$.
Then for all $x,x'\in\mathcal{X}$ and $i,j \in \I$,
\(
|P_x(i,j)-P_{x'}(i,j)| \;\le\; ({L_Q'}/{(2w)})\,\|x-x'\|.
\)
Moreover, for each row $i$,
\(
\sum_{j \in \I}|P_x(i,j)-P_{x'}(i,j)| \;\le\;  L_P\|x-x'\|
\)
with \(L_P := {L'_Q}/{w}\).
\end{lemma}
\begin{proof}{Proof of Lemma~\ref{lem:lip barker}:}
Write 
$A_x(i,j)=\sigma(\alpha_x(i,j))$ with $\sigma(t)=\tfrac{e^t}{1+e^t}$ and 
\(
\alpha_x(i,j)=\log p_x(j)-\log p_x(i)+\log q(i\mid j)-\log q(j\mid i) =(Q_j(x)-Q_i(x))/w
+\log q(i\mid j)-\log q(j\mid i).
\) Since the proposal is independent of $x$ and each $Q_i$ is
$L_Q'$-Lipschitz, we have
\[
|\alpha_x(i,j)-\alpha_{x'}(i,j)| \le\frac{|Q_j(x)-Q_j(x')|+|Q_i(x)-Q_i(x')|}{w} \le\frac{2L_Q'}{w}\|x-x'\|.
\]

Since $\sigma'(t)=\sigma(t)(1-\sigma(t))\le 1/4$, we have 
\(
|A_x(i,j)-A_{x'}(i,j)| \le \tfrac{1}{4}\,|\alpha_x(i,j)-\alpha_{x'}(i,j)| \le ({L_Q'}/{(2w)})\|x-x'\|.
\)
We note that, for off-diagonal entries,
\[
|P_x(i,j)-P_{x'}(i,j)|=q(j\mid i)\,|A_x(i,j)-A_{x'}(i,j)|
\le q(j\mid i)\,\tfrac{L_Q'}{2w}\,\|x-x'\| \le \tfrac{L_Q'}{2w}\|x-x'\|,
\]
and for the diagonal entry,
\[
|P_x(i,i)-P_{x'}(i,i)|
\le \sum_{k\ne i} q(k\mid i)\,|A_x(i,k)-A_{x'}(i,k)|
\le \tfrac{L_Q'}{2w}\,\|x-x'\| \sum_{k\ne i} q(k\mid i)
\le \tfrac{L_Q'}{2w}\,\|x-x'\|.
\]
Summing over $j$ gives 
\(
\sum_j |P_x(i,j)-P_{x'}(i,j)|
\le \tfrac{L_Q'}{2w}\|x-x'\|\big(\sum_{k\ne i} q(k\mid i)+1\big)
\le \tfrac{L_Q'}{w}\|x-x'\|.
\) \hfill \(\Box\)
\end{proof}

\begin{lemma}[Element/row-wise Lipschitz continuity of the MH kernel]
\label{lem:lipschitz-MH-nom}
Let $\mathcal I$ be a finite state space and consider target distributions
$p_x(i)\propto \exp\!\big(Q_{i}(x)/w\big)$, where each $Q_{i}(\cdot)$ is $L_Q'$--Lipschitz on \(\X\) and $w>0$ is fixed.
Let $q(j\mid i)>0$ be a proposal kernel with $\sum_{j} q(j\mid i)=1$ for all $i$.
Define the Metropolis--Hastings transition matrix by
\[
P_x(i,j)=
\begin{cases}
q(j\mid i)\,A_x(i,j), & j\neq i,\\[3pt]
1-\sum_{k\neq i} q(k\mid i)\,A_x(i,k), & j=i,
\end{cases}
\qquad
A_x(i,j)=\min\!\left\{\,1,\ \dfrac{p_x(j)\,q(i\mid j)}{p_x(i)\,q(j\mid i)}\right\}.
\]
Then for all $x,x'\in\mathbb R^d$ and $i,j\in\mathcal I$,
\(
|P_x(i,j)-P_{x'}(i,j)| \;\le\; q(j\mid i)\,({2L'_Q}/{w})\,\|x-x'\| 
\) whenever \(j \neq i\) and 
\(
|P_x(i,i)-P_{x'}(i,i)| \;\le\; ({4L'_Q}/{w})\,\|x-x'\|.
\)
In particular, for each row $i$, it holds that 
\(
\sum_{j\in\mathcal I} |P_x(i,j)-P_{x'}(i,j)| \;\le\;  L_P\|x-x'\| 
\)
with \(L_P := {4L'_Q}/{w}\).
\end{lemma}

\begin{proof}{Proof of Lemma~\ref{lem:lipschitz-MH-nom}:}
Set $
\alpha_x(i,j):=\log p_x(j)-\log p_x(i)+\log q(i\mid j)-\log q(j\mid i).
$
Note $A_x(i,j)=g(\alpha_x(i,j))$ with $g(t):=\min\{1,e^t\}$. Since $g$ is $1$--Lipschitz on $\mathbb R$, 
we have
\(
|A_x(i,j)-A_{x'}(i,j)| \;\le\; |\alpha_x(i,j)-\alpha_{x'}(i,j)|.
\)
Since the proposal is independent of $x$ and \(\log p_x(j) - \log p_x(i) = (Q_j(x)-Q_i(x))/w\), Lipschitz continuity
of $Q_i$ implies 
\[
|\alpha_x(i,j)-\alpha_{x'}(i,j)| \le\frac{
|Q_j(x)-Q_j(x')|
+|Q_i(x)-Q_i(x')|
}{w} \le\frac{2L_Q'}{w}\|x-x'\|.
\]

Therefore, for $j\ne i$, we have 
\(
|P_x(i,j)-P_{x'}(i,j)|
=q(j\mid i)|A_x(i,j)-A_{x'}(i,j)|
\le q(j\mid i)\frac{2L_Q'}{w}\|x-x'\|,
\)
and, for the diagonal entry,
\[
|P_x(i,i)-P_{x'}(i,i)|
= \Big|\sum_{k\ne i} q(k\mid i)\big(A_{x'}(i,k)-A_x(i,k)\big)\Big|
\le \sum_{k\ne i} q(k\mid i)\,|A_x(i,k)-A_{x'}(i,k)|
\le \tfrac{2L_Q'}{w}\,\|x-x'\|.
\]
Summing over $j$ in row $i$ and using $\sum_{j\ne i} q(j\mid i)\le 1$ yield the row-wise bound. \hfill \(\Box\)
\end{proof}

\begin{lemma}[Lipschitz continuous kernels for a general state space] \label{lem:lip-kernal-general}
Under Assumption~\ref{ass:lipschitz-scores}, consider a Markov kernel \(P_x\) induced by (i) a proposal distribution \(\nu(\cdot)\) on \(\I\) (which is independent of \(x\)) and (ii) either MH or Barker acceptance. Then, Assumption~\ref{ass:general-kernel-lipschitz} holds with $L_P = 2L'_{Q}/w$. In other words, \(
\sup_{I\in\mathcal I}
\|P_x(I,\cdot)-P_y(I,\cdot)\|_{\mathrm{TV}}
\ \le \ (2L'_{Q}/w)\|x-y\|
\) for all \(x, y \in \X\).
\end{lemma}
\begin{proof}{Proof of Lemma~\ref{lem:lip-kernal-general}:}
Let \(A_x(I, J)\) denote either MH or Barker acceptance distribution. 
For every measurable $S\subseteq\mathcal I$, we have
\begin{align*}
    P_x(I,S)-P_{y}(I,S)
    &=
    \int_{\mathcal I}
    A_x(I,J)\mathbf 1_S(J)\,\nu(dJ)
    +
    \mathbf 1_S(I)
    \int_{\mathcal I}
    \bigl(1-A_x(I,J)\bigr)\nu(dJ) \\
    &\quad-
    \int_{\mathcal I}
    A_{y}(I,J)\mathbf 1_S(J)\,\nu(dJ)
    -
    \mathbf 1_S(I)
    \int_{\mathcal I}
    \bigl(1-A_{y}(I,J)\bigr)\nu(dJ) \\
    &=
    \int_{\mathcal I}
    \bigl(A_x(I,J)-A_{y}(I,J)\bigr)
    \mathbf 1_S(J)\,\nu(dJ) 
    -
    \mathbf 1_S(I)
    \int_{\mathcal I}
    \bigl(A_x(I,J)-A_{y}(I,J)\bigr)\nu(dJ) \\
    &=
    \int_{\mathcal I}
    \bigl(\mathbf 1_S(J)-\mathbf 1_S(I)\bigr)
    \bigl(A_x(I,J)-A_{y}(I,J)\bigr)\nu(dJ).
\end{align*}
On the one hand, if \(A_x(I, J)\) is the MH acceptance distribution, then $ A_x(I,J)
    =
    1\wedge
    \exp\left(
        \frac{\widehat Q_x(J)-\widehat Q_x(I)}{w}
    \right)$. Because $u\mapsto1\wedge e^u$ is $1$-Lipschitz, Assumption~\ref{ass:lipschitz-scores} gives
    \(
     |A_x(I,J)-A_{y}(I,J)| \le
    \big({
        |\widehat Q_x(J)-\widehat Q_{y}(J)|
        +
        |\widehat Q_x(I)-\widehat Q_{y}(I)|
    }\big)/{w} \le
    ({2L'_Q}/{w})\|x-y\|.
    \)
On the other hand, if \(A_x(I, J)\) is the Barker acceptance distribution, then 
\(
A_x(I, J) = \exp({(\widehat Q_x(J)-\widehat Q_x(I))}/{w})/(1 + \exp({(\widehat Q_x(J)-\widehat Q_x(I))}/{w}).
\)
Because $u\mapsto e^u/(1+e^u)$ is $1$-Lipschitz (in fact, it is \((1/4)\)-Lipschitz), the same conclusion holds that \(|A_x(I,J)-A_{y}(I,J)| \le (2L_Q'/w)\|x-y\|\). 
It follows that 
\[
|P_x(I,S)-P_{y}(I,S)|\le
  \max | \mathbf 1_S(J)-\mathbf 1_S(I)|  \int_{\mathcal I}
    |A_x(I,J)-A_{y}(I,J)|\,\nu(dJ)
    \le
    \frac{2L_Q'}{w}\|x-y\|.
\]
Taking the supremum over measurable $S$ and $I\in\mathcal I$
completes the proof. \hfill \(\Box\)
\end{proof}

\begin{lemma}[Lipschitz continuity of the target distribution]
\label{lem:target-distribution-lipschitz}
Under Assumption~\ref{ass:lipschitz-scores}, the target distribution
\(p_x\) is Lipschitz continuous in \(x\).Specifically, with
\(L_p:=2L'_Q/w\), it holds that \[
\left\|\frac{dp_x}{d\nu}-\frac{dp_y}{d\nu}\right\|_{L^1(\nu)}
=
2\|p_x-p_y\|_{\mathrm{TV}}
=
\|\mathbf 1p_x-\mathbf 1p_y\|_{\infty\to\infty}
\le L_p\|x-y\|\] for all \(x,y\in\X\). In particular, when \(\I\) is finite, $\|p_x-p_y\|_1
=
\|\mathbf 1(p_x-p_y)^\top\|_{\infty\to\infty}
\le L_p\|x-y\|.$ 
\end{lemma}
\begin{proof}{Proof of Lemma~\ref{lem:target-distribution-lipschitz}} 
    Fix \(x,y\in\X\). For \(t\in[0,1]\), define
\[
r_t(I):=
\frac{\exp\!\left(\bigl[(1-t)Q_\tau(y,\xi_I)+tQ_\tau(x,\xi_I)\bigr]/w\right)}
{\int_{\I}\exp\!\left(\bigl[(1-t)Q_\tau(y,\xi_J)+tQ_\tau(x,\xi_J)\bigr]/w\right)\nu(dJ)}.
\]
Differentiation under the integral and
Assumption~\ref{ass:lipschitz-scores} yield
\[
\frac{d}{dt}r_t(I)
=\frac{r_t(I)}{w}\left(
Q_\tau(x,\xi_I)-Q_\tau(y,\xi_I)
-\int_{\I}\bigl[Q_\tau(x,\xi_J)-Q_\tau(y,\xi_J)\bigr]r_t(J)\nu(dJ)
\right),
\left\|\frac{d}{dt}r_t\right\|_{L^1(\nu)}
\le\frac{2L'_Q}{w}\|x-y\|.
\]
Therefore,
\[
\left\|\frac{dp_x}{d\nu}-\frac{dp_y}{d\nu}\right\|_{L^1(\nu)}
=\|r_1-r_0\|_{L^1(\nu)}
\le\int_0^1\left\|\frac{d}{dt}r_t\right\|_{L^1(\nu)}dt
\le\frac{2L'_Q}{w}\|x-y\|.
\]
The remaining equalities follow from the definition of total variation
and the induced \(\ell_\infty\)-operator norm. The finite-state result
follows by identifying the densities with their probability vectors.
\end{proof}

\begin{lemma}
\label{prop:2}
Let $\widehat Z_x$ denote the centered solution of the Poisson equation $\widehat Z_x-P_x\widehat Z_x=e_x$ with $p_x^\top\widehat Z_x=0$. Then, there exists a constant $B_Z$ 
such that
\[
\sup_{x\in\mathcal X} \ 
\max\big\{\|\widehat Z_x\|_\infty,\|P_x\widehat Z_x\|_\infty\big\}
\le B_Z,
\]
where  \(B_Z =2C_*DL'_Q\). Here, $D:=\max_{x,y\in\mathcal X}\|x-y\|_1$ is the diameter of \(\X\), $L'_Q$ is the Lipschitz modulus of \(Q(\cdot, \xi)\) in Lemma~\ref{lem:recourse_lipschitz_smoothing}, and $C_*$ is the fundamental matrix norm bound from Assumption~\ref{ass:invertibility}. Moreover, there exists a constant $L_Z < \infty$ such that
\(
\|\widehat Z_x-\widehat Z_{x'}\|_\infty
\le L_Z\|x-x'\|
\) and 
\(
\|P_x\widehat Z_x-P_{x'}\widehat Z_{x'}\|_\infty
\le\bigl(L_PB_Z+L_Z\bigr)\|x-x'\|
\)
for all \(x, x' \in \mathcal{X}\), 
where 
\(
L_Z:=C_*(2L_Q'+DL_Q'L_p)+2C_*^2DL_Q'(L_P+L_p).
\)
Here, \(L_p:={2L_Q'}/{w}\) 
and $L_P$ 
comes from either Lemma~\ref{lem:lip barker} or Lemma~\ref{lem:lipschitz-MH-nom}.
\end{lemma}

\begin{proof}{Proof of Lemma~\ref{prop:2}:}
\textbf{(I) Bounding \(\max\big\{\|\widehat Z_x\|_\infty,\|P_x\widehat Z_x\|_\infty\big\}\).}
Recall that, by definition, 
\(
e_x(i)=Q(x,\xi^{(i)})-Q(x^\star,\xi^{(i)})
-\sum_{j\in\mathcal I}p_x(j)\bigl(Q(x,\xi^{(j)})-Q(x^\star,\xi^{(j)})\bigr).
\)
By Lipschitz continuity of $Q(\cdot, \xi)$ and the definition of $D$, we have 
\(
|Q(x,\xi^{(i)})-Q(x^\star,\xi^{(i)})|\le L_Q'\|x-x^\star\|\le DL_Q'.
\)
Hence, $\|e_x\|_\infty\le2DL_Q'$. 
Since $p_x^\top e_x=0$, the centered Poisson solution is, for every $i\in\mathcal I$, 
\[
\widehat Z_x(i)
=
\sum_{j\in\mathcal I}
\left[(\mathrm{Id}-P_x+\mathbf1p_x^\top)^{-1}\right]_{ij}e_x(j),
\quad 
\widehat Z_x(i)-\sum_{j\in\mathcal I}P_x(i,j)\widehat Z_x(j)=e_x(i).
\]
Then, by Assumption~\ref{ass:invertibility} and for all $i \in \mathcal{I}$, 
\(
\|\widehat Z_x(i)\|_{\infty}
\;\le C_*\sup_{j\in\mathcal I}|e_x(j)| \le 2 C_* D L'_Q.
\)
Now, because $P_x$ is a stochastic row matrix, we have 
\(
\;\|(P_x\widehat Z_x)(i)\|_{\infty}
=
\big|\sum_{j\in\mathcal I}P_x(i,j)\widehat Z_x(j)\big| 
\le \sup_{i \in \mathcal{I}} \|\widehat Z_x(i)\|_{\infty}
\le 2 C_* D L'_Q.
\)
Setting $B_Z:=2 C_* D L'_Q$ completes this part of the proof.

\textbf{(II) Lipschitz continuity of \(\widehat{Z}_x\) and \(P_x\widehat{Z}_x\) on \(\X\).} Throughout this part, we use the infinity vector norm  and its induced
operator norm for matrices. To proceed, we recall that
\item[(A)] By Lemma~\ref{lem:recourse_lipschitz_smoothing},
\(
\sup_{i\in\mathcal I}
|Q(x,\xi^{(i)})-Q(x',\xi^{(i)})|
\le L_Q'\|x-x'\|.
\)
Moreover, since $x^\star\in\mathcal X$, we have 
$
\sup_{i\in\mathcal I}
|Q(x,\xi^{(i)})-Q(x^\star,\xi^{(i)})|
\le DL_Q'.
$

\item[(B)] The Poisson operator is uniformly invertible:
\(
\sup_x\big\|\,(I-P_x+\mathbf 1p_x^\top)^{-1}\big\|_{\infty \to \infty}\le C_*<\infty
\) by Assumption~\ref{ass:invertibility}.

\item[(C)] \emph{Row-wise Lipschitz continuity of $P_x$.}
By Lemma~\ref{lem:lip barker} (Barker) or Lemma~\ref{lem:lipschitz-MH-nom} (MH),
for every $i \in \I$, it holds that 
\(
\sum_{j \in \I}|P_x(i,j)-P_{x'}(i,j)|
\le L_P\,\|x-x'\|
\)
with $L_P={L'_Q}/{w}$ for Barker (resp.\ $L_P={2L'_Q}/{w}$ for MH).
Equivalently, in the $\infty\to\infty$ operator norm,
\(
\|P_x-P_{x'}\|_{\infty\to\infty}\le L_P\,\|x-x'\|.
\)

\item[(D)] \emph{Lipschitz continuity of $p_x$ in $\ell_1$.} By Lemma~\ref{lem:target-distribution-lipschitz}, 
\(
\|p_x-p_{x'}\|_1 \le L_p\,\|x-x'\|
\)
with \(L_p:={2L_Q'}/{w}\). 
Moreover,
$\|\mathbf 1(p_x-p_{x'})^\top\|_{\infty\to\infty}=\|p_x-p_{x'}\|_1
\le L_p\|x-x'\|$.

We first bound $\displaystyle\sup_{i\in\mathcal I}|e_x(i)-e_{x'}(i)|$. By the definition of $e_x$ and properties (A), (D) above, we have 
\begin{align*}
\sup_{i\in\mathcal I}|e_x(i)-e_{x'}(i)|
&\le 2 \max_{i\in\mathcal I}
|Q(x,\xi^{(i)})-Q(x',\xi^{(i)})|+ \|p_x-p_{x'}\|_1
\max_{i\in\mathcal I}
|Q(x',\xi^{(i)})-Q(x^\star,\xi^{(i)})|\\
&\le(2L_Q'+DL_Q'L_p)\|x-x'\|. 
\end{align*}

Next, we bound $\sup_{i\in\mathcal I}|\widehat Z_x(i)-\widehat Z_{x'}(i)|$.  By the representation of the Poisson solution, 
\begin{align*}
\widehat Z_x(i)-\widehat Z_{x'}(i)
&=\underbrace{\sum_{j\in\mathcal I}
\bigl[(\mathrm{Id}-P_x+\mathbf1p_x^\top)^{-1}\bigr]_{ij}
\bigl(e_x(j)-e_{x'}(j)\bigr)}_{E}\\
&\quad+\underbrace{\sum_{j\in\mathcal I}
\Bigl[
(\mathrm{Id}-P_x+\mathbf1p_x^\top)^{-1}
-(\mathrm{Id}-P_{x'}+\mathbf1p_{x'}^\top)^{-1}
\Bigr]_{ij}e_{x'}(j)}_{F}.
\end{align*}

\noindent{a. We first bound the \(E\) term.} By the triangle inequality,
\begin{align*}
\left|
\sum_{j\in\mathcal I}
\bigl[(\mathrm{Id}-P_x+\mathbf1p_x^\top)^{-1}\bigr]_{ij}
\bigl(e_x(j)-e_{x'}(j)\bigr)
\right|
& \ \le \ 
\sum_{j\in\mathcal I}
\left|\bigl[(\mathrm{Id}-P_x+\mathbf1p_x^\top)^{-1}\bigr]_{ij}\right|
\,|e_x(j)-e_{x'}(j)|\\
&\ \le \
\left(\max_{k\in\mathcal I}\sum_{j\in\mathcal I}
\left|\bigl[(\mathrm{Id}-P_x+\mathbf1p_x^\top)^{-1}\bigr]_{kj}\right|\right)
\sup_{\ell\in\mathcal I}|e_x(\ell)-e_{x'}(\ell)|\\
&\ = \ 
\|(\mathrm{Id}-P_x+\mathbf1p_x^\top)^{-1}\|_{\infty\to\infty}
\sup_{\ell\in\mathcal I}|e_x(\ell)-e_{x'}(\ell)|\\
&\ \le \ C_*(2L_Q'+DL_Q'L_p)\|x-x'\|.
\end{align*}
Here, the equality uses
$\|A\|_{\infty\to\infty}=\max_k\sum_j|A_{kj}|$,
and the last inequality follows from property (B) and the preceding bound on $e_x-e_{x'}$.

\noindent{b. We next bound the \(F\) term.} 
We recall the resolvent identity 
\[
A^{-1}-B^{-1}=A^{-1}(B-A)B^{-1}
\ \ 
\text{with} \ \ A:=\mathrm{Id}-P_x+\mathbf1p_x^\top,\;
B:=\mathrm{Id}-P_{x'}+\mathbf1p_{x'}^\top.
\]
Then,
\begin{align*}
&\sup_{i\in\mathcal I}\left|
\left[
\Big((\mathrm{Id}-P_x+\mathbf1p_x^\top)^{-1}
-(\mathrm{Id}-P_{x'}+\mathbf1p_{x'}^\top)^{-1}\Big)e_{x'}
\right]_i\right|\\[2pt]
&\;=\Big\|
\Big((\mathrm{Id}-P_x+\mathbf1p_x^\top)^{-1}
-(\mathrm{Id}-P_{x'}+\mathbf1p_{x'}^\top)^{-1}\Big)e_{x'}
\Big\|_\infty\\[2pt]
&\;=\Big\|
(\mathrm{Id}-P_x+\mathbf1p_x^\top)^{-1}
\Big((\mathrm{Id}-P_{x'}+\mathbf1p_{x'}^\top)
-(\mathrm{Id}-P_x+\mathbf1p_x^\top)\Big)
(\mathrm{Id}-P_{x'}+\mathbf1p_{x'}^\top)^{-1}e_{x'}
\Big\|_\infty\\[2pt]
&\;=\Big\|
(\mathrm{Id}-P_x+\mathbf1p_x^\top)^{-1}
\big(P_x-P_{x'}+\mathbf1(p_{x'}-p_x)^\top\big)
(\mathrm{Id}-P_{x'}+\mathbf1p_{x'}^\top)^{-1}e_{x'}
\Big\|_\infty.
\end{align*}
The chain of equalities continues as
\[
\begin{aligned}
&\quad=
\max_k\left|
\sum_j\big[(\mathrm{Id}-P_x+\mathbf1p_x^\top)^{-1}\big]_{kj}
\Big[
\big(P_x-P_{x'}+\mathbf1(p_{x'}-p_x)^\top\big)
(\mathrm{Id}-P_{x'}+\mathbf1p_{x'}^\top)^{-1}e_{x'}
\Big]_j
\right|\\[2pt]
&\quad\le
\max_k\sum_j
\left|\big[(\mathrm{Id}-P_x+\mathbf1p_x^\top)^{-1}\big]_{kj}\right|
\Big\|
\big(P_x-P_{x'}+\mathbf1(p_{x'}-p_x)^\top\big)
(\mathrm{Id}-P_{x'}+\mathbf1p_{x'}^\top)^{-1}e_{x'}
\Big\|_\infty\\[2pt]
&\quad=
\big\|(\mathrm{Id}-P_x+\mathbf1p_x^\top)^{-1}\big\|_{\infty\to\infty}
\Big\|
\big(P_x-P_{x'}+\mathbf1(p_{x'}-p_x)^\top\big)
(\mathrm{Id}-P_{x'}+\mathbf1p_{x'}^\top)^{-1}e_{x'}
\Big\|_\infty\\[2pt]
&\quad\le
\big\|(\mathrm{Id}-P_x+\mathbf1p_x^\top)^{-1}\big\|_{\infty\to\infty}
\big\|P_x-P_{x'}+\mathbf1(p_{x'}-p_x)^\top\big\|_{\infty\to\infty}
\big\|(\mathrm{Id}-P_{x'}+\mathbf1p_{x'}^\top)^{-1}\big\|_{\infty\to\infty}
\|e_{x'}\|_\infty\\[2pt]
&\quad\le
\big\|(\mathrm{Id}-P_x+\mathbf1p_x^\top)^{-1}\big\|_{\infty\to\infty}
\big(\|P_x-P_{x'}\|_{\infty\to\infty}+\|p_{x'}-p_x\|_1\big)
\big\|(\mathrm{Id}-P_{x'}+\mathbf1p_{x'}^\top)^{-1}\big\|_{\infty\to\infty}
\|e_{x'}\|_\infty,
\end{aligned}
\]
where the last inequality follows from
\[
\|P_x-P_{x'}+\mathbf1(p_{x'}-p_x)^\top\|_{\infty\to\infty}
\le\|P_x-P_{x'}\|_{\infty\to\infty}+\|p_{x'}-p_x\|_1.
\]
It follows that the \(F\) term is bounded by
$2C_*^2DL_Q'(L_P+L_p)\|x-x'\|$ using properties (B)--(D) and
$\|e_{x'}\|_\infty\le2DL_Q'$, as proved in (I). Combining the terms \(E\) and \(F\) gives
\(
\sup_{i\in\mathcal I}|\widehat Z_x(i)-\widehat Z_{x'}(i)|
\le L_Z\|x-x'\|.
\)

Finally, we prove that $P_x \widehat{Z}_x (i)$ is  Lipschitz continuous through
\begin{align*}
\hspace{1cm} \Big|(P_x\widehat Z_x)(i)-(P_{x'}\widehat Z_{x'})(i)\Big|
&\le\sum_j|P_x(i,j)-P_{x'}(i,j)|\,|\widehat Z_x(j)|
+\sum_jP_{x'}(i,j)|\widehat Z_x(j)-\widehat Z_{x'}(j)|\\
&\le\left(\sum_j|P_x(i,j)-P_{x'}(i,j)|\right)
\sup_j|\widehat Z_x(j)|
+\sup_j|\widehat Z_x(j)-\widehat Z_{x'}(j)|\\
&\le\bigl(L_PB_Z+L_Z\bigr)\|x-x'\|. \hspace{6.3cm} \Box
\end{align*}
\end{proof}

We are now ready to present a proof for Proposition~\ref{prop:poisson-noV}.

\begin{proof}{Proof of Proposition~\ref{prop:poisson-noV}:}
We prove the following equivalent claim:
\[
\mathbb E\!\left[-2\sum_{k=0}^{n-1}\alpha_{k+1}e_{k+1}\right]
\le
2(L_Q'+L_f)\bigl(L_PB_Z+L_Z\bigr)
\sum_{k=0}^{n-1}\alpha_{k+1}^2
+4B_Z\alpha_1.
\]
Let $\mathcal F_k:=\sigma\!\big(\{(x_t,I_t):0\le t\le k\}\big)$. Recall the Poisson equation in Lemma~\ref{prop:2} has
\(
e_{k+1}
=\widehat Z_{x_k}(I_{k+1})-(P_{x_k}\widehat Z_{x_k})(I_{k+1}).
\)
Since $I_{k+1}\sim P_{x_k}(I_k,\cdot)$ conditionally on $\mathcal F_k$, \(
\mathbb{E}\!\left[\widehat Z_{x_k}(I_{k+1})\mid \mathcal F_k\right]=(P_{x_k}\widehat Z_{x_k})(I_k).
\) The tower property gives 
\[
\mathbb E\!\left[-2\sum_{k=0}^{n-1}\alpha_{k+1}e_{k+1}\right]
=-2\sum_{k=0}^{n-1}\mathbb E\!\left[
\alpha_{k+1}\bigl((P_{x_k}\widehat Z_{x_k})(I_k)
-(P_{x_k}\widehat Z_{x_k})(I_{k+1})\bigr)\right].
\tag{I}
\]
Add and subtract $(P_{x_{k+1}}\widehat Z_{x_{k+1}})(I_{k+1})$
inside (I) and split the sum into
{\footnotesize
\[
\underbrace{-2\sum_{k=0}^{n-1}\alpha_{k+1}
\bigl((P_{x_k}\widehat Z_{x_k})(I_k)
-(P_{x_{k+1}}\widehat Z_{x_{k+1}})(I_{k+1})\bigr)}_{A}
\;\underbrace{-2\sum_{k=0}^{n-1}\alpha_{k+1}
\bigl((P_{x_{k+1}}\widehat Z_{x_{k+1}})(I_{k+1})
-(P_{x_k}\widehat Z_{x_k})(I_{k+1})\bigr)}_{B}
\tag{II}
\]
}

For the \(A\) term in (II), we shift the index to obtain that
\(
A = -2\alpha_1(P_{x_0}\widehat Z_{x_0})(I_0)
+2\alpha_n(P_{x_n}\widehat Z_{x_n})(I_n) - 2\sum_{k=1}^{n-1}(\alpha_{k+1}-\alpha_k)
(P_{x_k}\widehat Z_{x_k})(I_k).
\)
Using $|(P_x\widehat Z_x)(i)|\le B_Z$
from Lemma~\ref{prop:2} and the nonincreasing step sizes, we have
\[
\mathbb E[A]
\le2(\alpha_1+\alpha_n)B_Z
+2\sum_{k=1}^{n-1}(\alpha_k-\alpha_{k+1})B_Z
=4B_Z\alpha_1.
\tag{III}
\]
For the \(B\) sum in (II), Lemma~\ref{prop:2} gives
\(
\left|(P_{x_{k+1}}\widehat Z_{x_{k+1}})(I_{k+1})
-(P_{x_k}\widehat Z_{x_k})(I_{k+1})\right|
\le\bigl(L_PB_Z+L_Z\bigr)\|x_{k+1}-x_k\|.
\)
Now recall that \(x_{k+1}=\Pi_X\!\left(x_{k}
-\alpha_{k+1} Z_{x_{k}}(I_{k+1})\right)\). Since \(x_{k}\in X\), we have
\(\Pi_X(x_{k})=x_{k}\). Then, by the
nonexpansiveness of the projection operator, we have
\(
\|x_{k+1}-x_{k}\|_1 = 
\left\|
\Pi_X\!\left(x_{k}
-\alpha_k Z_{x_{k}}(I_{k+1})\right)
-\Pi_X(x_{k})
\right\|_1
\le
\alpha_k\|Z_{x_{k}}(I_{k+1})\|_1
\le \alpha_k(L_f+L_Q'),
\)
where we use the bound $\|Z_x(i)\|\le L_f+L_Q'$
established in Lemma~\ref{prop:L-smooth}. It follows that
\[
\mathbb E[B] \le 2 \bigl(L_PB_Z+L_Z\bigr)
\sum_{k=0}^{n-1}\alpha_{k+1}\mathbb E\|x_{k+1}-x_k\| \le 2 (L_f+L_Q')\bigl(L_PB_Z+L_Z\bigr)
\sum_{k=0}^{n-1}\alpha_{k+1}^2 
\tag{IV}
\]

Finally, combining (III) and (IV) completes the proof. \hfill \(\Box\)


\end{proof}

\subsection{Proof of Theorem~\ref{thm:continuous_convergence}}
\begin{proof}{Proof:}

By nonexpansiveness of the Euclidean projection and that $x^\star\in\mathcal X$, we have
\[
\|x_k-x^\star\|^2 \le
\|x_{k-1}-\alpha_kZ_{x_{k-1}}(I_k)-x^\star\|^2
=\|x_{k-1}-x^\star\|^2
-2\alpha_k \cdot Z_{x_{k-1}}(I_k)^{\top}(x_{k-1}-x^\star)
+\alpha_k^2\|Z_{x_{k-1}}(I_k)\|^2.
\]
Recall, from the derivation before Proposition~\ref{prop:poisson-noV}, that 
\(
Z_{x_{k-1}}(I_k)^{\top}(x_{k-1}-x^\star)
\ge F_w(x_{k-1})-F_w(x^\star)+e_k.
\)
Together with $\|Z_{x_{k-1}}(I_k)\|\le L_f+L_Q'$ from Lemma~\ref{prop:L-smooth}, this gives
\(
2 \alpha_k \bigl(F_w(x_{k-1})-F_w^\star\bigr)
\le \|x_{k-1}-x^\star\|^2-\|x_k-x^\star\|^2
-2\alpha_ke_k+\alpha_k^2(L_f+L_Q')^2.
\)
To pass to the updated iterate, note that
$\|x_k-x_{k-1}\|\le\alpha_k(L_f+L_Q')$ and $F_w$ is
$(L_f+L_Q')$-Lipschitz on $\mathcal X$ by Lemma~\ref{prop:L-smooth}. Then,
\(
F_w(x_k)-F_w(x_{k-1})
\le(L_f+L_Q')\|x_k-x_{k-1}\|
\le\alpha_k(L_f+L_Q')^2.
\)
Combining these inequalities yields
\[
2 \alpha_k \bigl(F_w(x_k)-F_w^\star\bigr)
\le \|x_{k-1}-x^\star\|^2-\|x_k-x^\star\|^2
-2\alpha_ke_k+3\alpha_k^2(L_f+L_Q')^2.
\]
Summing this inequality over $k=1,\ldots,n$ gives
\begin{align*}
2\sum_{k=1}^n\alpha_k\bigl(F_w(x_k)-F_w^\star\bigr)
&\le
\sum_{k=1}^n
\bigl(\|x_{k-1}-x^\star\|^2-\|x_k-x^\star\|^2\bigr)
-2\sum_{k=1}^n\alpha_ke_k
+3(L_f+L_Q')^2\sum_{k=1}^n\alpha_k^2\\
&=
\|x_0-x^\star\|^2-\|x_n-x^\star\|^2
-2\sum_{k=1}^n\alpha_ke_k
+3(L_f+L_Q')^2\sum_{k=1}^n\alpha_k^2.
\end{align*}
Taking expectations and dropping the nonpositive term
$-\mathbb E[\|x_n-x^\star\|^2]$, we obtain
\[
\begin{aligned}
2\sum_{k=1}^n\alpha_k\mathbb E[F_w(x_k)-F_w^\star]
&\le
\mathbb E[\|x_0-x^\star\|^2]
+\mathbb E\!\left[-2\sum_{k=1}^n\alpha_ke_k\right]
+3(L_f+L_Q')^2\sum_{k=1}^n\alpha_k^2\\
&\le
\mathbb E[\|x_0-x^\star\|^2]+C_0
+\bigl(C_1+3(L_f+L_Q')^2\bigr)\sum_{k=1}^n\alpha_k^2,
\end{aligned}
\]
where the last inequality follows from
Proposition~\ref{prop:poisson-noV}. But the convexity of $F_{w}$ and the definition of $\bar x_n$ imply that 
\(
\mathbb E\!\left[F_{w}(\bar x_n)-F_{w}(x^\star)\right]
\le
\frac{1}{\sum_{k=1}^n \alpha_{k}}\sum_{k=1}^n \alpha_{k}\,
\mathbb E\!\left[F_{w}(x_k)-F_{w}(x^\star)\right].
\)
Applying the upper bound for \(\sum_{k=1}^n\alpha_k\mathbb E[F_w(x_k)-F_w^\star]\) we just derived above produces
\[
\mathbb E\!\left[F_{w}(\bar x_n)-F_{w}(x^\star)\right]
\le \frac{\mathbb E[\|x_0-x^\star\|^2]+C_0
+\bigl(C_1+3(L_f+L'_Q)^2\bigr)\sum_{k=1}^n\alpha_k^2}{2 \sum_{k=1}^n \alpha_{k}}.
\]
For the step-size average \(\bar{x}_n^{\text{ss}} := (\sum_{k=1}^n\alpha_k x_k)/(\sum_{k=1}^n \alpha_k)\) with $\alpha_k=O(1/\sqrt{k})$, we have
$\sum_{k=1}^n\alpha_k=\Theta(\sqrt n)$ and
$\sum_{k=1}^n\alpha_k^2=O(1+\log n)$. Since $L_P=O(1/w)$, $B_Z=O(C_*)$,
$L_Z=O\!\left(C_*+\frac{C_*^2}{w}\right)$, $C_1=O\!\left(C_*+\frac{C_*^2}{w}\right)$, and $C_0=O(C_*)$, plugging these bounds into the above inequality leads to the claimed convergence rate. 

{Furthermore, for the linear averages} $\bar x_n^{\mathrm{lin}}
:=({\sum_{k=1}^n kx_k})/({\sum_{k=1}^n k}),$ multiplying the preceding
one-step inequality by $k/(2\alpha_k)$ gives
\[
k\bigl(F_w(x_k)-F_w^\star\bigr) \le \frac{k}{2\alpha_k}
\bigl(\|x_{k-1}-x^\star\|^2-\|x_k-x^\star\|^2\bigr) -ke_k+\frac32k\alpha_k(L_f+L_Q')^2. 
\]
We bound the three terms after summing over \(k\). First, since $\alpha_k=a/\sqrt{k}$, the sequence $k/\alpha_k$
is increasing. Summation by parts yields
\[
\begin{aligned}
\sum_{k=1}^n\frac{k}{2\alpha_k}
\bigl(\|x_{k-1}-x^\star\|^2-\|x_k-x^\star\|^2\bigr) &=
\frac{\|x_0-x^\star\|^2}{2\alpha_1}
-\frac{n\|x_n-x^\star\|^2}{2\alpha_n}
+\sum_{k=1}^{n-1}
\left(\frac{k+1}{2\alpha_{k+1}}-\frac{k}{2\alpha_k}\right)
\|x_k-x^\star\|^2\\
&\le
\frac{D^2}{2\alpha_1}
+D^2\sum_{k=1}^{n-1}
\left(\frac{k+1}{2\alpha_{k+1}}-\frac{k}{2\alpha_k}\right)
=\frac{D^2n}{2\alpha_n},
\end{aligned}
\]
where we used $\|x_k-x^\star\|\le D$ and dropped the
nonpositive terminal term. Second, the Poisson equation and the tower property give 
\(
\mathbb E\sum_{k=1}^n ke_k
=
\mathbb E\sum_{k=1}^n k
\bigl((P_{x_{k-1}}\widehat Z_{x_{k-1}})(I_{k-1})
-(P_{x_{k-1}}\widehat Z_{x_{k-1}})(I_k)\bigr).
\)
Shifting the index and collecting terms, we obtain
\[
\begin{aligned}
\mathbb E\sum_{k=1}^n ke_k
=\mathbb E\Bigl[
(P_{x_0}\widehat Z_{x_0})(I_0)
-n(P_{x_{n-1}}\widehat Z_{x_{n-1}})(I_n) +\sum_{k=1}^{n-1}(P_{x_{k-1}}\widehat Z_{x_{k-1}})(I_k)
\\ +\sum_{k=1}^{n-1}(k+1)
\bigl((P_{x_k}\widehat Z_{x_k})(I_k)
-(P_{x_{k-1}}\widehat Z_{x_{k-1}})(I_k)\bigr)
\Bigr].
\end{aligned}
\]
Then, Lemma~\ref{prop:2} and
$\|x_k-x_{k-1}\|\le\alpha_k(L_f+L_Q')$ imply
\(
\left|\mathbb E\sum_{k=1}^n ke_k\right|
\le2nB_Z
+(L_f+L_Q')\bigl(L_PB_Z+L_Z\bigr)
\sum_{k=1}^{n-1}(k+1)\alpha_k.
\)
Combining these bounds and applying convexity, we conclude that
\[
\begin{aligned}
\mathbb E[F_w(\bar x_n^{\mathrm{lin}})]-F_w^\star
&\le\frac{2}{n(n+1)}
\sum_{k=1}^n k\,\mathbb E[F_w(x_k)-F_w^\star]\\
&\le\frac{2}{n(n+1)}\Bigg[
\frac{D^2n}{2\alpha_n}
+2nB_Z
+\frac32(L_f+L_Q')^2\sum_{k=1}^n k\alpha_k\\
&\qquad\qquad
+(L_f+L_Q')\bigl(L_PB_Z+L_Z\bigr)
\sum_{k=1}^{n-1}(k+1)\alpha_k
\Bigg].
\end{aligned}
\]
For $\alpha_k=a/\sqrt{k}$,
$n/\alpha_n=O(n^{3/2})$,
$\sum_{k=1}^n k\alpha_k=O(n^{3/2})$, and
$\sum_{k=1}^{n-1}(k+1)\alpha_k=O(n^{3/2})$.
Using $B_Z=O(C_*)$ and
$L_PB_Z+L_Z
=O(C_*+C_*^2/w)$, we arrive at
\[
\hspace{4.5cm} \mathbb E[F_w(\bar x_n^{\mathrm{lin}})]-F_w^\star
=
O\!\left(
\left(C_*+\frac{C_*^2}{w}\right)n^{-1/2}
\right). \hspace{4.5cm} \Box
\]

\end{proof}

\subsection{Proof of Proposition~\ref{prop:adaptive-w-1over3}}

\begin{proof}{Proof of Proposition~\ref{prop:adaptive-w-1over3}:}
Fix $x^\star\in\argmin_{x\in\mathcal X}F(x)$ and define \(e_k:=e_{x_{k-1},w_{k-1}}(I_k)\) with 
\[
e_{x,w}(i):=
Q(x,\xi^{(i)})-Q(x^\star,\xi^{(i)})
-\sum_jp_{x,w}(j)
\bigl(Q(x,\xi^{(j)})-Q(x^\star,\xi^{(j)})\bigr)
\]
for all \(i \in \I\). 
Let $\widehat Z_{x,w}$ denote the corresponding centered solution to the 
Poisson equation such that \(\widehat Z_{x,w} - P_{x, w} \widehat Z_{x,w} = e_{x, w}\) and \(p_{x,w}^{\top} \widehat Z_{x,w} = 0\), and \(p_{x, w}\) denote the stationary distribution with respect to incumbent \(x\) and temperature parameter \(w\) with \(p_{x,w}(i) := {\exp\left(Q(x,\xi^{(i)})/w\right)}/
{\sum_{j\in\mathcal I}\exp\left(Q(x,\xi^{(j)})/w\right)}\). We first derive a useful lemma. Note that the Lipschitz continuity of \(p_{x, w}\) in \((x, w)\) holds for a general state space $\mathcal I$.
\begin{lemma}[Lipschitz continuity in the iterate and temperature]
\label{lem:kernel-temperature}
For $Q(\cdot, \cdot)$ satisfying the Lipschitz continuity in
Lemma~\ref{lem:recourse_lipschitz_smoothing},
suppose that the proposal kernel is fixed and the acceptance
distribution is either MH or Barker. Then, for all $x,y\in\mathcal X$ and $w,v>0$, it holds that
\[
\|p_{x,w}-p_{y,v}\|_1
\le
\frac{2L_Q'}{\max\{w,v\}}\|x-y\|
+2L_\xi'\operatorname{diam}(\Xi)
\left|\frac1w-\frac1v\right|.
\]
Moreover, the MH transition kernel satisfies
\[
\|P_{x,w}-P_{y,v}\|_{\infty\to\infty}
\le
\frac{4L_Q'}{\max\{w,v\}}\|x-y\|
+2L_\xi'\operatorname{diam}(\Xi)
\left|\frac1w-\frac1v\right|.
\]
For the Barker transition kernel, the bound improves to
\[
\|P_{x,w}-P_{y,v}\|_{\infty\to\infty}
\le
\frac{L_Q'}{\max\{w,v\}}\|x-y\|
+\frac{L_\xi'\operatorname{diam}(\Xi)}2
\left|\frac1w-\frac1v\right|.
\]
\end{lemma}

\begin{proof}{Proof of Lemma~\ref{lem:kernel-temperature}:}
We already proved in Lemma~\ref{lem:target-distribution-lipschitz} that 
\(
\|p_{x,w}-p_{y,w}\|_1 \le ({2L'_Q}/w)\|x-y\|.
\)
Fix a state $i_0 \in \I$. Since the recourse costs have bounded variation across scenarios,
differentiation under the integral is justified. Differentiating
with respect to the inverse temperature $w^{-1}$ gives
\begin{align*}
\int_{\mathcal I}
\left|\frac{\partial p_{y,w}(i)}{\partial(1/w)}\right|\nu(di) 
& =
\int_{\mathcal I}p_{y,w}(i)
\left|
Q(y,\xi^{(i)})
-\int_{\mathcal I}p_{y,w}(j)Q(y,\xi^{(j)})\,\nu(dj)
\right|\nu(di)\\
&=
\int_{\mathcal I}p_{y,w}(i)
\left|
Q(y,\xi^{(i)})-Q(y,\xi^{(i_0)})
-\int_{\mathcal I}p_{y,w}(j)
\bigl(Q(y,\xi^{(j)})-Q(y,\xi^{(i_0)})\bigr)\,\nu(dj)
\right|\nu(di)\\
&\le
\int_{\mathcal I}p_{y,w}(i)
\left|Q(y,\xi^{(i)})-Q(y,\xi^{(i_0)})\right|\,\nu(di)\\
&\qquad\quad+
\int_{\mathcal I}p_{y,w}(i)
\int_{\mathcal I}p_{y,w}(j)
|Q(y,\xi^{(j)})-Q(y,\xi^{(i_0)})|\,\nu(dj)\,\nu(di)\\
&=
2\int_{\mathcal I}p_{y,w}(i)
\left|Q(y,\xi^{(i)})-Q(y,\xi^{(i_0)})\right|\,\nu(di)\\
&\le
2\sup_{i\in\mathcal I}
\left|Q(y,\xi^{(i)})-Q(y,\xi^{(i_0)})\right|
\int_{\mathcal I}p_{y,w}(i)\,\nu(di) =
2\sup_{i\in\mathcal I}
\left|Q(y,\xi^{(i)})-Q(y,\xi^{(i_0)})\right|.
\end{align*}
Then, integrating between $1/w$ and $1/v$ yields
\[
\|p_{y,w}-p_{y,v}\|_1 \le
2\sup_{i\in\mathcal I}
|Q(y,\xi^{(i)})-Q(y,\xi^{(i_0)})| \cdot \left|\frac1w-\frac1v\right| \le
2L_\xi'\operatorname{diam}(\Xi)
\left|\frac1w-\frac1v\right|.
\]
In addition, 
$
\|p_{x,w}-p_{y,v}\|_1 \le \|p_{x,w}-p_{y,w}\|_1
+\|p_{y,w}-p_{y,v}\|_1 \leq (2L'_Q/w)\|x-y\| + 2L_\xi'\operatorname{diam}(\Xi)
\left|1/w-1/v\right|
$ and $\|p_{x,w}-p_{y,v}\|_1 \le \|p_{x,w}-p_{x,v}\|_1
+\|p_{x,v}-p_{y,v}\|_1 \leq (2L'_Q/v)\|x-y\| + 2L_\xi'\operatorname{diam}(\Xi)
\left|1/w-1/v\right|$. Thus, 
\(
\|p_{x,w}-p_{y,v}\|_1 \le
({2L_Q'}/{\max\{w,v\}})\|x-y\|
+2L_\xi'\operatorname{diam}(\Xi)
\left|1/w - 1/v\right|.
\)

We next bound the transition kernels and the proof is similar to those of Lemmas~\ref{lem:lip barker} and~\ref{lem:lipschitz-MH-nom}. Write the logarithmic acceptance ratio as $\alpha_{x,w}(i,j)
:=
\frac{Q(x,\xi^{(j)})-Q(x,\xi^{(i)})}{w}
+\log q(i\mid j)-\log q(j\mid i).$ Lipschitz continuity of $Q$ therefore gives $|\alpha_{x,w}(i,j)-\alpha_{y,v}(i,j)|
\le
({2L_Q'}/{\max\{w,v\}})\|x-y\|
+L_\xi'\operatorname{diam}(\Xi)
\left|1/w-1/v\right|.$  The remaining proof is the same, so we will omit it here.
\hfill \(\Box\)
\end{proof}

\paragraph{Proof of Proposition~\ref{prop:adaptive-w-1over3} (continued):}
Furthermore, we adapt the same resolvent argument to $\sup_i\big|
(P_{x,w}\widehat Z_{x,w})(i)
-(P_{y,v}\widehat Z_{y,v})(i)\big|$. The definition of $e_{x,w}$ and
$|Q(y,\xi^{(i)})-Q(x^\star,\xi^{(i)})|\le DL_Q'$ imply (i) $\|e_{x,w}\|_\infty\le2DL_Q'$ and (ii) $\|e_{x,w}-e_{y,v}\|_\infty \le 2 L_Q'\|x-y\|+DL_Q'\|p_{x,w}-p_{y,v}\|_1$, thus we have  for 
\begin{align*}
\widehat Z_{x,w}-\widehat Z_{y,v} = \ & (\mathrm{Id}-P_{x,w}+\mathbf1p_{x,w}^{\top})^{-1}
(e_{x,w}-e_{y,v})\\
& +\Big((\mathrm{Id}-P_{x,w}+\mathbf1p_{x,w}^{\top})^{-1} - 
(\mathrm{Id}-P_{y,v}+\mathbf1p_{y,v}^{\top})^{-1}\Big)e_{y,v} \\
= \ & (\mathrm{Id}-P_{x,w}+\mathbf1p_{x,w}^{\top})^{-1}
(e_{x,w}-e_{y,v})\\
& +(\mathrm{Id}-P_{x,w}+\mathbf1p_{x,w}^{\top})^{-1}
\bigl(P_{x,w}-P_{y,v} +\mathbf1(p_{y,v}-p_{x,w})^\top\bigr) \cdot
(\mathrm{Id}-P_{y,v}+\mathbf1p_{y,v}^{\top})^{-1}e_{y,v},
\end{align*}
where the second equality uses the resolvent identity \(A^{-1} - B^{-1} = A^{-1}(B-A)B^{-1}\) with \(A := \mathrm{Id}-P_{x,w}+\mathbf1p_{x,w}^{\top}\) and \(B := \mathrm{Id}-P_{y,v}+\mathbf1p_{y,v}^{\top}\). 
It follows from Assumption~\ref{ass:invertibility} that 
\(
\|\widehat Z_{x,w}-\widehat Z_{y,v}\|_\infty \le C_*\|e_{x,w}-e_{y,v}\|_\infty +2C_*^2DL_Q' \bigl(\|P_{x,w}-P_{y,v}\|_{\infty\to\infty}
+\|p_{x,w}-p_{y,v}\|_1\bigr).
\) 
Finally, the Poisson equation gives
$P_{x,w}\widehat Z_{x,w}=\widehat Z_{x,w}-e_{x,w}$. Hence,
\begin{align*}
& \ \sup_i\big|
(P_{x,w}\widehat Z_{x,w})(i)
-(P_{y,v}\widehat Z_{y,v})(i)\big| \\
\le \ & \ \|e_{x,w}-e_{y,v}\|_\infty +\| \widehat Z_{x,w}-\widehat Z_{y,v}\|_\infty   \\
\le \ & \  (C_*+1)\|e_{x,w}-e_{y,v}\|_\infty
+2C_*^2DL_Q'
\bigl(\|P_{x,w}-P_{y,v}\|_{\infty\to\infty}
+\|p_{x,w}-p_{y,v}\|_1\bigr)  \\
\le \ & \ 2(C_*+1)L_Q'\|x-y\|
+\bigl(C_*+1+2C_*^2\bigr)DL_Q'
  \|p_{x,w}-p_{y,v}\|_1
+2C_*^2DL_Q'
  \|P_{x,w}-P_{y,v}\|_{\infty\to\infty}
\\
\le \ & \ 
C_*^2\left[
\bigl(4L_Q'+16D(L_Q')^2\bigr)
\frac{\|x-y\|}{\min\{w,v\}}
+12DL_Q'L_\xi'\operatorname{diam}(\Xi)
\left|\frac1w-\frac1v\right|
\right],
\end{align*}
where we use $C_*\ge1$ and
$\min\{w,v\}\le1$ for sufficiently large $n$ considered. We now adapt the proof of Theorem~\ref {thm:continuous_convergence} with
$e_k:=e_{x_{k-1},w_{k-1}}(I_k)$. We obtain similarly that
\begin{align*}
 \mathbb E\sum_{k=1}^n ke_k
=\mathbb E\Bigl[
&(P_{x_0,w_0}\widehat Z_{x_0,w_0})(I_0)
-n(P_{x_{n-1},w_{n-1}}\widehat Z_{x_{n-1},w_{n-1}})(I_n) +\sum_{k=1}^{n-1}
(P_{x_{k-1},w_{k-1}}\widehat Z_{x_{k-1},w_{k-1}})(I_k)\\
&+\sum_{k=1}^{n-1}(k+1)
\bigl(
(P_{x_k,w_k}\widehat Z_{x_k,w_k})(I_k)
-(P_{x_{k-1},w_{k-1}}\widehat Z_{x_{k-1},w_{k-1}})(I_k)
\bigr)
\Bigr].   
\end{align*}
Collect the results to use. By Lemma~\ref{prop:2},
$\sup_i|(P_{x,w}\widehat Z_{x,w})(i)|=\mathcal O(C_*)$ and we just proved a bound for $\sup_i\big|
(P_{x,w}\widehat Z_{x,w})(i)
-(P_{y,v}\widehat Z_{y,v})(i)\big| $ :
\[
\sup_i\bigl|
(P_{x_k,w_k}\widehat Z_{x_k,w_k})(i)
-(P_{x_{k-1},w_{k-1}}\widehat Z_{x_{k-1},w_{k-1}})(i)
\bigr| = \mathcal O\!\left(
C_*^2\left[
\frac{\|x_k-x_{k-1}\|}{w_k}
+\left|\frac1{w_k}-\frac1{w_{k-1}}\right|
\right]\right),
\]
where we also use $w_k\le w_{k-1}$. It follows from $\|x_k-x_{k-1}\|\le\alpha_k(L_f+L_Q')$ that 
\[
\left|\mathbb E\sum_{k=1}^nke_k\right| = 2n \mathcal O(C_*) + \mathcal O\!\left(
C_*^2\sum_{k=1}^{n-1}(k+1)
\left[
\frac{\alpha_k}{w_k}
+\left|\frac1{w_k}-\frac1{w_{k-1}}\right|
\right]\right) 
\]
We next account for the approximation error. {\color{black}  
By the definition of $e_k$ and the
subgradient inequality, we have 
\begin{align*}
Z_{x_{k-1}}(I_k)^\top(x_{k-1}-x^\star)
&\ge
f(x_{k-1})-f(x^\star)
+\sum_jp_{x_{k-1},w_{k-1}}(j)
\bigl(Q(x_{k-1},\xi^{(j)})-Q(x^\star,\xi^{(j)})\bigr)+e_k\\
&\ge
F(x_{k-1})-F^\star
-\left[
\max_i Q(x_{k-1},\xi^{(i)})
-\sum_jp_{x_{k-1},w_{k-1}}(j)
Q(x_{k-1},\xi^{(j)})
\right]+e_k .
\end{align*}
Furthermore, because $\Delta_* \leq \widetilde \Delta_k$ for all sufficiently large $k$, there exists a \(k_0 \in \mathbb{N}_+\) such that $p_{x_{k-1},w_{k-1}}
\bigl(\mathcal I\setminus I^\star(x_{k-1})\bigr)
\le
\frac{|\mathcal I|-|I^\star(x_{k-1})|}
{|I^\star(x_{k-1})|}
e^{-\Delta_\star/w_{k-1}}
\le
|\mathcal I|e^{-\Delta_\star/w_{k-1}}$ for all \(k \geq k_0\). Recall that
$\Delta:=\sup_x\max_{i,j}|Q(x,\xi^{(i)})-Q(x,\xi^{(j)})|<\infty$, so for finitely many $k<k_0$, we use the crude bound $\max_i Q(x_{k-1},\xi^{(i)})
-\sum_j p_{x_{k-1},w_{k-1}}(j)Q(x_{k-1},\xi^{(j)})
\le \Delta,$ while for all $k \ge k_0$, 
we have $Z_{x_{k-1}}(I_k)^\top(x_{k-1}-x^\star)
\ge
F(x_{k-1})-F^\star
-\Delta|\mathcal I|e^{-\Delta_\star/w_{k-1}}+e_k$.  Following the projection argument as in the proof of 
Theorem~\ref{thm:continuous_convergence}, for $k\ge k_0$, we obtain
\[
2\alpha_k(F(x_k)-F^\star)
\le
\|x_{k-1}-x^\star\|^2-\|x_k-x^\star\|^2
+3\alpha_k^2(L_f+L_Q')^2
-2\alpha_ke_k
+2\alpha_k\Delta|\mathcal I|e^{-\Delta_\star/w_{k-1}}.
\]
For the finitely many $k<k_0$, the same argument with the crude bound
above contributes only a constant ${O}(1)$ independent of $n$. 
}
We multiply the above inequality by $k/(2\alpha_k)$ on both sides and apply
convexity. In the proof of Theorem~\ref{thm:continuous_convergence}, we have shown that $
\sum_{k=1}^n({k}/({2\alpha_k}))
\bigl(\|x_{k-1}-x^\star\|^2-\|x_k-x^\star\|^2\bigr) \le
({D^2n})/(2\alpha_n)$. Therefore, 
{\color{black}
\begin{align*}
\mathbb E[F(\bar x_n)]-F^\star
&\le
\frac{2}{n(n+1)}
\sum_{k=1}^n k\,\mathbb E[F(x_k)-F^\star]\\
&\le
\frac{2}{n(n+1)}
\Bigg[
\frac{D^2n}{2\alpha_n}
+\frac32(L_f+L_Q')^2\sum_{k=1}^n k\alpha_k
+\left|\mathbb E\sum_{k=1}^nke_k\right|
+\Delta|\mathcal I|
\sum_{k=k_0}^n k e^{-\Delta_\star/w_{k-1}} + {O}(1) 
\Bigg].
\end{align*}
Now let $\alpha_k=a/\sqrt{k}$ and
$w_k=(2\Delta_\star)/
(\log(k+2)+2\log|\mathcal I|)$. Then
$\alpha_k/w_k
=O(k^{-1/2}(\log k+\log|\mathcal I|))$ and
$|1/w_k-1/w_{k-1}|=O(k^{-1})$. Hence,
\[
\left|\mathbb E\sum_{k=1}^nke_k\right|
=
O\left(
C_*n+C_*^2n^{3/2}
(\log n+\log|\mathcal I|)
\right).
\]
Moreover,
$|\mathcal I|e^{-\Delta_\star/w_{k-1}}
=
(k+1)^{-1/2},$ so
$\Delta|\mathcal I|
\sum_{k=k_0}^n k e^{-\Delta_\star/w_{k-1}}
=O(n^{3/2})$.
Since $n/\alpha_n=O(n^{3/2})$ and
$\sum_{k=1}^n k\alpha_k=O(n^{3/2})$, it follows that
\[
\hspace{1cm}
0\le
\mathbb E[F(\bar x_n)]-F^\star
\le
O\left(
\frac{1+C_*^2(\log n+\log|\mathcal I|)}{\sqrt n}
+\frac{C_*}{n}
\right)
=
O\left(
\frac{1+C_*^2\log(n|\mathcal I|)}{\sqrt n}
\right). \hspace{1cm} \Box
\]

}

\end{proof}

\subsection{General Mellowmax Approximation} \label{apx-prop:continuous-mellowmax}
\begin{proposition}\label{prop:continuous-mellowmax}
Denote by $(\mathcal I,d_{\mathcal I})$ a compact metric space with diameter \(D_{\I} := \sup\{d_{\I}(I,J): I, J \in \I \}\) for the states. Suppose that there exists a constant $L_\I<\infty$, such that
\(
\left|
\widehat Q_x(I) - \widehat Q_x(J)
\right|
\leq
L_\I \, d_{\I}(I,J)
\)
for all \(x \in \X\) and \(I, J \in \I\). Suppose further that
there exist constants $c_\nu>0$ and $d_I>0$ such that
\(
\nu\bigl(B_{\mathcal I}(I,r)\bigr)
\ge
c_\nu\left({r}/{D_\I}\right)^{d_I}
\) for all \(I \in \I\) and \(0<r\le D_\I\). 
Then, for every $x\in\mathcal X$, it holds that
\[
0
\le
\sup_{I\in\mathcal I}\widehat Q(x,\xi_I)
-
\mathcal M_{w}^{\nu}(x)
\le
\inf_{0<r\le D_\I}
\left\{
L_Ir
+
w\log\frac{1}{c_\nu}
+
d_Iw\log\frac{D_\I}{r}
\right\}.
\]
In particular, if $d_Iw/L_\I\le D_\I$ (which always holds for a sufficiently small $w$), then
\[
\sup_{I\in\mathcal I}\widehat  Q(x,\xi_I)
-
\mathcal M_{w}^{\nu}(x)
\le
w \left ( \log\frac{1}{c_\nu}
+
d_I
\left[
1+\log\left(
\frac{L_\I D_\I}{d_I w}
\right)
\right] \right).
\]
\end{proposition}

\begin{proof}{Proof:}
We fix $x \in \mathcal X$ and write $Q^\star(x) = \sup_{I\in\mathcal I}\widehat  Q(x,\xi_I). $ Since $\mathcal I$ is compact and $I\mapsto \widehat  Q(x,\xi_I)$ is Lipschitz continuous, there exists an $I_x^\star\in\mathcal I$ that attains the supremum. Now because $\nu$ is a probability measure, we have 
\(
\mathcal M_{w}^{\nu}(x)
=
w\log
\int_{\mathcal I}
\exp\left(
\widehat  Q(x,\xi_I)/w
\right)\nu(dI)
\le
w\log
\int_{\mathcal I}
\exp\left(
{Q^\star(x)}/{w}
\right)\nu(dI)
=
Q^\star(x).
\)
This proves that  
\(
Q^\star(x)-\mathcal M_{w}^{\nu}(x)\ge0.
\)

Next, we fix any $r \in (0,D_\I]$. For every
$I\in B_{\mathcal I}(I_x^\star,r)$, the Lipschitz property gives
\(
\widehat  Q(x,\xi_I)
\ge
\widehat  Q(x,\xi_{I_x^\star})
-
L_\I d_{\mathcal I}(I,I_x^\star)
\ge
Q^\star(x)-L_\I r.
\)
It follows that
\[
\begin{aligned}
\int_{\mathcal I}
\exp\left(
\frac{\widehat  Q(x,\xi_I)}{w}
\right)\nu(dI)
&\ge
\int_{B_{\mathcal I}(I_x^\star,r)}
\exp\left(
\frac{\widehat  Q(x,\xi_I)}{w}
\right)\nu(dI) \ge 
\exp\left(
\frac{Q^\star(x)-L_\I r}{w}
\right)
\nu\bigl(B_{\mathcal I}(I_x^\star,r)\bigr).
\end{aligned}
\]
Taking logarithms and multiplying by $w$ on both sides yield
\(
\mathcal M_{w}^{\nu}(x)
\ge
Q^\star(x)-L_\I r
+
w\log
\nu\bigl(B_{\mathcal I}(I_x^\star,r)\bigr).
\)
But 
\(
w\log
\nu\bigl(B_{\mathcal I}(I_x^\star,r)\bigr)
\ge
w\log\left[
c_\nu
({r}/{D_\I})^{d_I}
\right]
=
-w\log({1}/{c_\nu})
-
d_Iw\log({D_\I}/{r}).
\)
It follows that 
\(
Q^\star(x)
-
\mathcal M_{w}^{\nu}(x)
\le
L_Ir
+
w\log({1}/{c_\nu})
+
d_Iw\log({D_\I}/{r}).
\)
Since this inequality holds for every $r\in(0,D_\I]$, taking the infimum over $r$ proves the first claim.

For $L_\I>0$, we find a smallest upper bound by choosing \(r\). Under the condition $d_Iw/L_\I\le D_\I$, the unconstrained minimizer \(r^\star={d_Iw}/{L_\I}\) is feasible. 
Substitution gives the desired result. \hfill \(\Box\)
\end{proof}


\subsection{Lower Bound of \(\|dp_x^w/d\nu\|_\infty^{-1}\) in Example~\ref{ex:uniform-mixing}} \label{apx-ex:uniform-mixing}

Recall that, because \(\I = \mathbb{S}^{d_\xi - 1}\), we take $B$ as $B_{\mathbb S}(I_x^\star,\theta):=
\{x\in \mathbb S^{d_\xi-1}:\angle(x,e_1)\le \theta\}$, where 
\(I_x^\star\in\argmax_I\widehat{Q}(I)\) and we assume WLOG that \(I_x^\star = e_1\) because \(\I\) is rotationally invariant. Because
\(I\mapsto\widehat Q(I)\) is
\(L'_\xi\|A\|_{2\to2}\)--Lipschitz, for every
\(J\in B_{\mathbb S}(I_x^\star,\theta)\), it holds that 
\(
\widehat Q(J)
\ge
\widehat Q(I_x^\star)
-
\left(2 \sin \frac{\theta}{2}\right) L'_\xi\|A\|_{2\to2}
\ge
\widehat Q(I_x^\star)
-
L'_\xi\|A\|_{2\to2} \theta.
\)
Then, 
\(
\left\|
{dp_x^w}/{d\nu}
\right\|_\infty^{-1}
=
\int_{\mathcal I}
\exp\left\{
(
\widehat Q(J)
-
\widehat Q(I_x^\star)
)/{w}
\right\}\nu(dJ)
\ge
\nu\!\left(B_{\mathbb S}(I_x^\star,\theta)\right)
\exp\left\{
-{L'_\xi\|A\|_{2\to2}\theta}/{w}
\right\}.
\)
We claim that the uniform measure of a spherical cap satisfies
\[
\nu\!\left(B_{\mathbb S}(I^\star_x,\theta)\right)
\ge
c_{d_\xi}\theta^{d_\xi-1}
\qquad
\forall \, 0<\theta\le\frac{\pi}{2},
\]
where $c_{d_\xi}>0$ depends only on $d_\xi$.

\begin{proof}{Proof of the claim:}
Every $x\in\mathbb S^{d_\xi-1}$ can be written as \(x=(\cos t,\sin t\,\omega),
t\in[0,\pi], \omega\in\mathbb S^{d_\xi-2},\) where $t=\angle(x,e_1)$. In these spherical coordinates, the surface
element is \(d\sigma_{d_\xi-1}(x)
=
\sin^{d_\xi-2}(t)\,dt\,d\sigma_{d_\xi-2}(\omega).\) Because $\nu$ is uniform on \(\I\),
\[
\nu(B_S(I^\star_x,\theta))
=
\frac{
\int_{\mathbb S^{d_\xi-2}}\int_0^\theta
\sin^{d_\xi-2}(t)\,dt\,d\sigma_{d_\xi-2}(\omega)
}{
\int_{\mathbb S^{d_\xi-2}}\int_0^\pi
\sin^{d_\xi-2}(t)\,dt\,d\sigma_{d_\xi-2}(\omega)
}
=
\frac{\int_0^\theta \sin^{d_\xi-2}(t)\,dt}
{\int_0^\pi \sin^{d_\xi-2}(t)\,dt}.
\]
For $0\le t\le \pi/2$, we have
\(
\sin t\ge \frac{2}{\pi}t.
\) Hence, for $0<\theta\le \pi/2$,
\[
\begin{aligned}
\nu(B_S(I^\star_x,\theta))
&\ge
\frac{1}{\int_0^\pi \sin^{d_\xi-2}(t)\,dt}
\int_0^\theta
\left(\frac{2t}{\pi}\right)^{d_\xi-2}\,dt \\
&=
\frac{(2/\pi)^{d_\xi-2}}
{(d_\xi-1)\int_0^\pi \sin^{d_\xi-2}(t)\,dt}
\,\theta^{d_\xi-1}.
\end{aligned}
\]
This proves the claim. \hfill \(\Box\)
\end{proof}

We now optimize the choice of $\theta \in (0,\pi/2]$ to maximize the lower bound \(c_{d_\xi} \theta^{d_\xi - 1} \cdot 
\exp\{-{L'_\xi\|A\|_{2\to2}\theta}/{w}\}\). Note that this maximum is attained at 
\(
\theta
=
\frac{(d_\xi-1)w}
{L_\xi\|A\|_{2\to2}}
\)
, whenever
\(
w
\le
\frac{\pi L_\xi\|A\|_{2\to2}}
{2(d_\xi-1)}.
\)
Therefore, the largest lower bound for \(\|dp_x^w/d\nu\|_\infty^{-1}\) is
\[
\hspace{1.5cm} \frac{\left(\frac{2}{\pi}(d_\xi-1)\right)^{d_\xi-2}\exp\{-(d_\xi-1)\}}
{\int_0^\pi \sin^{d_\xi-2}(t)\,dt} \cdot 
\left(
\frac{w}{L_\xi\|A\|_{2\to2}}
\right)^{d_\xi-1} \ =: \ C_{d_\xi} \, \left(
\frac{w}{L_\xi\|A\|_{2\to2}}
\right)^{d_\xi-1}. \hspace{1.5cm} \Box
\]

\subsection{Proof of Theorem~\ref{thm:1}} \label{apx-thm:1}

We first prove a lemma that characterizes the convergence of the law of \(I_n\) to the invariant distribution \(p^w_x\).
\begin{lemma}\label{lem:tv_bound} Under Assumptions~\ref{ass:general-kernel-lipschitz} and~\ref{ass:uniform-mixing}, 
consider the incumbent sequence $\{x_n: n\in \mathbb{N}_+\}$ on a measurable space $(\mathcal{X},\mathcal{B}_{\mathcal{X}})$
with respect to the natural filtration $\{\mathcal{F}_n: n\in \mathbb{N}_+\}$ generated by $(\mathcal{I}_m, X_m)_{m\le n}$:
\[
X_0, I_0  \to I_1 \to X_1 \to I_2 \to X_2 \to \cdots
\]
Conditional on $\mathcal{F}_n$, the next state $\mathcal{I}_{n+1}$ is sampled from $\mathcal{I}_n$ through the one-step Markov kernel $P_{x_n}$, which admits an invariant distribution $p^w_{x_n}$ on $\mathcal{I}$.

Fix $p > 0$ and suppose that the incumbent movement satisfies 
\(
\|X_n-X_{n-1}\|\le c\alpha_n
\)
almost surely, where $\{\alpha_n: n \in \mathbb{N}_+\}$ satisfies $\alpha_n\le C_\alpha n^{-p}$ for some $C_{\alpha} > 0$.  Then, for every sufficiently large $n$ such that
$f(n) := \lceil (p+1)\log n/|\log\bar\eta|\rceil<n/2$ , we have
\[
\left\|
\mathcal L\!
\left(
I_n\mid
\mathcal F_{n-f(n)}
\right)
-p_{X_{n-f(n)}}^w
\right\|_{\mathrm{TV}} 
\le
\bar\eta^{f(n)}
+\left(\frac{2^pcL_PC_\alpha}{1-\bar\eta}\right)
f(n) n^{-p}
=O\left(\frac{L_P}{1-\bar\eta} n^{-p}\log n\right).
\]
\end{lemma}

\begin{proof}{Proof of Lemma~\ref{lem:tv_bound}:}
Fix a sufficiently large $n$ and for $k \ge n-f(n)$,  we write
\(
\mu_k:=\mathcal L(I_k\mid\mathcal F_{n-f(n)}).
\) 
  Insert the law of a chain that starts from the state
$I_{n-f(n)}$ and then applies the frozen kernel $P_{X_{n-f(n)}}$ for
$f(n)$ steps. The triangle inequality gives 
\begin{align*}
&\left\|
\mathcal L\!\left(I_n\mid\mathcal F_{n-f(n)}\right)
-p_{X_{n-f(n)}}^w
\right\|_{\mathrm{TV}}
\\
&\quad\le
\underbrace{\left\|
\delta_{I_{n-f(n)}}P_{X_{n-f(n)}}^{\,f(n)}
-p_{X_{n-f(n)}}^w
\right\|_{\mathrm{TV}}}_{\text{(A) mixing error of the frozen chain}}
+
\underbrace{\left\|
\mathcal L\!\left(I_n\mid\mathcal F_{n-f(n)}\right)
-\delta_{I_{n-f(n)}}P_{X_{n-f(n)}}^{\,f(n)}
\right\|_{\mathrm{TV}}}_{\text{(B) actual endogenous chain versus frozen chain}}
\end{align*}
{The term (A) measures the 
mixing error and can be bounded as follows}. By definition of the invariant distribution, we have $p_{X_{n-f(n)}}^wP_{X_{n-f(n)}}=p_{X_{n-f(n)}}^w$, where $P_{X_{n-f(n)}}$ is understood as a transition Markov operator defined by $(\mu P_{X_{n-f(n)}})(A)=\int_\I P_{X_{n-f(n)}}(I,A) \mu (dI)$ for any measure \(\mu\). Then, applying the Dobrushin
contraction inequality from Assumption~\ref{ass:uniform-mixing} for $f(n)$ times yields
\[
\left\|
\delta_{I_{n-f(n)}}P_{X_{n-f(n)}}^{\,f(n)}
-p_{X_{n-f(n)}}^w
\right\|_{\mathrm{TV}}
\le
\bar\eta^{\,f(n)}
\left\|\delta_{I_{n-f(n)}}-p_{X_{n-f(n)}}^w\right\|_{\mathrm{TV}}
\le \bar\eta^{\,f(n)}.
\]

The term (B) measures 
the perturbation caused by using the changing kernels
$P_{X_{n-f(n)}},\ldots,P_{X_{n-1}}$ instead of repeatedly using the frozen
kernel $P_{X_{n-f(n)}}$; its control requires a one-step comparison followed
by an induction. 

\textbf{(I) one-step freezing error}. 
For every $k\ge n - f(n)$, we claim that 
\[
\|\mu_{k+1}-\mu_kP_{X_{n - f(n)}}\|_{\mathrm{TV}}
\le
cL_P\sum_{j=n - f(n)+1}^{k}\alpha_j.
\]

\begin{proof}{Proof of Claim \emph{(I)}:}
Fix a Borel set $B\subseteq\mathcal I$. Because $\mathcal F_{n - f(n)}\subseteq\mathcal F_k$, the tower property gives 
\begin{align*}
\mu_{k+1}(B)
&=\mathbb E\!\left[
\mathbf 1{\{I_{k+1}\in B\}}\mid\mathcal F_{n - f(n)}
\right]\\
&=\mathbb E\!\left[
\mathbb E\!\left[
\mathbf 1{\{I_{k+1}\in B\}}\mid\mathcal F_k
\right]\middle|\mathcal F_{n - f(n)}
\right]\\
&=\mathbb E\!\left[P_{X_k}(I_k,B)\mid\mathcal F_{n - f(n)}\right].
\end{align*}
Applying the frozen kernel
$P_{X_{n - f(n)}}$ to $\mu_k$ gives
\(
(\mu_kP_{X_{n - f(n)}})(B)
=\mathbb E\!\big[P_{X_{n - f(n)}}(I_k,B)\mid\mathcal F_{n - f(n)}\big].
\)
Subtracting this from the previous representation gives
\begin{align*}
\left|\mu_{k+1}(B)-(\mu_kP_{X_{n-f(n)}})(B)\right|
&=
\left|
\mathbb E\!\left[
P_{X_k}(I_k,B)-P_{X_{n-f(n)}}(I_k,B)
\,\middle|\,\mathcal F_{n - f(n)}
\right]
\right| \\
&\le
\mathbb E\!\left[
\left|
P_{X_k}(I_k,B)-P_{X_{n-f(n)}}(I_k,B)
\right|
\,\middle|\,\mathcal F_{n - f(n)}
\right] \\
&\le
\mathbb E\!\left[
\left\|
P_{X_k}(I_k,\cdot)-P_{X_{n-f(n)}}(I_k,\cdot)
\right\|_{\mathrm{TV}}
\,\middle|\,\mathcal F_{n - f(n)}
\right],
\end{align*}
where the first inequality follows from conditional Jensen's inequality and taking the supremum over all Borel sets $B$ yields the second inequality.
Taking the supremum over all Borel sets $B$ yields
\[
\left\|
\mu_{k+1}-\mu_kP_{X_{n-f(n)}}
\right\|_{\mathrm{TV}}
\le
\mathbb E\!\left[
\left\|
P_{X_k}(I_k,\cdot)-P_{X_{n-f(n)}}(I_k,\cdot)
\right\|_{\mathrm{TV}}
\,\middle|\,\mathcal F_{n - f(n)}
\right].
\]
Now we use the Lipschitz property with constant $L_P$ from Assumption~\ref{ass:general-kernel-lipschitz} to obtain
\[
\left\|
\mu_{k+1}-\mu_kP_{X_{n-f(n)}}
\right\|_{\mathrm{TV}} \le L_P\mathbb E\!\left[\|X_k-X_{n-f(n)} \|\mid\mathcal F_{n - f(n)} \right].
\]
Last, we can bound $\|X_k-X_{n - f(n)}\| $ by
\(
\|X_k-X_{n - f(n)}\|
\le\sum_{j=n - f(n)+1}^{k}\|X_j-X_{j-1}\| \le c\sum_{j=n - f(n)+1}^{k}\alpha_j.
\)
Substituting it back to the previous line proves the claim. \hfill \(\Box\)
\end{proof}

\textbf{(II) Propagation of the one-step errors by induction}. We next claim that, for every integer $m$ with $n - f(n)\le m\le n$, it holds that
\[
\left\| \mu_m-\delta_{I_{n - f(n)}}P_{X_{n - f(n)}}^{\,m-(n - f(n))}
\right\|_{\mathrm{TV}} 
\le
cL_P\sum_{k=n - f(n)}^{m-1}\bar\eta^{\,m-1-k}
\sum_{j=n - f(n)+1}^{k}\alpha_j.
\]

\begin{proof}{Proof of Claim \emph{(II)}:}
We prove this by \textbf{induction} on $m$. For $m= n- f(n)$, the left hand side is zero as \(\mu_{n - f(n)} = \delta_{I_{n - f(n)}} \). The right-hand side is non-negative so the \textit{base claim} holds. 

\textit{Induction step}. Now suppose that the claim holds for some $m \in \{n-f(n),\ldots,n-1\}$ and we can write by the triangle inequality
{\footnotesize
\[
\left\|\mu_{m+1}-\delta_{I_{n-f(n)}}P_{X_{n-f(n)}}^{\,m+1-(n-f(n))}
\right\|_{\mathrm{TV}} 
\le
\left\|\mu_{m+1}- \mu_m P_{X_{n-f(n)}} \right\|_{\mathrm{TV}}
+\left\|
\mu_m P_{X_{n-f(n)}} -\delta_{I_{n-f(n)}} P_{X_{n-f(n)}}^{\,m+1-(n-f(n))}
\right\|_{\mathrm{TV}}.
\]
}
The first term of RHS can be bounded by {Claim (I)} above and the second term can be bounded by applying the Dobrushin contraction once under the common frozen kernel $P_{X_{n-f(n)}}$. Then,
\[
\left\|\mu_{m+1}-\delta_{I_{n-f(n)}}P_{X_{n-f(n)}}^{\,m+1-(n-f(n))}
\right\|_{\mathrm{TV}} \le c L_P \sum_{j=n - f(n)+1}^{m}\alpha_j +  \bar{\eta} \left\| \mu_m-\delta_{I_{n - f(n)}}P_{X_{n - f(n)}}^{\,m-(n - f(n))}
\right\|_{\mathrm{TV}}.
\]
Now by the induction assumption, this can be further bounded by
\begin{align*}
\left\|\mu_{m+1}-\delta_{I_{n-f(n)}}P_{X_{n-f(n)}}^{\,m+1-(n-f(n))}
\right\|_{\mathrm{TV}} \le \ & cL_P\sum_{j=n - f(n)+1}^{m}\alpha_j
+  
cL_P\bar\eta  \sum_{k=n - f(n)}^{m-1}\bar\eta^{\,m-1-k}
\sum_{j=n - f(n)+1}^{k}\alpha_j \\ 
= \ & cL_P\sum_{k=n - f(n)}^{m}\bar\eta^{\,m-k}
\sum_{j=n - f(n)+1}^{k}\alpha_j.
\end{align*}
This is exactly the claim with $m + 1$ in place of $m$, so the induction step is complete. \hfill \(\Box\)
\end{proof}

\textbf{(III) Taking $m=n$ in Claim~(II) } gives
\[
\left\|
\mathcal L\!\left(I_n\mid\mathcal F_{n-f(n)}\right)
-\delta_{I_{n-f(n)}}P_{X_{n-f(n)}}^{\,f(n)}
\right\|_{\mathrm{TV}} \le c L_P 
\sum_{k=n - f(n)}^{n-1}\bar\eta^{\,n-1-k}
\sum_{j=n - f(n)+1}^{k}\alpha_j \le \frac{c L_P}{1 - \bar\eta} \sum_{j= n - f(n)+1}^{n-1} \alpha_j.
\]
Combining (I)--(III) produces
\[
\left\|
\mathcal L\!\left(I_n\mid\mathcal F_{n-f(n)}\right)
-p_{X_{n-f(n)}}^w
\right\|_{\mathrm{TV}} \le 
\bar\eta^{\,f(n)}
+
\frac{c L_P}{1 - \bar\eta} \sum_{j= n - f(n)+1}^{n-1} \alpha_j.
\]
Finally, we recall that $f(n)$ is $O(\log n)$ and $f(n) < n/2$ for a sufficiently large $n$. Therefore, 
\(
\sum_{j=n-f(n)+1}^{n-1}\alpha_j
\le C_\alpha f(n)(n/2)^{-p}.
\)
This completes the proof. \hfill \(\Box\)
\end{proof}

Now we are ready to prove Theorem~\ref{thm:1}.

\begin{proof}{Proof of Theorem~\ref{thm:1}:}
\textbf{We first compare the set of adversarial scenarios with respect to $X_{n-f(n)}$ and $X_{n-1}$}. Specifically, we claim that
\begin{equation}
\widehat I_{(\widetilde\Delta-2\varepsilon)/2}(X_{n-f(n)})
\subseteq
\widehat I_{\widetilde\Delta-2\varepsilon}(X_{n-1})
\subseteq
I_{\widetilde\Delta}(X_{n-1}). \label{claim-inclusion} \tag{Inclusion}
\end{equation}
\begin{proof}{Proof of the claim:}

For the first inclusion, we take any $I \in \widehat I_{(\widetilde\Delta-2\varepsilon)/2} (X_{n-f(n)})$ and notice that
\begin{align*}
\widehat Q_{X_{n-1}}(I) 
&\ge \widehat Q_{X_{n-f(n)}} (I) - L_{Q'}\|X_{n-1}- X_{n-f(n)} \| 
\\ 
&\ge \underbrace{\widehat Q^\star (X_{n-f(n)} ) - \frac{\widetilde\Delta-2\varepsilon}{2}}_{\text{from} \ I \in \widehat I_{(\widetilde\Delta-2\varepsilon)/2} (X_{n-f(n)}) }  
- L_{Q'}\|X_{n-1}- X_{n-f(n)} \|   
\\
& \ge \widehat Q^\star (X_{n-1} ) - \frac{\widetilde\Delta-2\varepsilon}{2} - 2 L_{Q'}\|X_{n-1}- X_{n-f(n)} \|,  
\end{align*}
where the last inequality follows from 
\(
|\widehat Q^\star(x)-\widehat Q^\star(y)|
\le L_{ Q'}\|x-y\|.
\)
Indeed, Assumption~\ref{ass:lipschitz-scores} implies that, for all \(I \in \I\), $\widehat Q_x(I)\le\widehat Q_y(I)
+L_{ Q'}\|x-y\| \le  \widehat Q^\star (y)
+L_{ Q'}\|x-y\|$. Taking the supremum over $I$ gives
one direction of this inequality, and interchanging $x$ and $y$ gives the other direction. Now recall that we have proved in Lemma~\ref{lem:tv_bound} that
\(
\|X_{n-1}-X_{n-f(n)}\|  \le
2^pcC_\alpha
f(n) n^{-p} .
\)
Because \(n\) is sufficiently large with \(2^pcC_\alpha f(n) n^{-p} \leq (\widetilde \Delta -2 \varepsilon)/(4L'_Q)\), we have \(\widehat Q_{X_{n-1}}(I) \geq \widehat Q^\star(X_{n-1}) - (\widetilde \Delta - 2\varepsilon)\) and so \(I \in \widehat I_{\widetilde\Delta-2\varepsilon}(X_{n-1})\).

For the second inclusion, we recall that \(\widehat Q^\star(x)\ge Q^\star(x)-\varepsilon.\)  Therefore, if
$I\in\widehat I_{\widetilde\Delta-2\varepsilon}(x)$ then
\(
Q(x,I) \ge \widehat Q_x(I) - \varepsilon \ge \widehat Q^\star (x) - ( \widetilde\Delta-2\varepsilon)-\varepsilon \ge Q^\star (x) - \widetilde\Delta,
\)
implying that \(I \in I_{\widetilde\Delta}(X_{n-1})\). \hfill \(\Box\)
\end{proof}

\textbf{Next, we apply the set inclusion claim to bound the target probability.}
The first inclusion in~\eqref{claim-inclusion} implies that
\(
\left\{I_n\notin I_{\widetilde\Delta}(X_{n-1})\right\}
\subseteq
\left\{
I_n\notin
\widehat I_{(\widetilde\Delta-2\varepsilon)/2}
(X_{n-f(n)})
\right\}.
\)
Taking conditional probabilities given filtration $\mathcal{F}_{n-f(n)}$ gives
\begin{align*}
\mathbb P \left\{I_n\notin I_{\widetilde\Delta}(X_{n-1}) \mid\mathcal {F}_{n-f(n)} 
\right\}
&\le
\mathbb P \left\{
I_n\notin
\widehat I_{(\widetilde\Delta-2\varepsilon)/2}
(X_{n-f(n)})
 \mid\mathcal {F}_{n-f(n)} 
\right\}
\\ 
& = 
\mathcal L(I_n\mid \mathcal F_{n-f(n)}  )
\left(
\widehat I_{(\widetilde\Delta-2\varepsilon)/2} 
(X_{n-f(n)})^c
\right).
\end{align*}
Recall from Lemma~\ref{lem:tv_bound} that
\(
\|
\mathcal L\!
\left(
I_n\mid
\mathcal F_{n-f(n)}
\right)
-p_{X_{n-f(n)}}^w
\|_{\mathrm{TV}} 
\le
\bar\eta^{f(n)}
+({(2^pcL_PC_\alpha)}/{(1-\bar\eta)})
f(n) n^{-p}.
\)
Using \(\mu(B)
\le
\pi(B)+\|\mu-\pi\|_{\mathrm{TV}}\) with $B = \widehat I_{(\widetilde\Delta-2\varepsilon)/2} 
(X_{n-f(n)})^c$, we obtain
\[
\mathbb P\left(
I_n\notin I_{\widetilde\Delta}(X_{n-1})
\mid \mathcal F_{n-f(n)}
\right )
\le 
p_{X_{n-f(n)}}^w 
\left(
\widehat I_{(\widetilde\Delta-2\varepsilon)/2} 
(X_{n-f(n)})^c
\right)
+
\bar\eta^{f(n)}
+\frac{2^pcL_PC_\alpha}{1-\bar\eta}
f(n) n^{-p}.
\]

To bound the first term on the right-hand side, we observe that
\[
p_{X_{n-f(n)}}^w 
\left(
\widehat I_{(\widetilde\Delta-2\varepsilon)/2} 
(X_{n-f(n)})^c
\right)=
 \frac{ \int_{\widehat I_{(\widetilde\Delta-2\varepsilon)/2} 
\left(X_{n-f(n)}\right)^c } \,
\exp\!\bigl\{\widehat{Q}_{X_{n-f(n)}}(I)/w\bigr\} d\nu(I)
}{
\int_{\mathcal I}
\exp\!\bigl\{\widehat{Q}_{X_{n-f(n)}} (J)/w\bigr\}\,d\nu(J)
}.
\]
Now
{\footnotesize
\begin{align*}
 \int_{\widehat I_{(\widetilde\Delta-2\varepsilon)/2} 
\left(X_{n-f(n)}\right)^c }
\exp\!\left\{\frac{\widehat{Q}_{X_{n-f(n)}}(I)}{w}\right\} d\nu(I) & \le \nu\left( \widehat I_{(\widetilde\Delta-2\varepsilon)/2} 
(X_{n-f(n)})^c \right)   \sup_{I \notin\widehat I_{(\widetilde\Delta-2\varepsilon)/2} 
(x) } \exp\left\{\frac{\widehat Q_x (I)}{w}\right\}
\\ 
&\le 
\nu\left( \widehat I_{(\widetilde\Delta-2\varepsilon)/2} 
(X_{n-f(n)})^c \right)  \exp\left\{\frac{\widehat Q^\star  (X_{n-f(n)}) - (\widetilde\Delta-2\varepsilon)/2 }{w}\right\}
\end{align*}
}
and
{\footnotesize
\begin{align*}
 \int_{\mathcal I}
\exp\!\left\{\frac{\widehat{Q}_{X_{n-f(n)}} (J)}{w}\right\}\,d\nu(J)
&\ge \nu\left( \widehat I_{(\widetilde\Delta-2\varepsilon)/4} 
(X_{n-f(n)}) \right)   \inf_{I \in\widehat I_{(\widetilde\Delta-2\varepsilon)/4} 
(x) } \exp\left\{\frac{\widehat Q_x (I)}{w}\right\}
\\ 
&\ge 
\nu\left( \widehat I_{(\widetilde\Delta-2\varepsilon)/4} 
(X_{n-f(n)}) \right)\exp\left\{\frac{\widehat Q^\star  (X_{n-f(n)}) - (\widetilde\Delta-2\varepsilon)/4 }{w}\right\}.
\end{align*}
}

Thus,
\[
p_{X_{n-f(n)}}^w 
\left(
\widehat I_{(\widetilde\Delta-2\varepsilon)/2} 
(X_{n-f(n)})^c
\right)
\le
\frac{\nu( \widehat I_{(\widetilde\Delta-2\varepsilon)/2} 
(X_{n-f(n)})^c ) }{\nu( \widehat I_{(\widetilde\Delta-2\varepsilon)/4} 
(X_{n-f(n)}) ) } \exp \left\{- \frac{\tilde{\Delta} -2\varepsilon}{4 w} \right\}.
\]

\textbf{Last, we summarize and remove the conditioning} to obtain
\begin{align*}
\mathbb P\left(
I_n\notin I_{\widetilde\Delta}(X_{n-1})
\right) = & \mathbb E\!\left[
\mathbb P\left(
I_n\notin I_{\widetilde\Delta}(X_{n-1})
\mid \mathcal{F}_{n-f(n)}
\right)
\right]
\\
\le & \bar\eta^{f(n)}
+\frac{2^pcL_PC_\alpha}{1-\bar\eta}
f(n) n^{-p}
+  \sup_{x \in \mathcal X} \frac{\nu( \widehat I_{(\widetilde\Delta-2\varepsilon)/2} 
(x )^c )}{\nu( \widehat I_{(\widetilde\Delta-2\varepsilon)/4} 
(x ) )} \exp \left\{- \frac{\tilde{\Delta} -2\varepsilon}{4 w} \right\}. \qquad \qquad \Box
\end{align*}

\end{proof}

\subsection{Proof of Example~\ref{ex:concentration}}
\label{app:proof-concentration}
\begin{proof}{Proof:} Recall from Lemma~\ref{lem:recourse_lipschitz_smoothing} that
$Q(x,\cdot)$ is $L'_\xi$-Lipschitz on $\Xi$, uniformly over
$x\in\mathcal X$. Since $\widehat Q_x(I)=Q(x,\xi_I)$ and
$\xi_I=I$ in Example~\ref{ex:concentration}, it follows that
$\widehat Q_x(\cdot)$ is also $L'_\xi$-Lipschitz on $\mathcal I=\Xi$.

We first establish a uniform lower bound on the measure of metric balls.
Equip $\Xi$ with the Euclidean metric and recall that $\nu$ is the
normalized Lebesgue measure on $\Xi$, i.e.,
$
\nu(A):=
\frac{\operatorname{vol}_{d_\xi}(A\cap\Xi)}
{\operatorname{vol}_{d_\xi}(\Xi)}.
$
For any $I\in\Xi$ and $0<r\le D_{\mathcal I}$, define the scaled set
$S:=I+(r/D_{\mathcal I})(\Xi-I)$. By convexity of $\Xi$,
$S\subseteq\Xi$. Moreover, for any $s\in S$, there exists
$J\in\Xi$ such that
$s=I+(r/D_{\mathcal I})(J-I)$, and hence
$
\|s-I\|
=
\frac{r}{D_{\mathcal I}}\|J-I\|
\le
\frac{r}{D_{\mathcal I}}\operatorname{diam}(\Xi)
=r.
$
Thus, $S\subseteq B_{\mathcal I}(I,r)\cap\Xi$, where
$B_{\mathcal I}(I,r):=\{J\in\mathcal I:\|J-I\|\le r\}$.
By the transformation properties of the Lebesgue measure under scaling
and translation~\citep[see][Theorem 2.20, 2.23]{rudin1974real},
$\operatorname{vol}_{d_\xi}(S)
=(r/D_{\mathcal I})^{d_\xi}\operatorname{vol}_{d_\xi}(\Xi)$.
Consequently,
\[
\nu\bigl(B_{\mathcal I}(I,r)\bigr)
\ge
\frac{\operatorname{vol}_{d_\xi}(S)}
{\operatorname{vol}_{d_\xi}(\Xi)}
=
\left(\frac{r}{D_{\mathcal I}}\right)^{d_\xi}.
\]
For any $x\in\mathcal X$, let
$I_x^\star\in\arg\max_{I\in\mathcal I}\widehat Q_x(I)$ and set
$r_{\widetilde\Delta}:=
\min\{D_{\mathcal I},\widetilde\Delta/(4L'_\xi)\}$.
Since $L'_\xi r_{\widetilde\Delta}\le\widetilde\Delta/4$, the
$L'_\xi$-Lipschitz continuity of $\widehat Q_x(\cdot)$ implies
$B_{\mathcal I}(I_x^\star,r_{\widetilde\Delta})
\subseteq\widehat I_{\widetilde\Delta/4}(x)$.
Therefore, by the preceding small-ball bound and $\nu(\mathcal I)=1$,
\[
\nu\bigl(\widehat I_{\widetilde\Delta/4}(x)\bigr)
\ge
\left(\frac{r_{\widetilde\Delta}}{D_{\mathcal I}}\right)^{d_\xi}
=
\min\left\{
1,\frac{\widetilde\Delta}{4L'_\xi D_{\mathcal I}}
\right\}^{d_\xi}, \Longrightarrow
\frac{\nu\bigl(\widehat I_{\widetilde\Delta/2}(x)^c\bigr)}
{\nu\bigl(\widehat I_{\widetilde\Delta/4}(x)\bigr)}
\le
\max\left\{
1,\frac{4L'_\xi D_{\mathcal I}}{\widetilde\Delta}
\right\}^{d_\xi}.
\]
Taking the supremum over $x\in\mathcal X$ completes the proof.
\end{proof}

\subsection{Proof of Theorem~\ref{thm:finite-high-prob}}

\begin{proof}{Proof:}
We denote $I_t^\star:=I^\star(X_t)$ for \(t \in \mathbb{N}_+\). {We first bound $\mathbb P(I_{t+1}\notin I_t^\star\mid\mathcal F_t)$} by discussing the following two cases. 

\noindent \textbf{Case 1}. If $I_t \in I_t^\star $, then for any $j \notin I_t^\star$, 
\[
\widehat Q_{X_t}(I_t)-\widehat Q_{X_t}(j) \ge Q(X_t,I_t)-Q(X_t,j)-2\varepsilon \geq \tilde{\Delta}_t - 2 \varepsilon \ge\Delta_\star-2\varepsilon.
\]

The invariance of $p_{X_t}^w$ gives 
\(
p_{X_t}^w((I_t^\star)^c) = \sum_{i\in\mathcal I}
p_{X_t}^w(i)P_{X_t}(i,(I_t^\star)^c) \ge
p_{X_t}^w(I_t)P_{X_t}(I_t,(I_t^\star)^c).
\)
Then, 
\begin{align*}
\mathbb P(I_{t+1}\notin I_t^\star\mid\mathcal F_t)  &\le \frac{p_{X_t}^w((I_t^\star)^c)}{p_{X_t}^w(I_t)} \ = \ \sum_{j\notin I_t^\star}
\exp\left\{
\frac{\widehat Q_{X_t}(j)-\widehat Q_{X_t}(I_t)}{w}
\right\} \ \le \ 
(|\mathcal I|-1)
\exp\left\{-\frac{\Delta_\star-2\varepsilon}{w}\right\}.
\end{align*}

\noindent \textbf{Case 2}. If $I_t \notin I_t^\star $, then we choose any $j\in I_t^\star$ and it follows from the definition of total variation that
\begin{align*}
\mathbb P(I_{t+1}\notin I_t^\star\mid\mathcal F_t) = P_{X_t}(I_t,(I_t^\star)^c)  &\le P_{X_t}(j,(I_t^\star)^c)
+\|P_{X_t}(I_t,\cdot)-P_{X_t}(j,\cdot)\|_{\mathrm{TV}}
\\ & \le (|\mathcal I|-1)
\exp\left\{-\frac{\Delta_\star-2\varepsilon}{w}\right\} + \bar\eta,
\end{align*}
where the second inequality follows from Case 1 and Assumption~\ref{ass:uniform-mixing}. Combining these two cases gives
\[
\mathbb P(I_{t+1}\notin I_t^\star\mid\mathcal F_t) \le (|\mathcal I|-1)
\exp\left\{-\frac{\Delta_\star-2\varepsilon}{w}\right\} + \bar\eta \mathbf 1_{\{I_t\notin I_t^\star\}}.
\]

We then take expectations and use the tower property to obtain
\(
\mathbb P(I_{t+1}\notin I_t^\star ) \le 
(|\mathcal I|-1)
\exp\left\{-{(\Delta_\star-2\varepsilon)}/{w}\right\} + \bar\eta \mathbb P {\{I_t\notin I_t^\star\}}.
\)
If $I_t\notin I_t^\star$, then either $I_{t} \notin I_{t-1}^\star $ or $I_t^\star \neq I_{t-1}^\star$ (if neither event happens, then $I_{t} \in  I_{t-1}^\star = I_t^\star$, contradicting the presumption). It follows that 
\[
\mathbb P {\{I_t\notin I_t^\star\}} \le  \mathbb P {\{I_t\notin I_{t-1}^\star\}} + \mathbb P {\{I_t^\star \neq I_{t-1}^\star\}}.
\]

{Second, we bound the second term in right-hand side above}. If two sets $I_t^\star$ and $I_{t-1}^\star$ differ, 
then WLOG there is a state $a \in I_{t-1}^\star \backslash I_t^\star$ (if \(I_{t-1}^\star \backslash I_t^\star = \varnothing\), then pick an $a \in I_t^\star \backslash I_{t-1}^\star$ and a similar proof as below works). For any state $b \in I_t^\star$, we have
\[
\widehat Q_{X_t}(b)-\widehat Q_{X_t}(a)
\ge\Delta_\star-2\varepsilon, \ 
\widehat Q_{X_{t-1}}(b)-\widehat Q_{X_{t-1}}(a)
\le2\varepsilon.
\]
It follows that
\begin{align*}
\Delta_\star-4\varepsilon
&\le
\bigl[
  \widehat Q_{X_t}(b)-\widehat Q_{X_t}(a)
\bigr]
-
\bigl[
  \widehat Q_{X_{t-1}}(b)-\widehat Q_{X_{t-1}}(a)
\bigr]
\\
&=
\bigl[
 \widehat Q_{X_t}(b)- \widehat Q_{X_{t-1}}(b)
\bigr]
+
\bigl[
  \widehat Q_{X_{t-1}}(a)-\widehat Q_{X_t}(a)
\bigr] \\
&\le
\left|
  \widehat Q_{X_{t-1}}(b)-\widehat Q_{X_t}(b)
\right|
+
\left|
  \widehat Q_{X_t}(a)-\widehat Q_{X_{t-1}}(a)
\right|.
\end{align*}
Thus,
\begin{align*}
\mathbb P {\{I_t^\star \neq I_{t-1}^\star\}} = \mathbb E \big[\mathbf{1}{\{I_t^\star \neq I_{t-1}^\star\}}\big] 
& \le \mathbb E \left[\frac{\left|
  \widehat Q_{X_{t-1}}(b)-\widehat Q_{X_t}(b)
\right|
+
\left|
  \widehat Q_{X_t}(a)-\widehat Q_{X_{t-1}}(a)
\right|}{\Delta_\star-4\varepsilon}\right]
\\ &\le 
\frac{2L_{Q'}\|X_t-X_{t-1}\|}{\Delta_\star-4\varepsilon} \le \frac{2 c L_{Q}' }{\Delta_\star-4\varepsilon} \alpha_t.
\end{align*}
We combine the results and obtain
\[
\mathbb P(I_{t+1}\notin I_t^\star ) \le 
(|\mathcal I|-1)
\exp\left\{-\frac{\Delta_\star-2\varepsilon}{w}\right\} + \bar\eta 
\left(\mathbb P {\{I_t\notin I_{t-1}^\star\}} + \frac{2 c  L_{Q}' }{\Delta_\star-4\varepsilon} \alpha_t  \right).
\]

Telescoping this inequality for
$t=1$ through $t=n-1$ gives a geometric coefficient
$\bar\eta^{\,n-1-t}$ for the movement term generated at time $t$, and a
geometric sum $\sum_{j=0}^{n-2}\bar\eta^j
=(1-\bar\eta^{\,n-1})/(1-\bar\eta)$. In other words, for any $n \ge 2$, we obtain
\[
\mathbb P\left(I_n\notin I^\star(X_{n-1})\right)
\le
\bar\eta^{\,n-1}
\mathbb P\left(I_1\notin I^\star(X_0)\right)
+\frac{2c\bar\eta L_{ Q}'}
{\Delta_\star-4\varepsilon}
\sum_{t=1}^{n-1}
\bar\eta^{\,n-1-t}\alpha_t
+
\frac{|\mathcal I|-1}{1-\bar\eta}
(1-\bar\eta^{\,n-1})
\exp\left\{-\frac{\Delta_\star-2\varepsilon}{w}\right\},
\]
which proves the first claim of Theorem~\ref{thm:finite-high-prob}.

Finally, recall that the preceding argument on the event $\{I_t^\star\ne I_{t-1}^\star\}$ implies
\[
\Delta_\star-4\varepsilon
\le2L_Q'\|X_t-X_{t-1}\| 
\le 2cC_\alpha L_Q't^{-p}. 
\]
But, for all $t\ge n_0 :=\max\left\{2,\,
1+\left\lceil
\left(\frac{2cC_\alpha L_Q'}{\Delta_\star-4\varepsilon}\right)^{1/p}
\right\rceil\right\}$, the last quantity is strictly smaller than $\Delta_\star-4\varepsilon$. Therefore, 
$\mathbb P(I_t^\star\ne I_{t-1}^\star)=0$ for all $t\ge n_0$. The prior recursion therefore simplifies to
\[
\mathbb P(I_{t+1}\notin I_t^\star)
\le\bar\eta\,\mathbb P(I_t\notin I_{t-1}^\star) 
+
(|\mathcal I|-1)
\exp\left\{-\frac{\Delta_\star-2\varepsilon}{w}\right\}.
\]
Telescoping this inequality, starting from \(t = n_0\), produces the second claim of Theorem~\ref{thm:finite-high-prob}. \hfill \(\Box\)

\end{proof}

\subsection{Proof of Theorem~\ref{thm:continuous-state-convergence}} \label{apx-thm:continuous-state-convergence}

\begin{proof}{Proof:} 
We shall keep using the framework for proving Theorem~\ref{thm:continuous_convergence}. It suffices to justify its extension from finite matrices to operators
on bounded measurable functions and verify the corresponding constants. The major descent and averaging arguments remain similar. We note that the score is exact, i.e.,
$\widehat Q_x(I)\equiv Q(x,\xi_I)$ (and hence $\epsilon=0$),
and let $\nu$ be a probability measure.

\noindent\textbf{(i) Properties inherited from Lemmas~\ref{lem:recourse_lipschitz_smoothing} and~\ref{prop:L-smooth}.} 

First, the proof of Lemma~\ref{lem:recourse_lipschitz_smoothing} applies uniformly to all $I\in\I$ as the dual polyhedron is independent of the state
space. It follows that
\(
-E^\top\pi^\star(x,I)\in\partial_xQ(x,\xi_I),
\qquad \|E^\top\pi^\star(x,I)\|\le L'_Q.
\)
Since $\nu$ is a finite nonzero measure and $Q(x,\cdot)$ is measurable and bounded for every $x\in\mathcal X$, we have $0< \int_{\mathcal I}\exp\{Q(x,\xi_I)/w\} \nu(dI) < \infty$. Hence, $F_w^\nu$ and $p_x^w$ are well-defined. 

Second, the proof of Lemma~\ref{prop:L-smooth} extends by replacing finite sums with integrals (note that the selected subgradients are
measurable and uniformly bounded, so their integrals are well-defined). In particular, we have 
\(z_w(x) :=\int_\I Z_x(I)p_x^w(dI)\in\partial F_w^\nu(x)\), \(\|Z_x(I)\|, \|z_w(x)\|\le L_f+L'_Q\), and \(|F_w^\nu(x)-F_w^\nu(y)| \le(L_f+L'_Q)\|x-y\|\) for all \(x, y \in \X\).

\noindent \textbf{(ii) The continuous-state Poisson operator and Proposition~\ref{prop:poisson-noV}.}

For a bounded measurable function $\varphi$, define $ (P_x\varphi)(I):=\int\varphi(J)P_x(I,dJ),$ and $((\mathbf1p_x^w)\varphi)(I):=\int\varphi(J)p_x^w(dJ)$. Note that these are bounded operators under the supremum norm and the MH kernel has invariant measure $p_x^w$ by the detailed balance equations. 
For any
probability measures $\mu$ and $\lambda$ on $\I$ and for all $t\ge1$, Assumption~\ref{ass:uniform-mixing} implies that
\[
\|\mu P_x^t-\lambda P_x^t\|_{\mathrm{TV}}
\le \bar\eta^t\|\mu-\lambda\|_{\mathrm{TV}}.
\]
Since $p_x^w$ is invariant for $P_x$, taking $\mu=\delta_I$ and $\lambda=p_x^w$ yields 
\[
\|P_x^t(I,\cdot)-p_x^w\|_{\mathrm{TV}} \ = \ \|\delta_I P_x^t-p_x^wP_x^t\|_{\mathrm{TV}} \ \le \ \bar\eta^t\|\delta_I-p_x^w\|_{\mathrm{TV}} \ \le \ \bar\eta^t,
\]
where the first equality uses the invariance of \(p^w_x\) and the last inequality uses the fact that 
$\|\mu-\lambda\|_{\mathrm{TV}}
:=\sup_B|\mu(B)-\lambda(B)|\le1$
for probability measures. Recalling the dual representation $\sup_{\|\varphi\|_\infty\le1}
\left|\int_\I\varphi(J)(\mu-\lambda)(dJ)\right|
=2\|\mu-\lambda\|_{\mathrm{TV}}$ of the total variation, we also have 
\[
\|P_x^t-\mathbf1p_x^w\|_{\infty\to\infty} =\sup_{\|\varphi\|_\infty\le1}\sup_{I\in\I}
\left|
\int_\I\varphi(J)
\bigl(P_x^t(I,dJ)-p_x^w(dJ)\bigr)
\right|  =2\sup_{I\in\I}
\|P_x^t(I,\cdot)-p_x^w\|_{\mathrm{TV}}
\le2\bar\eta^t.
\]
But the operator \((\mathrm{Id}-P_x+\mathbf1p_x^w)^{-1}\) admits the equality $(\mathrm{Id}-P_x+\mathbf1p_x^w)^{-1} = \mathrm{Id}+\sum_{t=1}^{\infty}(P_x^t-\mathbf1p_x^w)$. To see this, multiply the partial sum \(\left[\mathrm{Id}+\sum_{t=1}^{T}(P_x^t-\mathbf1p_x^w)\right]\) by $(\mathrm{Id}-P_x+\mathbf1p_x^w)$ to obtain
$\mathrm{Id}-(P_x^{T+1}-\mathbf1p_x^w)$, which converges to $\mathrm{Id}$ as \(T\) increases. Consequently,
\[
\sup_x\|(\mathrm{Id}-P_x+\mathbf1p_x^w)^{-1}\|_{\infty\to\infty}
\ \le \ 1+2\sum_{t=1}^\infty\bar\eta^t
\ = \ \frac{1+\bar\eta}{1-\bar\eta} \ := \ C_*.
\]
Note here that $C_*$ carries the same meaning as in Theorem~\ref{thm:continuous_convergence}, with the finite-dimensional inverse replaced by its bounded-operator counterpart.

Fix $x^\star\in\argmin_{x \in \X} F_w^\nu(x)$ and adapt the centered error in Proposition~\ref{prop:poisson-noV} through
\[
e_x(I):=Q(x,\xi_I)-Q(x^\star,\xi_I)
 -\int_\I\big[Q(x,\xi_J)-Q(x^\star,\xi_J)\big] \, p_x^w(dJ), \quad e_k :=e_{x_{k-1}}(I_k).
\]
Then, $\int e_x\,dp_x^w=0$ and $\|e_x\|_\infty\le2DL'_Q$.  For each $I\in\I$, define 
\[
\widehat Z_x(I):=
 \bigl[(\mathrm{Id}-P_x+\mathbf1p_x^w)^{-1}e_x\bigr](I)
 =\sum_{t=0}^\infty(P_x^te_x)(I),
\]
which satisfies $ \widehat Z_x(I)-(P_x\widehat Z_x)(I)=e_x(I)$ and $\int\widehat Z_x\,dp_x^w=0.$ 
We can then similarly bound it as 
\[
\max\left\{\sup_I|\widehat Z_x(I)|, \sup_I|(P_x\widehat Z_x)(I)| \right\} \ \le \ B_Z \ := \ 2C_*DL'_Q.
\]
Lemma~\ref{lem:target-distribution-lipschitz}, which we proved for a general state space, gives that
\[
\|\mathbf1p_x^w-\mathbf1p_y^w\|_{\infty\to\infty} \le L_p\|x-y\| \ \text{with} \ L_p := \frac{2L'_Q}{w}
\]
and Assumption~\ref{ass:general-kernel-lipschitz} gives
\(
\|P_x-P_y\|_{\infty\to\infty}\le2L_P\|x-y\|
\)
for all \(x, y \in \X\) (note that the factor of $2$ follows from the total-variation convention in Assumption~\ref{ass:general-kernel-lipschitz}). Furthermore, exactly as for the centered error in Lemma~\ref{prop:2}, we have 
\(
\|e_x-e_y\|_\infty
 \le(2L_Q'+DL_Q'L_p)\|x-y\|.
\)
The resolvent identity used in the proof of Lemma~\ref{prop:2} remains valid for
bounded invertible operators. Applying that same argument yields
\[
\sup_I|\widehat Z_x(I)-\widehat Z_y(I)| \le\bigl[C_*(2L_Q'+DL_Q'L_p)
 +2C_*^2DL_Q'(2L_P+L_p)\bigr]\|x-y\| := L_Z\|x-y\|.
\]
Adding and subtracting $P_x\widehat Z_y$ then gives
\[
\sup_I|(P_x\widehat Z_x)(I)-(P_y\widehat Z_y)(I)|
 \le\bigl(2L_PB_Z+L_Z\bigr)\|x-y\|.
\]
Therefore, the bounds required by Proposition~\ref{prop:poisson-noV} still hold, with the only change being its kernel operator modulus $L_P$ replaced by $2L_P$.

\noindent\textbf{(iii) Application of Proposition~\ref{prop:poisson-noV} and Theorem~\ref{thm:continuous_convergence}.} 

The remainder of the proof uses exactly the same arguments as in the proofs of Proposition~\ref {prop:poisson-noV} and Theorem~\ref{thm:continuous_convergence}, with $F_w$ replaced by $F_w^\nu$ and the constants defined above. They therefore yield the stated bounds for both the step-size average and the linear average. Substituting $\alpha_k=\alpha_0/\sqrt{k}$ and $L_P=2L_Q'/w$ gives the claimed rates. \hfill \(\Box\)

\end{proof}

\subsection{Proof of Corollary~\ref{cor:adaptive-transition-identification}} \label{apx-cor:adaptive-transition-identification}
\begin{proof}{Proof:}
For the sequence associated with the fixed integer, let $X_k$
be the incumbent solution after $k$ previous oracle calls and let $I_k$
be the state stored immediately before this update.  Thus the update at
$X_k$ returns $I_{k+1}$.  The initial call, indexed by $k=0$, returns
$I_1$ and initializes the recursion.  For every $k\ge1$, if
$\|X_k-X_{k-1}\|\le c_0k^{-p}$, the argument in the proof of
Theorem~\ref{thm:finite-high-prob} gives
\[
\mathbb P\!\left(I_{k+1}\notin I^\star(X_k)\right)
\le
\bar\eta\,
\mathbb P\!\left(I_k\notin I^\star(X_{k-1})\right)
+\frac{2c_0\bar\eta L_{Q'}}
{\Delta_\star-4\varepsilon}k^{-p}
+(|\mathcal I|-1)
\exp\!\left\{-\frac{\Delta_\star-2\varepsilon}{w}\right\}.
\]
Suppose instead that the movement condition fails.  The incumbent remains
fixed during the ensuing $g(\kappa)$ transitions.  The Dobrushin contraction
and invariance of $p_{X_k}^w$ imply, uniformly over the initial state,
\[
    \left\|
        \delta_{I_k}P_{X_k}^{g(\kappa)}-p_{X_k}^w
    \right\|_{\mathrm{TV}}
    \le \bar\eta^{g(\kappa)}.
\]
The uniform score gap and the approximation error remain
\[
    p_{X_k}^w\!\left(I^\star(X_k)^c\right)
    \le
    (|\mathcal I|-1)
    \exp\!\left\{-\frac{\Delta_\star-2\varepsilon}{w}\right\}.
\]
At the transition indexed by $\kappa\ge1$, exactly $\kappa$ previous oracle calls
have been completed, so $\max\{\kappa(x_z),1\}=\kappa$.  We choose $g(\kappa)$ such that
\[
    \bar\eta^{g(\kappa)}
    \le
    \frac{2c_0\bar\eta L_{Q'}}
    {\Delta_\star-4\varepsilon}\kappa^{-p}.
\]
For our case, we therefore choose 
\[
g\!\left(\kappa(x_z)\right)
:=
\max\left\{
    1,
    \left\lceil
    \frac{
        p\log\!\left(\max\{\kappa(x_z),1\}\right)
        -
        \log\!\left(
            \dfrac{2c_0\bar\eta L_{Q'}}
            {\Delta_\star-4\varepsilon}
        \right)
    }{
        |\log\bar\eta|
    }
    \right\rceil
\right\}.
\]
Consequently, the one-step recursive argument in the proof of Theorem~\ref{thm:finite-high-prob} applies. After $\kappa$ calls,
$\state(\bar x_z)=I_\kappa$ and
$\refpoint(\bar x_z)$ is the continuous component of $X_{\kappa-1}$.
Iterating from $k=1$ through
$k=\kappa-1$ gives us the same result as in Theorem~\ref{thm:finite-high-prob}.

It remains to obtain a convenient closed-form bound. For $\kappa\ge2$, splitting its index set at
$\lfloor(\kappa-1)/2\rfloor$ gives
\[
\sum_{t=1}^{\kappa-1}\bar\eta^{\,\kappa-1-t}t^{-p}
\le
\frac{\bar\eta^{\lceil(\kappa-1)/2\rceil}}{1-\bar\eta}
+\frac{2^p}{1-\bar\eta}\kappa^{-p}.
\]
Moreover,
\[
\bar\eta^{\lceil(\kappa-1)/2\rceil}
\le
\bar\eta^{-1/2}
\exp\!\left\{-\frac{|\log\bar\eta|}{2}\kappa\right\}
\le
\bar\eta^{-1/2}
\left(\frac{2p}{e|\log\bar\eta|}\right)^p \kappa^{-p},
\]
where the last inequality follows by maximizing
$u^p\exp\{-|\log\bar\eta|u/2\}$ over $u>0$. 
Substitution yields
\(
   \sum_{t=1}^{\kappa-1}
    \bar\eta^{\,\kappa-1-t}t^{-p}
    \le C_{\bar\eta,p}\kappa^{-p}
\). \hfill \(\Box\)
\end{proof}

\end{document}